\documentclass[11pt]{amsart}

\usepackage[utf8]{inputenc}
\usepackage[T1]{fontenc}

\usepackage[mathscr]{eucal}

\usepackage[a4paper, twoside=false, vmargin={2cm,3cm}, includehead]{geometry}

\newtheorem{lemma}{Lemma}[section]
\newtheorem{theorem}[lemma]{Theorem}
\newtheorem*{theorem*}{Theorem}
\newtheorem{corollary}[lemma]{Corollary}

\newtheorem{proposition}[lemma]{Proposition}
\newtheorem*{proposition*}{Proposition}
\newtheorem*{MCP}{Multiplicative correspondence principle}

\newtheorem{problem}{Problem}

\newtheorem{fact}{Fact}

\theoremstyle{definition}
\newtheorem*{convention}{Convention}
\newtheorem*{notation}{Notation}
\newtheorem*{definition}{Definition}
\newtheorem*{example}{Example}
\newtheorem*{remark}{Remark}

\newcommand{\vide}{\rule{0mm}{5mm}}
\newcommand{\doublevide}{\rule{0mm}{7mm}}

\newcommand{\C}{{\mathbb C}}
\newcommand{\E}{{\mathbb E}}
\newcommand{\N}{{\mathbb N}}

\newcommand{\Q}{{\mathbb Q}}
\newcommand{\R}{{\mathbb R}}
\newcommand{\T}{{\mathbb T}}
\newcommand{\Z}{{\mathbb Z}}

\newcommand{\CA}{{\mathcal A}}
\newcommand{\CB}{{\mathcal B}}
\newcommand{\CC}{{\mathcal C}}
\newcommand{\CF}{{\mathcal F}}
\newcommand{\CH}{{\mathcal H}}
\newcommand{\CM}{{\mathcal M}}

\newcommand{\CP}{{\mathcal P}}

\newcommand{\CX}{{\mathcal X}}

\newcommand{\tN}{{\widetilde N}}

\newcommand{\ba}{{\mathbf{a}}}
\newcommand{\bc}{{\mathbf{c}}}
\newcommand{\bg}{{\mathbf{g}}}

\newcommand{\bk}{{\mathbf{k}}}
\newcommand{\bt}{{\mathbf{t}}}
\newcommand{\bu}{{\mathbf{u}}}
\newcommand{\bx}{{\mathbf{x}}}

\newcommand{\bell}{{\boldsymbol{\ell}}}
\newcommand{\balpha}{{\boldsymbol{\alpha}}}

\newcommand{\bzero}{{\boldsymbol{0}}}

\newcommand{\ov}{\overline}
\newcommand{\ve}{\varepsilon}
\newcommand{\wt}{\widetilde}
\newcommand{\e}{\mathrm{e}}
\newcommand{\ZN}{\Z_{\widetilde N}}
\newcommand{\norm}[1]{\lVert #1 \rVert}
\newcommand{\lip}{{\text{\rm Lip}}}
\newcommand{\inv}{^{-1}}
\DeclareMathOperator{\spec}{Spec}
\DeclareMathOperator{\id}{id}
\newcommand{\dmult}{d_\textrm{mult}}
\newcommand{\one}{\mathbf{1}}

\begin{document}

\title[Uniformity of
multiplicative functions and partition regularity]{Uniformity of
multiplicative functions and partition regularity of some quadratic
equations}

\author{Nikos Frantzikinakis}
\address[Nikos Frantzikinakis]{University of Crete, Department of mathematics, Voutes University Campus, Heraklion 71003, Greece} \email{frantzikinakis@gmail.com}
\author{Bernard Host}
\address[Bernard Host]{
Universit\'e Paris-Est Marne-la-Vall\'ee, Laboratoire d'analyse et
de math\'ematiques appliqu\'ees, UMR CNRS 8050, 5 Bd Descartes,
77454 Marne la Vall\'ee Cedex, France }
\email{bernard.host@univ-mlv.fr}

\begin{abstract}
Since the theorems of Schur and van der Waerden, numerous partition
regularity results have been  proved for linear equations, but
progress has been scarce for non-linear ones, the hardest case being
equations in three variables.  We prove partition regularity for certain equations
involving
 quadratic forms in three variables, showing for example that
  the equations $16x^2+9y^2=n^2$ and $x^2+y^2-xy=n^2$ are partition regular, where
  $n$ is allowed to vary freely in $\N$.
For each such problem we establish a  density analogue that can be
formulated in ergodic terms as a recurrence property for  actions by
dilations on a probability space. Our key tool for establishing such
recurrence properties is a decomposition result for multiplicative
functions which is of independent interest. Roughly speaking, it
states that the arbitrary multiplicative function of modulus $1$ can be decomposed
into two terms, one that is approximately periodic and another that
has
small Gowers uniformity norm of degree three.
\end{abstract}

\thanks{The  first author was partially supported by
 Marie Curie IRG  248008.}




\maketitle

\section{Introduction and main results}
\subsection{Partition regularity results for quadratic forms}
An important question in Ramsey theory is to determine which
algebraic equations, or systems of equations, are partition regular
over the natural numbers. In this article,  we restrict our attention to
polynomials in three variables, in which case  partition regularity of
 the equation  $p(x,y,z)=0$  amounts to  saying that,
 for any partition of  $\N$ into finitely many
cells, some cell contains \emph{distinct} $x,y,z$ that satisfy the
equation.

 The case where the polynomial $p$ is linear
was completely solved by  Rado~\cite{R33}: For $a,b,c\in
\N$  the equation $ax+by=cz$ is partition regular if and only if
either $a$, $b$, or $a+b$ is equal to $c$.
The situation is much less clear   for second or higher degree
equations and only scattered results are known.
A notorious old question of Erd\"{o}s and Graham~\cite{EG80} is
whether the equation $x^2+y^2=z^2$ is partition regular. As Graham
remarks in \cite{G08} ``There is actually very little data (in
either direction) to know which way to guess''. More generally, one
may ask for which $a,b,c\in \N$ is the equation
\begin{equation}\label{E:StrongPartitionRegular}
ax^2+by^2=cz^2
\end{equation}
 partition regular. A necessary condition is that at least  one of
$a$, $b$, and $a+b$ equals $c$, but currently there are no $a,b,c\in
\N$ for which partition regularity of 
\eqref{E:StrongPartitionRegular} is known.

In this article, we  study the  partition regularity of equation
\eqref{E:StrongPartitionRegular}, and other quadratic equations,
under the relaxed condition that the  variable $z$ is allowed to
vary freely in $\N$.
\begin{definition}
The equation $p(x,y,n)=0$ is \emph{partition
regular in} $\N$ if for any partition of $\N$
into finitely many cells, for some $n\in \N$,
one of the cells  contains \emph{distinct} $x,y$ that satisfy the equation.
\end{definition}
A classical result of Furstenberg-S\'ark\"ozy  \cite{Fu77, Sa78}  is
that the equation $x-y=n^2$ is partition regular.
 Other examples of translation invariant equations are provided by the polynomial van der Waerden Theorem
 of Bergelson and Leibman~\cite{BL96},
 but not much is known in the non-translation invariant case.
A result of Khalfalah and Szemer\'edi \cite{KS06}  is that the
equation $x+y=n^2$ is partition regular.
Again, the
situation is much less clear when one considers non-linear
polynomials in $x$ and $y$, as is the case for the equation
 $a x^2+by^2=n^2$ where $a,b\in \N$.
It is one of the main goals of this article to produce the first
positive results in this direction. For example, we   show that the
equations
$$
16x^2+9y^2=n^2    \quad \text{ and } \quad x^2+y^2- xy=n^2
$$
are partition regular (note that $16x^2+9y^2=z^2$ is not partition
regular).
In fact we prove a  more general result for  homogeneous quadratic
forms in three variables.
\begin{theorem}[The three squares theorem]\label{th:partition-regular1}
Let $p$ be the quadratic form
\begin{equation}
\label{E:Q2} p(x,y,z)=ax^2+by^2+cz^2+dxy+exz+fyz
\end{equation}
where  $a, b, c$ are non-zero  and $d,e,f$ are arbitrary integers.
Suppose that all three forms $p(x,0,z)$, $p(0,y,z)$, $p(x,x,z)$ have
non-zero square discriminants. Then the equation $ p(x,y,n)=0 $ is
partition regular.
\end{theorem}
The last hypothesis means that
the three integers
$$
\Delta_1:=e^2-4ac,  \quad \Delta_2:=f^2-4bc,  \quad \Delta_3:=(e+f)^2-4c(a+b+d)
$$
are non-zero squares.
As a special case, 
 we get the following result:
\begin{corollary}\label{Corol1}
Let  $a,b,$ and $a+b$ be non-zero squares. Then the equation $
ax^2+by^2=n^2 $ is partition regular. More generally, if   $a,b,$
and $a+b+c$ are  non-zero squares, then the equation $
ax^2+by^2+cxy=n^2 $ is partition regular.
\end{corollary}

A partition $\mathscr C_1,\ldots, \mathscr C_r$ of $\N$ induces
another partition $\wt{\mathscr C}_1,\ldots, \wt{\mathscr C}_r$ by
the following rule: $x\in \wt{\mathscr C}_i$ if and only if $x^2\in
\mathscr C_i$.  Applying Theorem~\ref{Corol1} for the  induced
partition we deduce non-trivial results even for linear equations:
\begin{corollary}
 Let  $a, b,$ and $a+b$ be non-zero squares.
Then the equation
$a x+by=n^2$ is partition regular.
\end{corollary}
Although combinatorial  tools,  Fourier analysis tools, and the circle method have been used
successfully to prove partition regularity of equations that enjoy
some linearity features (also for non-linear equations with at least
four variables), we have not found such tools adequate for the fully
non-linear setup we are interested in.
    Instead, we  found greater utility to the  recently developed toolbox of  higher order Fourier analysis
     that relies on inverse theorems for the Gowers uniformity norms and various quantitative equidistribution results
     on nilmanifolds.
    We give a summary of our  proof strategy   in the next subsections.

\subsection{Parametric reformulation}
In order to prove Theorem~\ref{th:partition-regular1} we exploit some
special features of  the solution sets of the equations involved given in parametric form.
In particular,  we have the following result that is proved in
Appendix~\ref{SS:AppNumberTheory}:
\begin{proposition} \label{prop:linearfactors}
Let the quadratic form $p$ satisfy the hypothesis of
Theorem~\ref{th:partition-regular1}. Then there exist
$\ell_0,\ell_1$ positive and $\ell_2,\ell_3$ non-negative integers
with $\ell_2\neq \ell_3$, such that for every $k,m,n\in \N$, the
integers $x=k\ell_0 m(m+\ell_1n)$ and $y=
k\ell_0(m+\ell_2n)(m+\ell_3n)$ satisfy the equation $p(x,y,z)=0$ for
some $z\in \N$.
\end{proposition}
For example,  the equation  $16x^2+9y^2=z^2$ is satisfied by the
integers $x=km(m+3n)$, $y=k(m+n)(m-3n)$, $z=k(5m^2+9n^2+6mn)$
(replacing $m$ with  $m+3n$ leads to coefficients  of the announced
form),
 and  the equation  $x^2+y^2-
xy=z^2$ is satisfied by the integers $x=km(m+ 2n)$, $y=
k(m-n)(m+n)$, $z=k( m^2+n^2+mn)$.

The  key properties of the patterns involved in Proposition~\ref{prop:linearfactors} are: $(a)$
they are  dilation invariant, which follows from homogeneity, $(b)$
they ``factor linearly''
which follows from our assumption that the discriminants $\Delta_1,\Delta_2$ are
squares,
 and $(c)$ the coefficient of $m$ in all forms can be taken to be $1$ which follows from our assumption that the discriminant $\Delta_3$ is a square.

Using Proposition~\ref{prop:linearfactors}, we see that
Theorem~\ref{th:partition-regular1}  is a consequence of the
following result:
\begin{theorem}[Parametric reformulation]
\label{th:partition-regular2}
Let $\ell_0,\ell_1$ be positive and $\ell_2,\ell_3$ non-negative integers
with $\ell_2\neq \ell_3$.
Then for every partition of $\N$ into
finitely many cells, there exist $k,m,n\in\N$ such that the integers
$k\ell_0m(m+\ell_1n)$ and $k\ell_0(m+\ell_2n)(m+\ell_3n)$ are
distinct and belong to  the same cell.
\end{theorem}


\subsection{From partition regularity to multiplicative functions}\label{SS:link}

Much like the translation invariant case, where partition regularity
results can be deduced from corresponding density statements with
respect to a translation invariant density, we deduce
Theorem~\ref{th:partition-regular2}
  from the density regularity result of Theorem~\ref{th:density-regular} that
involves a dilation invariant density (a notion defined in
Section~\ref{SS:dildens}).

In Section~\ref{SS:actdil} we use a multiplicative version of the
correspondence principle of Furstenberg to recast
Theorem~\ref{th:density-regular}  as a recurrence property for
measure preserving actions of the multiplicative semigroup $\N$ on a
probability space (Theorem~\ref{th:recurecen1}).

In Section~\ref{SS:redpos} we   use a corollary of the spectral
theorem for unitary operators  (see identity \eqref{E:spectralo}) to
transform the recurrence result into   a
 positivity property  for an integral of averages of products of multiplicative functions
 (Theorem~\ref{th:ergo2b}).
 It is then this positivity property that we seek to prove,
and the heavy-lifting is done by a decomposition result for
multiplicative functions which is the main number theoretic result
of this article (Theorem~\ref{th:strong-average-intro}). Assuming
this result (we prove it in Sections~\ref{S:U^2}-\ref{S:U^3}), the proof of the positivity property of
Theorem~\ref{th:ergo2b} is completed in Section~\ref{SS:assuming}.
The reader will also find there
 a detailed sketch
of our proof strategy for this step. We discuss the decomposition
result next.

\subsection{Multiplicative functions and Gowers uniformity}
Our  proof of the positivity property mentioned
above necessitates that we decompose an arbitrary multiplicative
function into two components, one that we can easily control,  and
another that behaves randomly enough to have a negligible
contribution. For our purposes, 
randomness  
is  measured by the Gowers uniformity norms. Before proceeding to
the precise statement of the decomposition result we start with some
informal discussion regarding the uniformity norms and the
uniformity properties (or lack thereof) of multiplicative functions.

\subsubsection*{Gowers uniformity}
We recall the definition of the   $U^2$ and  $U^3$-Gowers uniformity
norm from \cite{G01}. Here and later, for a function $f$  defined on
a finite set $A$ we write
$$
\E_{x\in A}f(x):=\frac 1{|A|}\sum_{x\in A}f(x).
$$
\begin{definition}[Gowers uniformity norms]
Given   $ N\in\N$ and $f\colon \Z_N\to \C$, we define the
$U^2(\Z_N)$-\emph{Gowers norm}  of $f$ as follows
$$
\norm{f}_{U^{2}(\Z_N)}^4=\E_{h\in \Z_N}|\E_{n\in \Z_N}f(n+h)\cdot
\overline{f}(n)|^2\ ;
$$
and the  $U^3(\Z_N)$-\emph{Gowers norm} of $f$  as follows
$$
\norm{f}_{U^{3}(\Z_N)}^8=\E_{h_1,h_2\in \Z_N}|\E_{n\in
\Z_N}f(n+h_1+h_2)\cdot \overline{f}(n+h_1)\cdot
\overline{f}(n+h_2)\cdot f(n)|^2.
$$
\end{definition}
 In an informal way, having a small $U^2$-norm is interpreted as a
property of $U^2$-uniformity, and having a small $U^3$-norm as a
stronger property of $U^3$-uniformity. Since $\norm
f_{U^3(\Z_N)}\geq\norm f_{U^2(\Z_N)}$ for every function $f$ on
$\Z_N$, $U^3$-uniformity implies $U^2$-uniformity.

We recall that the \emph{Fourier transform} of a function $f$ on
$\Z_N$ is defined by
$$
\widehat f(\xi):=\E_{n\in\Z_N}f(n)\,\e(-n\xi/N)\ \ \text{ for }\
\xi\in\Z_N,
$$
where, as is standard, $\e(x):=\exp(2\pi ix)$.  A direct computation
gives the following identity that links the $U^2$-norm of a function
$f$ on $\Z_N$ with its Fourier coefficients:
\begin{equation}\label{E:U2formula}
\norm{f}_{U^{2}(\Z_N)}^4=\sum_{\xi\in\Z_N}|\widehat f(\xi)|^4.
\end{equation}
Using this and Parseval's identity, we see that the
$U^2$-uniformity of a function can be interpreted as the function
having small Fourier coefficients, that is, small correlation with
linear phases. We would like to stress though,
  that this property does not guarantee $U^3$-uniformity; a
function bounded by $1$  may have small  Fourier coefficients, but large
$U^3$-norm. In fact, eliminating all possible obstructions to
$U^3$-uniformity necessitates the study of correlations with all
 quadratic phases $\e(n^2\alpha+n\beta)$ and the larger class of $2$-step nilsequences
 of bounded complexity (see Theorem~\ref{th:inverse}).


\subsubsection*{Multiplicative functions} In this article, we slightly abuse
 standard terminology\footnote{The functions we call multiplicative
are often called \emph{completely multiplicative} as opposed to
functions like the M\"obius  that are  called multiplicative.}
and define multiplicative functions as follows:
\begin{definition}[Multiplicative functions] A
\emph{multiplicative function} is a function  $\chi\colon \N\to \C$
that satisfies  $\chi(mn)=\chi(m)\chi(n)$ for all $m,n\in\N$. We
denote by $\CM$ the family of multiplicative functions of modulus
$1$.
\end{definition}
Note that elements of $\CM$ are   determined by their values on the
primes.

It so happens that several multiplicative functions do not share the
strong uniformity properties of the M\"obius function established
in~\cite{GT12b}. The next example illustrates some simple but very
important obstructions to uniformity.
\begin{example}[Obstructions to uniformity]
Let $\chi\in\CM$ be defined by $\chi(2)=-1$ and $\chi(p)=1$ for
every prime $p\neq 2$. Equivalently,  $\chi(2^m(2k+1))=(-1)^m$ for
all $k,m\geq 0$. Then $\E_{1\leq n\leq N}\chi(n)= 1/3+o(1)$ and this
non-zero mean already gives an obstruction to $U^2$-uniformity. But
this is not the only obstruction. Indeed,  we have $\E_{1\leq n\leq
N}(-1)^n\chi(n)=-2/3+o(1)$,
  and this  implies that  $\chi-1/3$ does not
 have small $U^2$-norm.
\end{example}
Examples similar to the previous one show that normalized
multiplicative functions can have significant correlation with
  periodic phases and thus this is an
  obstruction to
$U^2$-uniformity that we should  take   into account.
However, it  is a non-trivial fact that plays a central role in this
article, that correlation with periodic phases are, in a sense to
be made precise later, the only obstructions
 not only to $U^2$-uniformity but also to $U^3$-uniformity of multiplicative functions.
 For $U^2$-uniformity this is already
indicated by an old result of Daboussi and Delange~\cite{DD74, DD82} which states that
if $\alpha$ is irrational, then $ \sup_{\chi\in \CM}|\E_{1\leq n\leq N}\chi(n) \ \! \e(n\alpha)|\to 0$ as $N\to \infty$. The proof of
this result was  later simplified and extended by K\'atai \cite{K86}
(for good quantitative versions of these results see \cite{BSZ12,
MV77}), and it is a simple  orthogonality criterion  obtained in
this article of K\'atai (see Lemma~\ref{lem:katai}) that is going to
be a key number theoretic input for this article.

\subsection{$U^3$-Decomposition of multiplicative functions}
\label{subsec:decomposition} We proceed now to the formal statement of
our main  decomposition result.
We are given
positive integers $\ell_1,\ell_2, \ell_3$ and these are considered
fixed through this article.
From this point on we let $\ell$ be the number
$$
\ell:=\ell_1+\ell_2+\ell_3.
$$
 Also, given $N\in \N$ we let
$$
[N]:=\{1,\ldots, N\}.
$$
It is often easier to work on a cyclic group rather  than  an
interval of integers, as this makes Fourier analysis tools more
readily available. In order to avoid roundabout issues, we introduce
the following notation. Given $N\in \N$,  we denote by $\wt{N}$ the
smallest prime that is greater than $10\ell N$. By Bertrand's
postulate we have $\wt N \leq 20\ell N$. For every multiplicative
function $\chi\in \CM$ and every $N\in\N$, we denote by
 $\chi_N$ the function on $\Z_\tN$, or on $[\tN]$, defined by
\begin{equation}
\label{eq:def-chiN}
\chi_N(n)=\begin{cases} \chi(n)& \text{if }n\in[N];\\
0&\text{otherwise.}
\end{cases}
\end{equation}
The domain of $\chi_N$ will be each time clear from the context.
Working with the truncated function $\chi\cdot \one_{[N]}$, rather
than the function $\chi$, is a technical maneuver and the reader
will not lose  much by ignoring the cutoff. We should stress  that
from this point on, Gowers norms are going to be defined and Fourier
analysis is going to happen on the group $\Z_\tN$ and not on the
group $\Z_N$.

We can now state the main decomposition result that will be used
below in the proof of the theorems of partition regularity.
 Its essence is that the restriction of an arbitrary
multiplicative function $\chi\in \CM$ on a finite interval $[N]$ can be decomposed into
three pieces, one that is approximately periodic, one that has small
$L^1$-norm, and one that has extremely small $U^3$-norm. In
addition, the structured component enjoys  some very important
features, for example, it is a convolution product of $\chi_N$
with a positive kernel that is independent of $\chi$  and its approximate period
is bounded by a constant that does not depend on $\chi$ or on $N$.

\begin{definition}
By a \emph{kernel} on $\ZN$ we mean a non-negative function with
average $1$.
\end{definition}

\begin{theorem}[Strong decomposition on average for the $U^3$-norm]
\label{th:strong-average-intro} For every positive finite measure
$\nu$ on the compact group $\CM$ of multiplicative functions having
modulus $1$,
 every function $F\colon\N \times\N \times
\R^+\to\R^+$,  every $\ve >0$,  and
every
sufficiently large  $N\in \N$, depending only on  $F$ and $\ve$,
 there exist positive integers
  $Q$ and  $R$ that are  bounded by a constant which depends only on $F$ and
  $\ve$,
  such that, for  every $\chi\in\CM$, the function $\chi_N$ admits the
 decomposition
$$
\chi_N(n)=\chi_{N,s}(n)+\chi_{N,u}(n)+\chi_{N,e}(n) \quad \text{ for
every }\  n\in\ZN,
$$
where  $\chi_{N,s}$, $\chi_{N,u}$, and
   $\chi_{N,e}$ satisfy the following properties:
\begin{enumerate}
\item
\label{it:decomU31}
\vide $\chi_{N,s}=\chi_N*\psi_{N,1}$ and  $\chi_{N,s}+\chi_{N,e}=\chi_N*\psi_{N,2}$,
where  $\psi_{N,1}$ and $\psi_{N,2}$ are kernels on $\Z_\tN$ that do not depend on $\chi$, and
  the convolution product is  defined in $\ZN$;

\item\label{it:decompU32}
\vide $\displaystyle |\chi_{N,s}(n+Q)-\chi_{N,s}(n)|\leq \frac R\tN$
for every $n\in\ZN$,  where  $n+Q$ is taken $\!\!\! \mod \tN$;
\item
\label{it:decomU33}
\vide $\displaystyle \norm{\chi_{N,u}}_{U^3(\ZN)}\leq\frac 1{F(Q, R,\ve)}$;
\item
\label{it:decomU34}
\vide $ \E_{n\in \ZN} \int_\CM |\chi_{N,e}(n)|\,d\nu(\chi)\leq\ve$.
\end{enumerate}
\end{theorem}
\subsubsection*{Remarks}
(1) For arbitrary bounded sequences, decomposition results with
similar flavor have been proved in \cite{G10, GW11, GT10,
  Sz12, T06}, but working in this generality necessitates the
use of structured components that  do  not  satisfy the strong
rigidity condition of  Property \eqref{it:weakU3-2}.
 An additional important feature of our result
  is  that the structured component  is defined by a
convolution product with a kernel that is independent of $\chi$. All
these properties
play an important role in the derivation of the combinatorial result in
Section~\ref{S:Recurrence}.

Extensions of Theorem~\ref{th:strong-average-intro}
that give information about the $U^k$-norm of the component $\chi_{N,u}$ for $k\geq 4$
appear to be significantly more complicated, and as they are not needed for the purposes
of this article will not be established here.

(2)
We only plan to use this theorem  for the function $F(x,y,z)=c
x^2y^2/z^4$ where $c$ is a constant that depends on $\ell$ only.
Restricting the statement to this function  does not simplify
our proof though.

(3)
 The bound~\eqref{it:decomU34} is not  uniform in $\chi$ as in part~\eqref{it:weakU3_3}
of Theorem~\ref{th:Decomposition-weakU3-intro} below. We do not know if
this weakening on the bound is needed or is an artifact of our
proof.

(4)
It is a consequence of Property~\eqref{it:decomU31} that for fixed $F,N,  \ve, \nu$, the
maps $\chi\mapsto \chi_{s}, \chi\mapsto\chi_{u}, \chi\mapsto
\chi_{e}$ are continuous, and
 $\norm{\chi_s}_{L^\infty(\ZN)}\leq 1$,
 $\norm{\chi_u}_{L^\infty(\ZN)}\leq 2$,  $\norm{\chi_e}_{L^\infty(\ZN)}\leq 2$.


\medskip
Most of the work in the proof of
Theorem~\ref{th:strong-average-intro} goes into verifying the
decomposition result of Theorem~\ref{th:Decomposition-weakU3-intro}
that gives weaker bounds on the uniform component of the
decomposition. Two ideas that play a prominent role in its proof,
roughly speaking, are:

(a)   A multiplicative function that has $U^2$-norm bounded away
from zero correlates with a linear  phase that has frequency close
to a rational with small denominator.

(b)  A multiplicative function that has $U^3$-norm bounded away
from zero necessarily has $U^2$-norm bounded away from zero.

The proof of (a) uses classical  Fourier analysis tools and  is
given in Section~\ref{S:U^2}. The key number theoretic input is the
orthogonality criterion of K\'atai stated in Lemma~\ref{lem:katai}.

  The proof of (b) is much harder and is done in
several steps using higher order Fourier analysis machinery. In
Section~\ref{S:CorrelationNil} we combine K\'atai's criterion with a
quantitative equidistribution result on nilmanifolds  of Green and
Tao (Theorem~\ref{th:Leibman}) to
 study the correlation of multiplicative functions with
 nilsequences. These results  are then combined
 in Section~\ref{S:U^3}
with modifications of  the $U^3$-inverse theorem  (Theorem~\ref{th:inverse}) and
the  factorization result
  (Theorem~\ref{th:factorization}) of Green and Tao, to
conclude the proof of the weak decomposition result of
Theorem~\ref{th:Decomposition-weakU3-intro}. We defer the reader to
Section~\ref{subsec:prelim} for a more detailed sketch of our proof
strategy.

Upon proving the weak decomposition result of
Theorem~\ref{th:Decomposition-weakU3-intro}, the proof of
Theorem~\ref{th:strong-average-intro}  consists of a Fourier
analysis energy increment argument, and  avoids the use of finitary
ergodic theory  and  the use of the Hahn-Banach theorem, tools that
are typically used  for other decomposition results (see \cite{G10,
GW11,GW11b, GT10, T06}).

\subsection{Further directions}
Theorem~\ref{th:partition-regular1} establishes that the equation
\begin{equation}\label{eq:basic1}
ax^2+by^2=cn^2
\end{equation}
 is partition regular
 provided that all three integers
  $ac, bc$, $(a+b)c$,  are non-zero squares.
Two interesting
cases, not covered by the previous result,  are the following:
\begin{problem}\label{prob1}
Are the equations $x^2+y^2=n^2$ and $x^2+y^2=2n^2$ partition
regular?\footnote{Note that the equation $x^2+y^2=3n^2$ does not have solutions in $\N$.
Furthermore, the equation $x^2+y^2=5n^2$  has solutions in $\N$ but it is not
partition regular. Indeed, if we partition the integers in $6$ cells
according to whether their first non-zero digit in the $7$-adic
expansion is $1,2,\ldots, 6$, it turns out  that for every
$n\in \N$ the equation has no solution on any single partition
cell.}
\end{problem}
 Let us explain why we cannot yet handle these equations using the methods
 of this article. The equation
$x^2+y^2=2z^2$ has the following solutions: $x=k(m^2-n^2+2mn)$,
$y=k(m^2-n^2-2mn)$, $z=k(m^2+n^2)$, where $k,m,n\in \Z$. The values
of $x$ and $y$ do not factor in linear terms, which leads to the
major obstacle of not being able to establish uniformity estimates
analogous to the ones stated in Lemma~\ref{L:UnifromityEstimates2}.
The equation $x^2+y^2=z^2$ has the following solutions:
$x=k(m^2-n^2)$, $y=2kmn$, $z=k(m^2+n^2)$. In this case, it is
possible to  establish the needed uniformity estimates but we are
not able to carry out the argument of Section~\ref{SS:assuming}
 in order to prove the relevant positivity property (see footnote~\ref{foot3} below for more details).

 A set $E\subset\N$ has \emph{positive
\emph{(additive)} upper density} if $\limsup_{N\to\infty}|E\cap [N]|/N>0$.
It turns out that the equations of Corollary~\ref{Corol1}
 have non-trivial solutions on
every
 infinite arithmetic progression, making the following statement  plausible:
\begin{problem}
Does every set $E\subset \N$ with positive density contain distinct
$x,y\in \N$ that satisfy  the equation $16x^2+9y^2=n^2$ for some $n\in\N$?
\end{problem}
We say that the equation $p(x,y,z)=0$, $p\in \Z[x,y,z]$,  has
\emph{no local obstructions} if for every infinite arithmetic progression $P$, there exist distinct $x,y,z\in P$ that satisfy the equation. For example, the equations $x^2+y^2=2z^2$ and $16x^2+9y^2=25z^2$
 have no local obstructions.
\begin{problem}
Let $p\in \Z[x,y,z]$ be
 a homogeneous quadratic form and suppose that the equation $p(x,y,z)=0$ has no local obstructions. Is it true that
every subset  of $\N$ of positive  density contains distinct $x,y,z$ that satisfy the equation?
\end{problem}
As we mentioned before,  there are no values of $a,b,c\in \N$  for
which  the equation
\begin{equation}\label{eq:basic2}
ax^2+by^2=cz^2
\end{equation}
is known to be partition regular and  the condition ``at least one of $a,b,$ and  $a+b$
 equals  $c$'' is necessary for partition regularity.
\begin{problem}
Are there  $a,b,c\in \N$ for which  equation \eqref{eq:basic2} is
partition regular? 
\end{problem}
Notable examples are the equations $x^2+y^2=z^2$ and $x^2+y^2=2z^2$.
In \cite{GR12} it is conjectured that the second equation is
partition regular.

\subsection{Notation and conventions}
  We denote by $\N$ the set of positive integers.

 \smallskip

\noindent For $N\in \N$ we denote by $[N]$ the set $\{1,\ldots,
N\}$.

\smallskip

\noindent For a function $f$  defined on a finite set $A$ we write $
\E_{x\in A}f(x)=\frac 1{|A|}\sum_{x\in A}f(x). $

\smallskip

\noindent With $\CM$ we denote the set of multiplicative functions
$\chi\colon \N\to \C$ with modulus $1$.

\smallskip

\noindent Throughout, we assume that we are given
$\ell_1,\ell_2,\ell_3\in \N$ and we set $\ell=\ell_1+\ell_2+\ell_3$.

\smallskip
\noindent A  kernel on $\Z_N$  is a non-negative function on $\Z_N$
with average $1$.

\smallskip

\noindent For  $N\in \N$ we let $\tN$ be the smallest prime that is
larger than $10\ell N$ (then $\tN\leq 20 \ell N$).

\smallskip

\noindent Given $\chi\in \CM$ and $N\in \N$ we let $\chi_N\colon
[\tN]\to \C$ be defined by $\chi_N=\chi \cdot \one_{[N]} $. The domain
of $\chi_N$ is sometimes  thought to be $\Z_\tN$.

\smallskip  \noindent For technical reasons, throughout the
article all Fourier analysis happens on $\Z_\tN$   and all
uniformity norms are defined on $\Z_\tN$.

\smallskip
\noindent If $x$ is a real, $\e(x)$ denotes the number $\exp(2\pi i
x)$, $\norm x$ denotes the distance between $x$ and the nearest
integer, $\lfloor x\rfloor$ the largest integer smaller or equal than $x$,
and $\lceil x\rceil$ the smallest integer greater or equal than $x$.

\smallskip

\noindent Given $s\in \N$ we write $\bk=(k_1,\dots,k_s)$ for a point
of $\Z^s$ and $\norm\bk=|k_1|+\dots+|k_s|$. For
$\bu=(u_1,\dots,u_s)\in\T^s$, we write
$\bk\cdot\bu=k_1u_1+\dots+k_su_s$.

\smallskip

\noindent Let $f$ be a  function on a metric space $X$ with distance
$d$. We define
$$
\norm f_{\lip(X)}=\sup_{x\in X}|f(x)|+\sup_{\substack{x,y\in X\\
x\neq y}}\frac{|f(x)-f(y)|}{d(x,y)}.
$$

\smallskip

\noindent There is a proliferation of constants in this article and
our general principles are as follows: The constants
$\ell_1,\ell_2,\ell_3,$ are considered as fixed throughout the
article, and quantities depending only on these numbers are
considered as universal constants. The letters   $\ell_0$, $\ell$,
and  $c, c_1, c_2,...$,   are reserved for  constants of this type
independently of whether they represent small or large quantities.
Quantities that depend on one or more variables are denoted by Roman
capital letters $C, D, K,...$ if they represent large quantities,
and by low case Greek letters $\gamma, \delta, \ve,\dots$ if they
represent small quantities. It will be very clear from the context
when we deviate from these rules.

\subsection{Acknowledgements.} We would like to thank Wenbo Sun for pointing out a mistake  in
the proof of Proposition~\ref{lem:equid-square} in  an earlier  version of this article.

\section{Proof of partition regularity assuming the decomposition result.}\label{S:Recurrence}
The goal of this  section is to prove the combinatorial Theorem~\ref{th:partition-regular2} assuming
the decomposition result of Theorem~\ref{th:strong-average-intro}
(this is proved in Sections~\ref{S:U^2}-\ref{S:U^3}).
We begin by giving three successive reformulations of Theorem~\ref{th:partition-regular2}.

\subsection{Reduction to a density regularity result}\label{SS:dildens}
We first recast Theorem~\ref{th:partition-regular2} as a density regularity
statement for dilation invariant densities on the integers.
\begin{definition}[Multiplicative F\o lner sequence]
  The sequence $(\Phi_N)_{N\in\N}$ of finite subsets of $\N$  is a
\emph{multiplicative F\o lner sequence} if  for every $n\in\N$
$$
\lim_{N\to+\infty} \frac{|n\Phi_N \triangle  \Phi_N|}{|\Phi_N|}=0
$$
where
$n\Phi_N:=\{nx\colon x\in\Phi_N\}$.
\end{definition}
The sequence $(\Phi_N)_{N\in\N}$ defined by $
\Phi_N:=\{p_1^{k_1}\cdots p_N^{k_N}\colon 0\leq k_1,\ldots, k_N\leq
N\}$,
 where $p_1,p_2,\ldots$ is the sequence of primes, serves as a
 typical example.
To a given  multiplicative F\o lner sequence we
associate a notion of multiplicative density as follows:
\begin{definition}[Multiplicative density]\label{def:multdensity}
The \emph{multiplicative density} $\dmult(E)$ of a subset $E$ of
$\N$ (relatively to the multiplicative  F\o lner sequence
$(\Phi_N)$) is defined as
$$ \dmult(E):= \limsup_{N\to+\infty}
\frac{|E\cap\Phi_N|}{|\Phi_N|} .
$$
\end{definition}
We remark that the multiplicative density and the additive
density are non-comparable measures of largeness.
For
instance,   the set of odd numbers has zero multiplicative density with respect to any multiplicative  F\o lner sequence,
as has any set that omits all multiples of some positive
integer. On the other hand, it is not hard to construct sets with
multiplicative density $1$ that have additive density $0$ (see for
instance \cite{Be05}).

An important property of the multiplicative density, and the reason
we work with this notion of largeness,  is that for every $E\subset
\N$ and $n\in\N$, we have
$$
\dmult(nE)=\dmult(E)=\dmult(n\inv E), \ \ \text{ where }\ \  n\inv
E:=\{x\in\N\colon nx\in E\}.
$$

Since a  multiplicative density is clearly sub-additive,
any finite partition of $\N$ has at least one cell with positive
multiplicative density. Hence,  Theorem~\ref{th:partition-regular2} follows
from the following stronger result:
\begin{theorem}[Density regularity]
\label{th:density-regular}
 Let $\ell_0,
\ell_1,\ell_2,\ell_3$ be as in Proposition~\ref{prop:linearfactors}.
Let  $E\subset\N$ be a set with  positive multiplicative density.
Then there exist $k,m,n\in\N$ such that the integers $k\ell_0 m(m+\ell_1n)$ and
$k\ell_0 (m+\ell_2n)(m+\ell_3n)$ are distinct and belong to $E$.
\end{theorem}
In fact, we show that for a set of $(m,n)\in\N^2$ of positive
(additive) density the asserted property holds for a set of $k\in
\N$ of positive multiplicative density.
\subsection{Reduction to a recurrence results for actions by dilations}\label{SS:actdil}
 Our next goal is
to reformulate the density statement of Theorem~\ref{th:density-regular}
as a recurrence statement in ergodic theory.

\begin{definition}
An \emph{action by dilations} on a probability space $(X,\CB,\mu)$
is a family  $(T_n)_{n\in\N}$ of invertible measure preserving
transformations of $(X,\CB,\mu)$ that satisfy
$$ T_1:=\id \ \ \text{ and
}\text{for every }\ m,n\in\N,\quad T_m\circ T_n=T_{mn}. $$
\end{definition}
We remark  that an action by dilations on a probability space $(X,\CB,\mu)$
can be extended to a measure preserving action $(T_r)_{r\in\Q^+}$ of the multiplicative group $\Q^+$
by defining
$$
T_{a/b}:=T_aT_b\inv\ \ \text{ for all }\ \  a,b\in\N.
$$

We will use a multiplicative version of the (additive)
correspondence principle of Furstenberg \cite{Fu81}.
Its proof
can be found in
\cite{Be05}.
\begin{MCP}
Let $E$ be a subset of $\N$.
Then there exist an action by dilations  $(T_n)_{n\in\N}$ on a
probability space $(X,\CB, \mu)$, and a set $A\in \CB$ with
$\mu(A)=\dmult(E)$, such that  for every $k\in\N$ and for all
$n_1,n_2,\dots , n_k\in\N$, we have
$$
\dmult\bigl( n_1^{-1}E\cap  n_2^{-1}E\cap \dots\cap
n_k^{-1}E\bigr)\geq \mu(T_{n_1}^{-1}A\cap T_{n_2}^{-1}A\cap\dots\cap
T_{n_k}^{-1}A).
$$
\end{MCP}
Using this correspondence principle we can recast Theorem~\ref{th:density-regular}  as a recurrence statement in ergodic theory regarding actions by dilations.
\begin{theorem}[Recurrence]
\label{th:recurecen1}
 Let $
\ell_1,\ell_2,\ell_3$ be   as in Proposition~\ref{prop:linearfactors}.
  Let $(T_n)_{n\in\N}$ be an action by dilations on a probability
space $(X,\CB,\mu)$. Then for every $A\in \CB$ with $\mu(A)>0$,
there exist $m,n\in\N$, such that the integers $m(m+\ell_1n)$ and
$(m+\ell_2n)(m+\ell_3n)$  are distinct,  and
$$
\mu\bigl(T_{m(m+\ell_1n)}\inv A\cap
T_{(m+\ell_2n)(m+\ell_3n)}\inv A\bigr)>0.
$$
\end{theorem}
\begin{proof}[Proof of Theorem~\ref{th:density-regular} assuming Theorem~\ref{th:recurecen1}]
Let  $E\subset  \N$  have positive multiplicative density.
Let the probability space $(X,\CB, \mu)$, the action
$(T_n)_{n\in\N}$, and the set $A\in \CB$ be associated to $E$ by the previous correspondence
principle. Let $\ell_0,\ell_1,\ell_2,\ell_3$ be as in  Proposition~\ref{prop:linearfactors}. In order to show  that there exist  integers $k,m, n\in\N$
satisfying the conclusions of Theorem~\ref{th:density-regular}, it
suffices to show that there exist   $m, n\in\N$ so that
 the integers $\ell_0m(m+\ell_1n)$ and
$\ell_0(m+\ell_2n)(m+\ell_3n)$ are distinct, and satisfy
$$
\mu\bigl(T_{\ell_0 m(m+\ell_1n)}\inv A\cap T_{\ell_0 (m+\ell_2n)(m+\ell_3n)}\inv
A\bigr)>0.
$$
Since $\mu$ is $T_{\ell_0}$-invariant, the left hand side equals
$
\mu\bigl(T_{m(m+\ell_1n)}\inv A\cap T_{(m+\ell_2n)(m+\ell_3n)}\inv A\bigr),
$
and the existence of $m,n\in \N$ satisfying the asserted properties
follows from Theorem~\ref{th:recurecen1}.
\end{proof}
 The degenerate case  where $\ell_1=\ell_2$ (and similarly if $\ell_1=\ell_3$ or $\ell_2\ell_3=0$)
is rather trivial. Indeed,
  we are then reduced to
establishing positivity for $\mu\bigl(T_{m}\inv A\cap
T_{m+\ell_3n}\inv A\bigr)$. Letting $m=\ell_3, n=2^{n'}-1$, we
further reduce matters to showing that $ \mu\bigl( A\cap T_2^{-n'}
A\bigr)>0 $ for some $n'\in \N$, and this follows from the Poincar\'e
recurrence theorem applied to  $T_2$. Therefore,   in the rest of
this article, we can and \emph{we will assume that $\ell_1,\ell_2,
\ell_3$ are distinct positive integers}. In this case, our goal is to  show that
the set of pairs $(m,n)$ satisfying the conclusion of
Theorem~\ref{th:recurecen1}  has positive (additive) density in
$\N^2$:
\begin{theorem}[Recurrence on the average]
 \label{th:Recurrence}
 Let $\ell_1,\ell_2,\ell_3\in \N$ be  distinct.
  Let $(T_n)_{n\in\N}$ be an action by dilations on a probability
space $(X,\CB,\mu)$. Then for every $A\in \CB$ with $\mu(A)>0$ we have
$$\liminf_{N\to\infty}\,
\E_{(m,n)\in\Theta_N}\;
\mu\bigl(T_{m(m+\ell_1n)}\inv A\cap
T_{(m+\ell_2n)(m+\ell_3n)}\inv A\bigr)>0
$$
where $\Theta_N=\{(m,n)\in[N]\times[N]\colon 1\leq m+\ell_in\leq N\text{ for }i=1,2,3
\}$.
\end{theorem}
\begin{remark}
\label{rem:three} In fact, we  prove more: the $\liminf$ is greater
or equal than  a positive constant that
 depends only on the measure of the set $A$.
\end{remark}
\begin{proof}[Proof of Theorem~\ref{th:recurecen1} assuming Theorem~\ref{th:Recurrence}]
It suffices to notice  that for $N$ sufficiently large we have $|\Theta_N|\geq c_1N^2$ and  the cardinality of the set of pairs $(m,n)\in\Theta_N$ that
satisfy $m(m+\ell_1n)=(m+\ell_2n)(m+\ell_3n)$ is bounded by $c_2N$ for some
constants $c_1$ and $c_2$ that depend only on $\ell_1,\ell_2, \ell_3$.
\end{proof}

\subsection{Reduction to a positivity property for multiplicative functions}\label{SS:redpos}
Next, we  show that Theorem~\ref{th:Recurrence} is equivalent to a
  positivity property for multiplicative functions.

Recall that the set $\CM$ consists of all multiplicative functions of modulus $1$. When endowed with the topology of
pointwise convergence, $\CM$ is a compact (metrizable) Abelian
group. If  $\{p_1,p_2,\dots\}$ denotes the set of primes,  then a
multiplicative function $\chi$ is determined by its
values on the primes. The map $\chi\mapsto (\chi(p_n))_{n\in \N}$ is
an isomorphism between the groups $\CM$ and $\T^\N$. The space of
multiplicative functions $\CM$ is the dual group of the
multiplicative group $\Q^+$, the duality being given by
$$
\chi(m/n)=\chi(m)\overline\chi(n) \ \ \text{ for every }\ \ \chi\in\CM
\ \text{ and every }\ m,n\in\N.
$$

Recall that an  action $(T_n)_{n\in\N}$ by dilations on a probability
space $(X,\CB,\mu)$ extends to a measure preserving action of the
multiplicative group $\Q^+$ on the same space. Since $\CM$ is the
dual group of this countable Abelian group, by the spectral theorem
for unitary operators, for every function $f\in L^2(\mu)$ there
exists a positive finite
  measure $\nu$ on the compact Abelian group $\CM$,  called the \emph{spectral measure of $f$}, such that, for all $m,n\in\N$,
\begin{equation}\label{E:spectralo}
\int_X T_mf\cdot T_n\overline f\,d\mu
=\int T_{m/n}f\cdot \overline f\,d\mu=\int_{\CM}\chi(m/n)\,d\nu(\chi)
=\int_\CM\chi(m)\, \overline\chi(n)\,d\nu(\chi).
\end{equation}

Let $A\in\CB$. Letting $f=\one_A$ in~\eqref{E:spectralo} and using
the multiplicativity of elements of $\CM$, we get for all
$m,n\in\N$ that
\begin{align*}\label{eq:spectral}
  \mu\big(T_{m(m+\ell_1n)}\inv A\cap
T_{(m+\ell_2n)(m+\ell_3n)}\inv A\big)&=\int_{\CM}
\chi\big(m(m+\ell_1n)\big)\,\overline\chi\big(m+\ell_2n)(m+\ell_3n)\big)\,d\nu(\chi)\\
&=\int_{\CM}
\chi(m)\chi(m+\ell_1n)\,\overline\chi(m+\ell_2n)\,\overline\chi(m+\ell_3n)d\nu(\chi).
\end{align*}
 From this identity we deduce   that Theorem~\ref{th:Recurrence} is equivalent to the following result:

\begin{theorem}[Spectral reformulation of recurrence result I]
\label{th:ergo2}
 Let $\ell_1,\ell_2,\ell_3\in \N$ be distinct and
$(T_n)_{n\in\N}$ be an action by dilations on a probability space
$(X,\CB,\mu)$. Then for every $A\in
\CB$ with $\mu(A)>0$, writing $\nu$ for the spectral measure of
${\bf 1}_A$, we have
\begin{equation}
\label{eq:average-integral1}
\liminf_{N\to\infty}  \int_\CM\E_{(m,n)\in\Theta_N}
\chi(m)\chi(m+\ell_1n)\,\overline\chi(m+\ell_2n)\,\overline\chi(m+\ell_3n)\,d\nu(\chi)>0
\end{equation}
where $\Theta_N$ is as  in Theorem~\ref{th:Recurrence}.
\end{theorem}
\begin{remark}
An alternate (and arguably more natural) way to try to prove
Theorem~\ref{th:recurecen1} is to replace the additive averages
in Theorem~\ref{th:ergo2} with multiplicative ones. Upon doing this,
one is required to analyze averages of the form
$$
 \E_{m,n\in
\Phi_N}
\chi\big(m (m+\ell_1n)\big)\,\overline\chi\big((m+\ell_2n)(m+\ell_3n)\big)
$$
where $(\Phi_N)_{N\in\N}$ is a multiplicative F\o lner sequence in $\N$ and $\chi \in \CM$. Unfortunately, we were not able to
prove anything useful for these multiplicative averages,  although one suspects that a positivity property similar to the one in  \eqref{eq:average-integral1} may hold.
\end{remark}
Next, for technical reasons we   reformulate Theorem~\ref{th:ergo2} as a positivity property
involving averages over $\Z_\tN$. This is going to be the final form
of the recurrence statement that we aim to study.  Recall that $\wt
N$ was defined in Section~\ref{subsec:decomposition}
 and the functions
$\chi_N$ on $\ZN$ were defined by~\eqref{eq:def-chiN} in the same
section.

\begin{theorem}[Spectral reformulation of recurrence result II]
\label{th:ergo2b}
Let $(T_n)_{n\in\N}$ be an
action by dilations on a probability space $(X,\CB,\mu)$. Then for
every $A\in \CB$ with $\mu(A)>0$, writing $\nu$ for the spectral
measure of ${\bf 1}_A$, we have
\begin{equation}
\label{eq:average-integral1b}
\liminf_{N\to\infty}  \int_\CM\E_{m,n\in\ZN}
\one_{[N]}(n)\,
\chi_N(m)\chi_N(m+\ell_1n)\,\overline\chi_N(m+\ell_2n)\,\overline\chi_N(m+\ell_3n)\,d\nu(\chi)>0,
\end{equation}
where in the above average  the expressions $m+\ell_in$ can be considered as elements of  $\Z$ or $\ZN$ without affecting the value of the average.
\end{theorem}
We check that Theorems~\ref{th:ergo2} and~\ref{th:ergo2b} are equivalent.
Using the definition of the set $\Theta_N$ given in Theorem~\ref{th:Recurrence}, we can rewrite the averages that appear in the statement of Theorem~\ref{th:ergo2} as follows
\begin{multline*}
\E_{(m,n)\in\Theta_N}
\chi(m)\chi(m+\ell_1n)\,\overline\chi(m+\ell_2n)\,\overline\chi(m+\ell_3n)=\\
\frac{\wt N^2}{|\Theta_N|}\;\E_{m,n\in[\wt N]} \one_{[N]}(n)\,
\chi_N(m)\chi_N(m+\ell_1n)\,\overline\chi_N(m+\ell_2n)\,\overline\chi_N(m+\ell_3n).
\end{multline*}
 The value of the last  expression remains unchanged when we replace each term $m+\ell_in$ by $m+\ell_in\bmod \wt N$.
Using this identity and that $cN^2\leq |\Theta_N|\leq N^2$ for some positive constant $c$ that depends only on $\ell$, we get the asserted equivalence.

\subsection{Some estimates involving Gowers norms}
\label{subsec:gowers}
Next we establish
two elementary estimates that will be  used in the sequel.
The first one will be used in Section~\ref{S:U^3}.
\begin{lemma}
\label{lem:U2-intervals}  Let $N$ be prime. For every  function $a\colon \Z_N\to \C$
 and for every arithmetic progression $P$ contained in the
interval $[N]$, we have
$$
 \big|\E_{n\in [N]}\one_P(n)\cdot a(n) \big|\leq
 c_1\norm{a}_{U^2(\Z_N)}
$$
for some universal constant $c_1$.
\end{lemma}
\begin{proof}
Since $N$ is prime, the $U^2$-norm of a function on $\Z_N$ is
invariant under any change of variables of the form $x\mapsto ax+b$,
where $a,b\in\N$  and $a\neq 0 \! \! \! \mod{N}$. By a change of variables of this
type, we are reduced to the case that $P$ is an interval
$\{0,\dots,m\}$ with $0\leq m<N$, considered as a subset of
$\Z_N$. A direct computation then shows that
$$
 |\widehat{\one_P}(\xi)|\leq
\frac{2}{N||\xi/N||}= \frac{2}{\min\{\xi,N-\xi\}}  \quad \text{for
}\ \xi=1,\ldots,N-1,
$$
and as a consequence
$$
\norm{\widehat{\one_P}(\xi)}_{l^{4/3}([N])}
\leq c_1
$$
for some universal constant $c_1$. Using this estimate, Parseval's
identity,   H\"older's inequality, and identity \eqref{E:U2formula},
we deduce that
$$
\big|\E_{n\in [N]}\one_P(n)\cdot a(n) \big|=\big|\sum_{\xi\in [N]}\widehat \one_P(\xi)\cdot \widehat a(\xi) \big|\leq
c_1\cdot\Bigl(\sum_{\xi\in[N]} |\widehat a(\xi)|^4\Bigr)^{1/4}=
c_1\,\norm a_{U^2(\Z_N)}.\qed
$$
\renewcommand{\qed}{}
\end{proof}
The next estimate is key for the proof of Theorem~\ref{th:ergo2b}.
It is the reason we seek for  a $U^3$-decomposition result in this
article.
\begin{lemma}[$U^3$-uniformity estimates]\label{L:UnifromityEstimates2}
Let  $a_i$, $i=0,1,2,3$, be functions on $\ZN$ with
$\norm{a_i}_{L^\infty(\ZN)}\leq 1$ and $\ell_1,\ell_2,\ell_3\in \N$ be distinct. Then
there exists a constant $c_2$, depending only on $\ell$, such that
 $$
\big|\E_{m,n\in\ZN}\one_{[N]}(n)\cdot a_0(m)
a_1(m+\ell_1n) a_2(m+\ell_2n)a_3(m+\ell_3n)\big| \leq c_2
\min_{0\leq j\leq 3}(\norm{a_j}_{U^3(\ZN)})^{1/2}+\frac{2}{\tN}.
$$
\end{lemma}
\begin{proof} We first  reduce matters to estimating a similar average that does not contain
 the term $\one_{[N]}(n)$.
Let $r$ be an integer that will be specified later and satisfies
$0<r< N/2$. We define the ``trapezoid function'' $\phi$ on $\ZN$ so
that $\phi(0)=0$, $\phi$ increases linearly from $0$ to $1$ on the
interval $[0,r]$, $\phi(n)=1$ for $r\leq n\leq N-r$, $\phi$
decreases linearly from
 $1$ to $0$ on $[N-r,N]$, and $\phi(n)=0$ for $N<n<\wt N$.

The absolute value of the difference between the average in the statement and
$$
\E_{m,n\in\ZN}\phi(n)\cdot a_0(m)\cdot a_1(m+\ell_1n)\cdot
a_2(m+\ell_2n)\cdot a_3(m+\ell_3n)
$$
is bounded by $2r/\wt N$.

Moreover, it is classical that
$$
\sum_{\xi\in\ZN}|\widehat\phi(\xi)|\leq \frac{2N}{r}\leq \frac{\wt N}r
$$
  and thus
\begin{multline*} \Bigl|\E_{m,n\in\ZN}\phi(n)\cdot
a_0(m)\cdot a_1(m+\ell_1n)\cdot
a_2(m+\ell_2n)\cdot a_3(m+\ell_3n)\Bigr|\leq \\
 \frac{\wt N}r\,\max_{\xi\in\ZN}
\Bigl|\E_{m,n\in\ZN}\e(n\xi/\tN)\cdot a_0(m)\cdot
a_1(m+\ell_1n)\cdot a_2(m+\ell_2n)\cdot a_3(m+\ell_3n)\Bigr|.
\end{multline*}

Furthermore, notice that upon replacing $a_0(n)$ with
$a_0(n)\e(\ell_1^*n\xi/\tN)$ and $a_1(n)$ with
$a_1(n)\e(-\ell_1^*n\xi/\tN)$, where $\ell_1^*\ell_1=1\!\!\!
\mod{\tN}$, the $U_3$-norm of all sequences remains unchanged, and
the term $\e(n\xi/\tN)$ disappears. We are thus left with estimating
the average
$$\E_{m,n\in\ZN} a_0(m)\cdot a_1(m+\ell_1n)\cdot
a_2(m+\ell_2n)\cdot a_3(m+\ell_3n),
$$
which   is   known  (see for example \cite[Theorem 3.1]{T06}) to be
bounded by
$$ U:=\min_{0\leq j\leq
3}\norm{a_j}_{U^3(\ZN)}.
$$

Combining the preceding estimates, we get that  the average in the
statement is bounded by
$$
\frac {2r}{\wt N} +
\frac{2\wt N}r U.
$$
Assuming that $U\neq 0$ and choosing   $r= \lfloor \sqrt{U}\tN/(8\ell)\rfloor+1$ (then $r\leq
\tN/(8\ell)\leq N/2$)
gives  the announced bound.
\end{proof}

\subsection{A positivity property}
We  derive now a positivity property that will be used in the
proof of Theorem~\ref{th:ergo2b} in the next subsection. Here we make
essential use of the fact that the spectral measure  $\nu$ is
associated to a non-negative function on $X$, and also that the
function that defines the convolution product is non-negative.
\begin{lemma}[Hidden non-negativity]\label{L:HiddenPositivity}
Let the action by dilations  $(T_n)_{n\in\N}$ on the probability
space $(X,\CB,\mu)$, the subset $A$ of $X$, and the spectral measure
$\nu$ on $\CM$ be as in Theorem~\ref{th:ergo2b}. Let $\psi$ be a
non-negative function defined on $\ZN$. Then
 $$
\int_\CM (\chi_N*\psi)(n_1)\cdot (\chi_N*\psi)(n_2)\cdot
 (\overline\chi_N*\psi)(n_3)\cdot (\overline\chi_N*\psi)(n_4) \ d \nu(\chi)\geq 0
$$
for every $n_1,n_2,n_3,n_4\in\Z_\tN $.
\end{lemma}
\begin{proof}
The  convolution product $\chi_N*\psi$ is defined on the
group $\ZN$ by the formula
$$
(\chi_N*\psi)(n)=\E_{k\in\ZN}\psi(n-k)\cdot\chi_N(k).
$$
It follows that for every $n\in[\wt N]$ there exists a sequence
$(a_n(k))_{k\in\ZN}$ of non-negative numbers that are  independent of $\chi$,
such that for every $\chi\in\CM$ we have
$$
(\chi_N*\psi)(n) =\sum_{k\in\ZN}a_n(k)\, \chi(k).
$$
The left hand side of the expression in the statement is thus equal to
\begin{multline*}
\sum_{k_1,k_2,k_3,k_4\in\ZN}\prod_{i=1}^4 a_{n_i}(k_i) \int_\CM
\chi(k_1)\cdot \chi(k_2)\cdot
 \overline\chi(k_3)\cdot \overline\chi(k_4) \, d \nu(\chi)=\\
 \sum_{k_1,k_2,k_3,k_4\in \ZN}\prod_{i=1}^4 a_{n_i}(k_i)
\int_\CM \chi(k_1k_2)\cdot
 \overline\chi(k_3k_4) \, d \nu(\chi)=\\
 \sum_{k_1,k_2,k_3,k_4\in \ZN}\prod_{i=1}^4 a_{n_i}(k_i)
 \int_X T_{k_1k_2}\one_A \cdot T_{k_3k_4}\one_A\,d\mu
\end{multline*}
where the last equality follows from equation \eqref{E:spectralo}.
This expression is non-negative since the function $\one_A$ is
non-negative, completing the proof.
\end{proof}

\subsection{Proof of Theorem~\ref{th:ergo2b}
assuming Theorem~\ref{th:strong-average-intro}}\label{SS:assuming}
  We start with a brief sketch of our
 proof strategy. Roughly speaking, Theorem~\ref{th:strong-average-intro}
enables us to decompose  the restriction of an arbitrary
multiplicative function on a finite interval  into three terms, a
close to  periodic term, a ``very uniform'' term, and an error term.
In the course of the proof of Theorem~\ref{th:ergo2b} we study these
three terms separately. The order of the different steps is
important as well as the precise properties of the decomposition.
 First, we show that the uniform term has a negligible contribution
in evaluating the averages in \eqref{eq:average-integral1b}. To do
this we use  the  uniformity estimates established in
Lemma~\ref{L:UnifromityEstimates2}. It is for this part of the proof
that it is very important to work with patterns that factor into
products of linear forms in two variables, otherwise we have no
 way of controlling the corresponding averages
 by Gowers uniformity norms. At this
point, the error term is shown to have negligible contribution, and
thus can be ignored.
   Lastly, the
structured term $\chi_s$  is dealt by restricting the variable $n$
to a suitable sub-progression where each function $\chi_s$ gives
approximately the same value to all four linear
forms;\footnote{\label{foot3}This  coincidence of values is very
important, not having it is a key technical obstruction that stops
us from handling equations like $x^2+y^2=n^2$. Restricting the range
of both variables $m$ and $n$ does not seem to help either, as this
creates problems with  handling the error term in the
decomposition.}
 it  then becomes possible to  establish the asserted positivity.
 In fact, the restriction to a
sub-progression step is  rather delicate, as it has to take place
before the component $\chi_e$ is eliminated (this explains also why
we do not restrict both variables $m$ and $n$ to a sub-progression), and in addition one
has to guarantee that the terms left out are non-negative, a
property that follows from
Lemma~\ref{L:HiddenPositivity}.\footnote{In a sense, our approach
follows the general principles of the circle method. Each
multiplicative function  is decomposed into two components, with
Fourier transform supported on major arc and minor arc frequences.
The contribution of the ``major arc component'' is further analyzed
to deduce the asserted positivity. The ``minor arc component'' is
shown to have negligible contribution, and this step is the hardest,
it is done using higher order Fourier analysis tools in the course
of proving Theorem~\ref{th:strong-average-intro}.}

 We  now enter the main body of the proof.
Recall that $\ell_1,\ell_2, \ell_3\in \N$ are fixed and distinct and that $\ell=\ell_1+\ell_2+\ell_3$. We stress also  that
in  this proof the quantities $m+\ell_in$ are computed in $\ZN$,
that is, modulo $\wt N$.

Let the action by dilations  $(T_n)_{n\in\N}$ on the probability space $(X,\CB,\mu)$,
 the set $A\in \CB$ with $\mu(A)>0$, and the spectral measure $\nu$ of $f=\one_A$,  be as in
Theorem~\ref{th:ergo2b}.
 We let
\begin{gather*}
\delta:=\mu(A)=\int f \ d\mu\ ;\\
\ve:=c_3\delta^2\ \ \text{ and }\ \
 F(x,y,z)\:=c_4^2\,
\frac{x^2y^2}{z^4},
\end{gather*}
where $c_3$ and $c_4$ are positive constants that will be specified
later, what is important is that  they depend only on $\ell$. Our goal is
for all large values of $N$ (how large will depend only on $\delta$) to bound from below the average
$$
 A(N):= \int\E_{m,n\in\ZN}\one_{[N]}(n)\cdot  \chi_N(m)\cdot \chi_N(m+\ell_1n) \cdot
\overline\chi_N(m+\ell_2n) \cdot \overline\chi_N(m+\ell_3n)\ d\nu(\chi).
$$

We start by applying the decomposition result of
Theorem~\ref{th:strong-average-intro}, taking as input the spectral
measure $\nu$, the number $\ve$, and the function $F$ defined
above. Let
$$
 Q:=Q(F,N, \ve, \nu)=Q(N, \delta, \nu),\quad  R:=R(F,N, \ve,  \nu)=R(N,\delta,\nu)
$$ be
the numbers provided by Theorem~\ref{th:strong-average-intro}. We
recall  that $Q$ and $R$ are bounded by a constant that depends only
on $\delta$. From this point on we assume that $N$ is sufficiently
large, depending only on $\delta$,  so that the conclusions of
Theorem~\ref{th:strong-average-intro} hold. To ease the notation a
bit, we omit the subscript $N$ when we use the functions
$\chi_{N,s},\chi_{N,u},\chi_{N,e}$  provided by
 Theorem~\ref{th:strong-average-intro},
and for
 $\chi \in \CM$, we write
$$
\chi_N(n)=\chi_s(n)+\chi_u(n)+\chi_\e(n), \quad n\in \ZN,
$$
for the decomposition that satisfies Properties~\eqref{it:decomU31}--\eqref{it:decomU34} of Theorem~\ref{th:strong-average-intro}.

Next, we  use the uniformity estimates of
Lemma~\ref{L:UnifromityEstimates2} in order to eliminate the uniform
component $\chi_u$ from the average $A(N)$. We let
$$
\chi_{s,e}=\chi_s+\chi_e
$$ and
 $$
A_1(N):=\int_\CM  \E_{m,n\in\ZN} \one_{[N]}(n)\cdot\chi_{s,e}(m)\cdot
\chi_{s,e}(m+\ell_1n) \cdot \overline\chi_{s,e}(m+\ell_2n) \cdot
\overline\chi_{s,e}(m+\ell_3n)\ d\nu(\chi). $$
 Using
Lemma~\ref{L:UnifromityEstimates2}, Property~\eqref{it:decomU31} of
Theorem~\ref{th:strong-average-intro}, and the estimates  $|\chi_N(n)|\leq 1$, $|\chi_{s,e}(n)|\leq 1$ for every $n\in\ZN$,
we get that
\begin{equation}\label{E:a1}
|A(N)-A_1(N)|\leq \frac{4\, c_2}{F(Q,R,\ve)^{\frac{1}{2}}}+\frac{8}{\tN}
\end{equation}
where $c_2$ is the constant provided by
Lemma~\ref{L:UnifromityEstimates2} and depends only on $\ell$.

 Next, we try to eliminate the error term $\chi_e$. But before doing this, it is
  important to first restrict the range of $n$ to a suitable sub-progression; the utility of
  this  maneuver will be clear on our next step when we estimate the contribution of the leftover term
  $\chi_s$. We stress that we cannot postpone this restriction on the range of $n$ until after
  the term $\chi_e$ is eliminated, if we did this  the contribution of the term $\chi_e$ would
   swamp the positive lower bound we get from the term $\chi_s$.
 We let
 \begin{equation}
\label{eq:def-epsilon}
\eta:=\frac{\ve}{QR}.
\end{equation}
 By Property~\eqref{it:decomU31} of Theorem~\ref{th:strong-average-intro},
Lemma~\ref{L:HiddenPositivity}  applies to $\chi_{s,e}$. Note that
the integers $Qk$, $1\leq k\leq \eta N$, are distinct elements of
the interval $[N]$.  It follows that
\begin{multline*}
\sum_{m,n\in\ZN}\int_\CM
\one_{[N]}(n)\cdot\chi_{s,e}(m)\cdot\chi_{s,e}(m+\ell_1n)\cdot
\overline{\chi}_{s,e}(m+\ell_2n)\cdot\overline{\chi}_{s,e}(m+\ell_3n)\,d\nu(\chi)\geq
\\
\sum_{m\in\ZN}\sum_{k=1}^{\lfloor \eta N\rfloor}
\int_\CM
\chi_{s,e}(m)\cdot\chi_{s,e}(m+\ell_1Qk)\cdot
\overline{\chi}_{s,e}(m+\ell_2Qk)\cdot\overline{\chi}_{s,e}(m+\ell_3Qk)\,d\nu(\chi).
\end{multline*}
Therefore, we have
\begin{equation}
\label{eq:A2} A_1(N)\geq \frac{\lfloor\eta N\rfloor}{\tN}A_2(N)\geq
\frac\eta{40\,\ell}A_2(N) = \ve\, \frac{1}{40\,\ell QR}\,A_2(N)
\end{equation}
where
\begin{multline*}
A_2(N):= \\
\int_\CM  \E_{m\in\ZN}\,\E_{k\in [\lfloor \eta N\rfloor]}\,
\chi_{s,e}(m)\cdot
\chi_{s,e}(m+\ell_1kQ) \cdot \overline\chi_{s,e}(m+\ell_2kQ) \cdot
\overline\chi_{s,e}(m+\ell_3kQ)\ d\nu(\chi).
\end{multline*}
We let
\begin{equation}
\label{eq:defA3}
A_3(N):= \int_\CM\E_{m\in\ZN}\E_{k\in[\lfloor \eta N\rfloor]}
\chi_{s}(m)\cdot\chi_s(m+\ell_1Qk)\cdot
\overline{\chi}_{s}(m+\ell_2Qk)\cdot\overline{\chi}_{s}(m+\ell_3Qk)\,d\nu(\chi).
\end{equation}
Since for every $n\in\ZN$ we have $|\chi_s(n)|\leq 1$, and since
$|\chi_{s,e}(n)|=|\chi_s(n)+\chi_e(n)|\leq 1$ by
Property~\eqref{it:decomU31} of
 Theorem~\ref{th:strong-average-intro}, we deduce that
\begin{equation}
\label{eq:A3}
|A_2(N)-A_3(N)|\leq 4\,\int_\CM\E_{m\in\ZN} |\chi_e(m)|\,d\nu(\chi)<4\ve
\end{equation}
where the last estimate follows by Part~\eqref{it:decomU34} of
Theorem~\ref{th:strong-average-intro}.

Next, we study the term $A_3(N)$.  We utilize Property~\eqref{it:decompU32} of
Theorem~\ref{th:strong-average-intro}, namely
$$
|\chi_s(n+Q)-\chi_s(n)|\leq \frac{R}{\tN} \quad \text{ for  every } \ n\in\ZN.
$$
We get for $m\in \ZN$, $1\leq k\leq\eta N$,  and for $i=1,2,3$, that
$$
|\chi_s(m+\ell_i Qk)-\chi_s(m)|\leq\ \ell_ik\,\frac {R}\tN\leq \ell
\eta N\,\frac {R}\tN \leq \frac \ve Q
$$
where the last estimate follows from \eqref{eq:def-epsilon} and the estimate $\tN\geq \ell N$.
Using this estimate in conjunction with the definition~\eqref{eq:defA3} of $A_3(N)$, we get
$$
A_3(N)\geq \int_\CM\E_{m\in\ZN}|\chi_{s}(m)|^4\,d\nu(\chi)
-\frac{3\ve}Q.
$$
  We
 denote by $\one$  the multiplicative function that is identically
 equal  to $1$.
 We claim that
$\nu(\{\one\})\geq\delta^2$. Indeed,  if $(\Phi_N)_{N\in \N}$ is a
multiplicative F\o lner sequence in $\N$ we have
$$
\nu(\{{\bf 1}\})=\lim_{N\to\infty} \int_\CM
\Big|\frac{1}{|\Phi_N|}\sum_{n\in \Phi_N} \chi(n)\Big|^2 \
d\nu(\chi)= \lim_{N\to\infty} \int
\Big|\frac{1}{|\Phi_N|}\sum_{n\in \Phi_N} T_nf \Big|^2\ d\mu,
$$  and this is greater or equal than
$$ \lim_{N\to\infty}
\Big|\int \frac{1}{|\Phi_N|}\sum_{n\in \Phi_N} T_nf \ d\mu \Big|^2=
\Big|\int f \ d\mu\Big|^2=\delta^2,$$ proving our claim. Using this
we deduce that
$$
\int_\CM\E_{m\in\ZN}|\chi_{s}(m)|^4\,d\nu(\chi)\geq
\nu(\{\one\}) \cdot\E_{m\in\ZN} |\one_s(m)|^4
\geq \delta^2\, \bigl|\E_{m\in\ZN} \one_s(m)|^4.
$$
Since $\one_s=\one_N*\psi$ for some kernel $\psi$  on $\ZN$ we have
$$\E_{m\in\ZN}\one_s(m)=
\E_{m\in\ZN}\E_{k\in\ZN}\one_N(k)\psi(m-k)=
\E_{k\in\ZN}\one_N(k) =
\frac N{\tN}\geq \frac 1{20\,\ell}.
$$
Combining the above we get
\begin{equation}
\label{eq:A3b} A_3(N)\geq  \frac
{\delta^2}{20^4\,\ell^4}-\frac{3\ve }Q.
\end{equation}

Putting~\eqref{E:a1},  \eqref{eq:A2}, \eqref{eq:A3}, and
\eqref{eq:A3b} together,   we get
$$A(N)
\geq\ve\, \frac{1}{40\,\ell QR}\,\Bigl(\frac
{\delta^2}{20^4\,\ell^4}-7\ve \Bigr)
-\frac{4c_2}{F(Q,R,\ve)^{\frac{1}{2}}}-\frac{8}{\tN}.
$$
Recall that $\ve=c_3\delta^2$, for some positive constant $c_3$ that we
left unspecified until now.  We choose $c_3<1$, depending only on
$\ell$, so that
$$
\frac 1{40\,\ell} \Bigl(\frac
{\delta^2}{20^4\,\ell^4}-7\ve\Bigr) \geq c_5 \delta^2
$$
for some positive constant $c_5$ that depends only on
$\ell$. Then we have
$$
A(N)\geq\delta^2\,\frac{c_5 \ve }{QR}-\frac{4c_2}{F(Q,R,\ve)^{\frac{1}{2}}}.
$$
Recall that $$ F(Q,R,\ve)= c_4^2\frac{Q^2R^2}{\ve^4}
$$
where $c_4$ was not determined until this point. We choose
$$
c_4:=\frac{8c_2 c_3}{c_5 }
$$
and upon recalling that $\ve=c_3 \delta^2$    we get
$$
A(N)\geq\delta^2\,\frac{c_5\ve }{QR}-c_2\, \frac{4\ve^2}{c_4 QR}=\frac {c_5\delta^2\ve}{2\,QR}=
\frac {c_3 c_5 \delta^4}{2\,QR}>0.
$$
Recall that $Q$ and $R$ are bounded by a constant that depends only
on $\delta$. Hence, $A(N)$ is greater than a positive constant that
depends only on $\delta$, and in particular is independent of $N$,
provided that $N$ is sufficiently large, depending only on $\delta$,
as indicated above. This completes the proof of
Theorem~\ref{th:ergo2b}. \qed


\section{Fourier analysis of multiplicative functions}\label{S:U^2}
In this section  we study the Fourier coefficients of multiplicative
functions. Our   goal is to establish  a decomposition
$$
\chi_N=\chi_{N,s}+\chi_{N,u}
$$
 similar to the one given in Theorem~\ref{th:Decomposition-weakU3-intro},
 but with the $U^2$-norm in place of the $U^3$-norm. We will then
 use this result in Section~\ref{S:U^3} as our starting point in the
 proof of the decomposition result for the $U^3$-norm.

 \begin{convention}In this section the functions $\chi_N$ are defined on $\ZN$.
 In particular, all  convolution products are defined   on $\ZN$ and
 the Fourier coefficients of $\chi_N$ are given by
$$
\widehat{\chi_N}(\xi):=\E_{n\in\ZN}\chi_N(n)\, \e(-n \xi/\wt N)
\quad \text{ for } \ \xi\in\ZN.
$$
\end{convention}

\begin{theorem}[Weak uniform decomposition for the $U^2$-norm]
\label{th:Decomposition-weakU2-intro} For every $\theta>0$ there
exist  positive integers $Q:=Q(\theta)$ and  $R:=R(\theta)$, and for
every sufficiently large $N$,  depending only on
 $\theta$,  there exists a kernel $\phi_{N,\theta}$ on $\ZN$
with the following properties:

For   every $\chi\in\CM$, writing
$$
\chi_{N,s}=\chi_N*\phi_{N,\theta}\ \ \text{ and }\ \
\chi_{N,u}=\chi_N-\chi_{N,s},
$$
we have
\begin{enumerate}
\item
\label{it:decomU22} \vide $\displaystyle
|\chi_{N,s}(n+Q)-\chi_{N,s}(n)|\leq \frac{R}{\tN}$ for every $n\in\ZN$,
where  $n+Q$ is taken \!\!\!$\mod\tN$;
\item
\label{it:decomU23}
 $\vide\displaystyle\norm{\chi_{N,u}}_{U^2(\ZN)}\leq\theta$.
\end{enumerate}
Moreover, for all $\theta$ and $\theta'$, and every $N$ such that
$\phi_{N,\theta}$ and $\phi_{N,\theta'}$ are defined, we have
\begin{equation}
\label{eq:phi-increases} \text{if }0<\theta'\leq\theta,\ \ \text{
then }\ \  \widehat{\phi_{N,\theta'}}(\xi)\geq
\widehat{\phi_{N,\theta}}(\xi)\geq  0\  \ \text{ for every }\ \
\xi\in\ZN.
\end{equation}
\end{theorem}
The monotonicity property~\eqref{eq:phi-increases} plays a central
role in the derivation of Theorem~\ref{th:strong-average-intro} from
Theorem~\ref{th:Decomposition-weakU3-intro} in
Section~\ref{subsec:proof_strong}. This  is  one of the reasons why
we construct the kernels $\phi_{N,\theta}$  explicitly in
Section~\ref{subsec:kernels}.

The
 values of  $Q$ and $R$ given by  Theorem~\ref{th:Decomposition-weakU2-intro}
  will be  used later in Section~\ref{S:U^3}, and  they do not coincide
  (in fact, they are much smaller) with the values of $Q$ and $R$ in
   Theorems~\ref{th:Decomposition-weakU3-intro} and  \ref{th:strong-average-intro}.

\subsection{K\'atai's orthogonality criterion}
\label{subsec:katai} We start with the key number theoretic input
that we need in this section and which will also be used later in Section~\ref{S:CorrelationNil}.
\begin{lemma}[Orthogonality criterion~\cite{K86}]
\label{lem:katai} For every $\ve>0$ 
there exists
$\delta:=\delta(\ve)>0$ and $K:=K(\ve)$ such that the
following holds: If $N\geq K$  and $f\colon [N]\to \C$ is a function with
$|f|\leq 1$, and
$$
\max_{\substack{p,p'\text{ \rm primes}\\1<p<p'<K}}\bigl|\E_{n\in[\lfloor
N/p'\rfloor]} f(pn)\overline{f}(p'n)\bigr|< \delta,
$$
 then
$$
\sup_{\chi\in\CM}\bigl|\E_{n\in[N]}\chi(n)f(n)\bigr|<\ve.
$$
\end{lemma}
The dependence of $\delta$ and $K$ on $\ve$ can be made
explicit (for good bounds see \cite{BSZ12}) but we do not need such extra information
here.

Lemma~\ref{lem:katai} is not strictly speaking contained in the
paper~\cite{K86} that shows only  asymptotic results. Moreover,
K\'atai considers only the case of a function of modulus $1$,
written as $\e(t(n))$, but the estimates are valid without any
change for functions of modulus at most $1$. For
completeness we give the derivation.

\begin{proof}
 Let  $\ve>0$ be given.
In \cite{K86}, the  letter $f$ is used to denote a multiplicative
function that we denote here by  $\chi$. Let $K$  be a
positive integer that will be made explicit below and will depend
only on $\ve$. We write $\CP$ for  the set of primes $p$ with
$p<K$ and let
$$
\CA_\CP:=\sum_{p\in\CP}\frac 1p.
$$
Let $N\in \N$,  $\chi\in\CM$,  and $f:[N]\to \C$ be a function with $|f|\leq 1$. We let
$$
S(N):=\sum_{n\in[N]} \chi(n)f(n).
$$
After correcting some typos in \cite{K86},\footnote{The sign $+$ and
the universal constants  are missing on the right hand side.}
 inequality~(3.8) of~\cite{K86} reads as follows
\begin{equation}
\label{eq:katai1} \frac{|S(N)|^2\CA_\CP^2}{N^2}\leq c+ c'\cdot
\CA_\CP+ c''\sum_{p\neq p'\in\CP}\frac 1N \Big|\sum_{n\leq
\min(N/p,N/p')}f(pn)\,\overline{f}(p'n)\Big|
\end{equation}
for some positive universal constants $c,c',c''$. We choose $K$ sufficiently large  so that
$c\CA_\CP^{-2}+c'\CA_\CP^{-1}\leq \ve^2/2$, and then choose
$\delta:= \ve^2\CA_\CP^2/(2c'' |\CP|^2)$. Thus defined, note that
$K$ and $\delta$ depend on $\ve$ only.
 Assuming that the function $f$
satisfies the hypothesis of the lemma, and inserting the previous bounds in
\eqref{eq:katai1} we get the desired estimate.
 \end{proof}

\subsection{An application to Fourier coefficients of multiplicative functions}
Next, we use the orthogonality criterion of K\'atai to prove that
the Fourier coefficients of the restriction of a multiplicative
function on an interval $[N]$ are small unless the frequency is
close to a rational with small denominator. Furthermore, the
implicit constants do not depend on the multiplicative function or
the integer $N$.

\begin{corollary}[$U^2$ non-uniformity]
\label{cor:katai} For every $\theta>0$ there exist positive integers
$Q:=Q(\theta)$  and $V:=V(\theta)$ such that, for every sufficiently
large $N$, depending only on $\theta$, for every $\chi\in\CM$, and
every $\xi\in\ZN$, we have the following implication
\begin{equation}
\label{eq:Fourier_chi} \text{if }\
|\widehat{\chi_N}(\xi)|\geq\theta,\quad  \text{ then }\quad
\Bigl\Vert\frac{Q\xi}{\wt N}\Bigr\Vert\leq \frac{QV}{\wt N}.
\end{equation}
\end{corollary}
\begin{proof}
Let  $\delta:=\delta(\theta)$ and  $K:=K(\theta)$ be defined  by
Lemma~\ref{lem:katai} and let $Q=K!$. Suppose that $N>K$. Let  $p$
and $p'$ be primes with $p<p'\leq K$ and  let $\xi\in\ZN$. Since $Q$
is a multiple of $p'-p$ we have
$$
\norm{Q\xi/\wt N} \leq \frac{Q}{p'-p} \norm{(p'-p)\xi/ \tN} \leq
Q\,\norm{(p'-p)\xi/\wt N}.
$$
Since $\wt{N}<20\,\ell N$, we deduce
$$
|\E_{n\in[\lfloor N/p'\rfloor]}\,\e(p'n\xi/ \wt{N}) \e(-pn\xi
\wt{N})| \leq \frac {2p'} {N\norm {(p'-p)\xi/\wt{N}}} \leq\frac
{2KQ}{N\norm{Q\xi/\wt N}} \leq\frac{40\,\ell KQ}{\wt N\norm{Q\xi/\wt
N}}.
$$
Let $V=40\,\ell K/\delta$. If  $\norm{Q\xi/\wt{N}} > QV/ \wt{N}$,
then  the rightmost  term of the last inequality   is smaller than
$\delta$, and thus,
 by Lemma~\ref{lem:katai}  we have
$$
|\widehat{\chi_N}(\xi)|=
\bigl|\E_{n\in[\wt{N}]}\,\chi_N(n)\e(-n\xi/\wt{N})\bigr| =
\frac{N}{\wt{N}}\bigl|\E_{n\in
[N]}\chi(n)\e(-n\xi/\wt{N})\bigr|<\theta
$$
contradicting \eqref{eq:Fourier_chi}. Hence, $\norm{ Q\xi /\wt
N}\leq QV/\wt N$, completing the proof.
\end{proof}

\subsection{Some kernels}
\label{subsec:kernels}
  Next, we make some explicit choices for the constants $Q$ and $V$
   of Corollary~\ref{cor:katai}. This will enable us to compare the Fourier transforms of the
   kernels $\phi_{N,\theta}$ defined below for different values of $\theta$ and to establish the monotonicity property~\eqref{eq:phi-increases}.

For every $\theta>0$ we define
\begin{gather}
\notag \CA(N,\theta):=\Bigl\{\xi\in\ZN\colon \exists\chi\in\CM
\ \text{ such that } \ |\widehat{\chi_N}(\xi)|\geq\theta^2\Bigr\}\ ;\\
\notag W(N,q,\theta):= \max_{\xi\in \CA(N,\theta)} \wt{N}\Bigl\Vert q\,\frac{\xi}{\wt{N}}\Bigr\Vert\ ;\\
\label{def:Qtheta} Q(\theta):=\min_{k\in\N}\Bigl\{k!\colon\ \limsup_{N\to+\infty} W(N,k!,\theta)<+\infty\Bigr\}\ ;\\
\label{def:D(theta)}V(\theta):=1+\Bigl\lceil\frac
1{Q(\theta)}\,\limsup_{N\to+\infty}  W(N,Q(\theta),\theta)
\Bigr\rceil.
\end{gather}
It follows from  Corollary~\ref{cor:katai} that the set of integers
used in the definition of $Q(\theta)$ is non-empty, hence
$Q(\theta)$ is well defined. The essence of the preceding
definitions is that for every real number $V'>V(\theta)-1$, the
implication~\eqref{eq:Fourier_chi} is valid with $\theta^2$
substituted for $\theta$, $V'$ for $V$, $Q(\theta)$ for $Q$, and for
every sufficiently large $N$.
Furthermore, it follows from these definitions that for $0<\theta'\leq\theta$, we
have $Q(\theta')\geq Q(\theta)$,  and thus
\begin{equation}
\label{eq:Q-multliple} \text{for }\ 0<\theta'\leq\theta,\ \text{ the
integer }\  Q(\theta')\ \text{ is a multiple of }\ Q(\theta).
\end{equation}
Moreover, it can be checked that
\begin{equation}
\label{eq:c-devreases} V(\theta)\   \text{ increases as } \ \theta \ \text{ decreases}.
\end{equation}

Next, we use the constants just defined  to build the
kernels $\phi_{N,\theta}$ of
Theorem~\ref{th:Decomposition-weakU2-intro}. We recall that a kernel
on  $\Z_{\wt{N}}$ is a non-negative function $\phi$ on $\Z_{\wt{N}}$
with $\E_{n\in\Z_{\wt{N}}}\phi(n)=1$. We define the \emph{spectrum} of a function $\phi$
to be  the set
$$
\spec(\phi):=\bigl\{\xi\in\ZN\colon \widehat\phi(\xi)\neq 0\bigr\}.
$$

For every $m\geq 1$ and $\tN>2m$ the ``Fejer kernel'' $f_{N,m}$ on
$\Z_\tN$ is defined by
$$
f_{N,m}(x)=\sum_{-m\leq\xi\leq
m}\bigl(1-\frac{|\xi|}m\bigr)\,\e\bigl(x\,\frac{\xi}{\tN}\bigr).
$$
The spectrum  of $f_{N,m}$ is the subset $(-m,m)$ of $\ZN$. Let
$Q_N(\theta)^*$ be the inverse of $Q(\theta)$ in $\Z_\tN$, that is,
the unique integer in $\{1,\dots,\tN-1\}$ such that
$Q(\theta)Q_N(\theta)^*=1\bmod \tN$.  For $N>
2Q(\theta)V(\theta)\lceil \theta^{-2}\rceil$ we define
\begin{equation}
\label{eq:def-phi} \phi_{N,\theta}(x)=f_{N,
Q(\theta)V(\theta)\lceil\theta^{-4}\rceil}(Q_N(\theta)^*x).
\end{equation}
An equivalent formulation is that $f_{N,
Q(\theta)V(\theta)\lceil\theta^{-4}\rceil}(x)=\phi_{N,\theta}(Q(\theta)x)$.
The
spectrum  of the kernel $\phi_{N,\theta}$ is the set
\begin{equation}
\label{eq:def-Xi} \Xi_{N,\theta}:=\Big\{\xi\in \ZN\colon
\Bigl\Vert\frac{Q(\theta) \xi}\tN\Bigr\Vert<
\frac{Q(\theta)V(\theta)\lceil\theta^{-4}\rceil}\tN\Big\},
\end{equation}
and we have
\begin{equation}
\label{eq:fourier-phi} \widehat{\phi_{N,\theta}}(\xi) =\begin{cases}
\displaystyle 1-\Bigl\Vert \frac{Q(\theta)\xi}\tN\Bigr\Vert\, \frac
\tN{ Q(\theta)V(\theta)\lceil\theta^{-4}\rceil}
&\ \  \text{if }\  \xi\in\Xi_{N,\theta} \  ;\\
0 & \ \ \text{otherwise.}
\end{cases}
\end{equation}
We remark that the cardinality of  $\Xi_{N,\theta}$ depends only on $\theta$.
\subsection{Proof of Theorem~\ref{th:Decomposition-weakU2-intro}}
First, we claim that  property~\eqref{eq:phi-increases} of
Theorem~\ref{th:Decomposition-weakU2-intro} holds. Indeed, suppose
that $\theta\geq\theta'>0$ and that $N$ is sufficiently large so
that $\phi_{N,\theta'}$ and $\phi_{N,\theta}$ are defined. We have
to show that $\widehat{\phi_{N,\theta'}}(\xi)\geq
\widehat{\phi_{N,\theta}}(\xi)$ for every $\xi$. Using
\eqref{eq:Q-multliple} and  \eqref{eq:c-devreases} we get that
$\Xi_{N,\theta'}$ contains the set $\Xi_{N,\theta}$. Thus,  we can
assume that $\xi$ belongs to the latter set as the estimate is
obvious otherwise. In this case, the claim follows
from~\eqref{eq:Q-multliple}, \eqref{eq:c-devreases}, and the
formula~\eqref{eq:fourier-phi} giving the Fourier coefficients of
$\phi_{N,\theta}$.

Next, we show the remaining assertions~\eqref{it:decomU22}
and~\eqref{it:decomU23} of
Theorem~\ref{th:Decomposition-weakU2-intro} for the decomposition
$$
\chi_{N,s}:=\phi_{N,\theta}*\chi_N, \quad
\chi_{N,u}:=\chi_N-\phi_{N,\theta}*\chi_N.
$$
Let  $\theta>0$, assume that $N$ is sufficiently large  depending
only on  $\theta$, and let $Q=Q(\theta)$, $V=V(\theta)$,
$\phi_{N,\theta}$, $\Xi=\Xi_{N,\theta}$,  be defined by
\eqref{def:Qtheta}, \eqref{def:D(theta)}, \eqref{eq:def-phi},
\eqref{eq:def-Xi} respectively.

For every $\chi\in\CM$,  if $|\widehat{\chi_N}(\xi)|\geq\theta^2$,
then by the definition of $Q$ we have $\norm{Q\xi/\tN}\leq QV/\tN$
and thus
 $\widehat {\phi_{N,\theta}}(\xi)\geq 1-\theta^4$ by~\eqref{eq:fourier-phi}.
 It follows that
 $|\widehat{\chi_N}(\xi)-\widehat{\phi_{N,\theta}*\chi_N}(\xi)|$ $\leq
 \theta^4\leq\theta^2$.
 This last bound is clearly also true when $|\widehat{\chi_N}(\xi)| <\theta^2$ and thus
using identity \eqref{E:U2formula}  we get
\begin{multline*}
\norm{\chi_N-\phi_{N,\theta}*\chi_N}_{U^2(\Z_\tN)}^4
=\sum_{\xi\in\ZN}
  |\widehat{\chi_N}(\xi)-\widehat{\phi_{N,\theta}*\chi_N}
    (\xi)|^4\leq \\
  \theta^4
\sum_{\xi\in\ZN}|\widehat{\chi_N}(\xi)-\widehat{\phi_{N,\theta}
   *\chi_N}(\xi)|^2
\leq  \theta^4 \sum_{\xi\in\ZN}|\widehat{\chi_N}(\xi)|^2\leq
   \theta^4,
\end{multline*}
where the last estimate follows from Parseval's identity. Hence,
$\norm{\chi_N-\phi_{N,\theta}*\chi_N}_{U^2(\Z_\tN)}\leq\theta$,
proving Property~\eqref{it:decomU23}.

Lastly, for
$\chi\in\CM$ and $n\in\Z_\tN$, using the Fourier inversion formula
and the estimate $|\e(x)-1|\leq \norm{x}$, we get
$$
|(\phi_{N,\theta}*\chi_N)(n+Q)-(\phi_{N,\theta}*\chi_N)(n)|\leq
2\sum_{\xi\in\ZN}|\widehat{\phi_{N,\theta}}(\xi)|\cdot\Bigl\Vert
Q\frac\xi \tN\Bigr\Vert\leq 2 \frac{|\Xi_{N,\theta}|
QV\lceil\theta^{-4}\rceil}{\tN},
$$
where the last estimate follows from \eqref{eq:def-Xi}. As  remarked
above, $|\Xi_{N,\theta}|$ depends only on $\theta$,  thus the last
term in this inequality is bounded by $R/\wt N$ for some constant
$R$ that depends only on $\theta$. This establishes
Property~\eqref{it:decomU22} of
Theorem~\ref{th:Decomposition-weakU2-intro} and finishes the proof.
\qed

\section{Modifications of the inverse and factorization theorems}
\label{S:InverseFactorization} In this section, we state and prove
some consequences of an inverse theorem  of Green and Tao
\cite{GT08}
 and a
factorization theorem by the same authors~\cite{GT12a} that  are
particularly tailored to the applications being pursued   in
subsequent sections. These results combined, prove that a function
that has
 $U^3$-norm bounded away from zero, either has  $U^2$-norm bounded away from zero, or else
 correlates in a sub-progression with a totally equidistributed $2$-step polynomial nilsequence
 of a very special form.

Essentially all  definitions and results of this section extend without important changes to
 arbitrary nilmanifolds. We restrict to the case  of $2$-step nilmanifolds
 as these are the only ones needed in this article and  the notation is somewhat lighter in this case.

\subsection{Nilmanifolds}
\label{subsec:nilmanifolds}
We introduce some notions from \cite{GT12a}. We record here only the properties that
 we  need in this section and defer supplementary material to the next section and to Appendix~\ref{ap:A}.

Let $X=G/\Gamma$ be a
$2$-step nilmanifold.  Throughout, we assume  that $G$ is a connected and simply
 connected $2$-step nilpotent Lie group and $\Gamma$ is a discrete co-compact subgroup
of $G$. We view elements of $G/\Gamma$ as ``points'' on the
nilmanifold $X$ rather than equivalence classes, and denote  them by
$x,y,$ etc. The nilmanifold $X$ is endowed with a base point $e_X$
which is the projection on $X$ of the unit element of $G$.
 The action of $G$ on $X$ is denoted by
$(g,x)\mapsto g\cdot x$. The \emph{Haar measure} $m_X$ of $X$ is the unique probability measure on $X$ that is invariant under this action.

  We denote by  $m$  the dimension of $G$ and by $r$  the dimension of $G_2:=[G,G]$, the commutator subgroup of $G$.
We implicitly assume that $G$ is endowed with a \emph{strong Mal'cev basis} $\CX$ that is adapted to the natural filtration (see Definition~2.1 in \cite{GT12a}. We record here the properties of $\CX$ that will
be needed:    $\CX$ is  a basis
$(\xi_1,\dots,\xi_m)$ of the Lie algebra $\mathfrak g$ that satisfies
\begin{enumerate}
\item
The map
$$
\psi\colon (t_1,\dots,t_m)\mapsto\exp(t_1\xi_1)\cdot\ldots\cdot\exp(t_m\xi_m)
$$
is a homeomorphism from $\R^m$ onto $G$;
\item
$G_2=\psi\bigl(\{0\}^{m-r}\times \R^r\bigr)$;
\item
\label{it:Malcev-Gamma}
$\Gamma=\psi(\Z^m)$.
\end{enumerate}
In \cite{GT12a} the authors introduce a parameter called the degree of rationality of $\CX$ and seek to obtain results where all implied constants depend polynomially on this parameter. We have no need to do this and so this rationality parameter is not going to appear in our article.

Let $\mathfrak g$ be endowed with the Euclidean structure making  $\CX$  an orthonormal basis.
 This induces a  Riemannian structure on $G$ that is invariant under right translations. The group  $G$ is endowed with the associated geodesic distance, which we denote by  $d_G$. This distance is invariant under right translations. We remark that in \cite{GT12a} the authors use a different metric, but it is equivalent with $d_G$, and the implied constant depends only on $X$ and the choice of the Mal'cev basis,  so this does not make any difference for us.

Let the space $X=G/\Gamma$ be endowed with the quotient metric
$d_X$. Writing $p\colon G\to X$ for the quotient map, the metric
$d_X$ is defined by
$$
d_X(x,y)=\inf_{g,h\in G}\{d_G(g,h)\colon p(g)=x,\ p(h)=y\}.
$$
Since  $\Gamma$ is discrete it follows that the infimum
is attained. More precisely, there exists a compact subset $F_0$ of
$G$, such that for every   $x,x'\in X$
\begin{equation}
\label{eq:F} \text{ there exist }\ h,h'\in
F_0\ \text{ with }\ d_G(h,h')=d_X(x,x') \text{ and } p(h)=x, p(h')=x'.
\end{equation}
We frequently  use
the fact that if $f$ is a smooth function on $X$, then $\norm
f_{\lip(X)}\leq\norm f_{\CC^1(X)}$. We also use  the following simple fact:
\begin{lemma}
\label{lemp:ap1} For every bounded subset $F$ of $G$ there exists
a constant $H>0$ such that
\begin{enumerate}
\item
$d_X(g\cdot x,g\cdot x')\leq Hd_X(x,x')$ for all $x,x'\in X$ and
every $g\in F$;
\item
  for   every  $f\in\CC^m(X)$ and every $g\in F$,
writing $f_g(x):=f(g\cdot x)$, we have $\norm{f_g}_{\CC^m(X)}\leq
H\norm f_{\CC^m(X)}$.
\end{enumerate}
\end{lemma}
\begin{proof}
Let $F_0$ be as in~\eqref{eq:F}. Since the multiplication $G\times
G\to G$ is smooth, its restriction to any compact set is Lipschitz
and thus there exists  a constant $H>0$ with $d_G(gh,gh')\leq
Hd_G(h,h')$ for every $g\in F$ and  $h,h'\in F_0$. The first
statement now follows immediately from~\eqref{eq:F}. Since
 the map $(g,x)\mapsto g\cdot x$, from $
K\times X$ to $X$, is smooth, the second statement follows as well.
\end{proof}

\begin{definition}[Vertical torus]
  We keep the same notation as above.
 The \emph{vertical torus} is the
sub-nilmanifold  $G_2/(G_2\cap\Gamma)$. The Mal'cev basis induces an
isometric identification between $G_2$ and $\R^r$, and thus of the
vertical torus endowed with the quotient metric, with $\T^r$ endowed
with its usual metric. Every $\bk\in\Z^r$ induces a character
$\bu\mapsto \bk\cdot \bu$ of the vertical torus. A function $F$ on
$X$ is a    \emph{nilcharacter with frequency $\bk$} if $F(\bu\cdot
x)=\e(\bk\cdot\bu)F(x)$ for every $\bu\in \T^r=G_2/(G_2\cap \Gamma)$ and
every $x\in X$. The nilcharacter is \emph{non-trivial} if its
frequency is non-zero.
\end{definition}

\begin{definition}[Maximal torus and horizontal characters]
Let $X=G/\Gamma$ be a $2$-step nilmanifold,  let $m$ and $r$ be as
above, and let $s:=m-r$.
 The Mal'cev basis induces an
isometric identification between the \emph{maximal torus}
$G/(G_2\Gamma)$, endowed with the quotient metric, and $\T^s$,
endowed with its usual metric. A \emph{horizontal character} is a
continuous group homomorphism $\eta\colon G\to\T$  with  a trivial
restriction on $\Gamma$.  In Mal'cev coordinates, it is given by $\eta(x_1,\dots, x_m)=k_1x_1+\dots+k_sx_s\bmod 1$ for $(x_1,\dots,x_m)\in \R^m$, where $k_1,\dots,k_s$ are integers called the coefficients of $\eta$.   We write $\norm{\eta}:=|k_1|+\cdots+|k_s|$.
The horizontal character $\eta$ factors   through the maximal torus, and  induces a character of this group,  given by  $\balpha\mapsto\bk\cdot\balpha:=k_1\alpha_1+\dots+k_s\alpha_s$ for
$\balpha=(\alpha_1,\dots,\alpha_s)\in\T^s$.

\end{definition}
 Note that a trivial  nilcharacter
is any function that factors through the maximal torus.

\subsection{A corollary of the inverse theorem}
We start by stating the  inverse theorem of Green and Tao for the $U^3$-norms.
\begin{theorem}[{The inverse theorem for the $U^3(\Z_N)$-norm~\cite[Theorem 12.8]{GT08}}]
\label{th:inverse}\ For every  $\ve>0$ there exist
$\delta:=\delta(\ve)>0$ and a $2$-step nilmanifold $X:=X(\ve)$, such
that for every  sufficiently large $N$, depending only on $\ve$, and
for every $f\colon\Z_N\to\C$ with $|f|\leq 1$ and $\norm
f_{U^3(\Z_N)}\geq\ve$, there exists a function $\Phi\colon X\to\C$
with $\norm\Phi_{\lip(X)}\leq 1$ and an element $g\in G$ such that
$|\E_{n\in[N]}f(n)\Phi(g^n\cdot e_X)|\geq\delta$.
\end{theorem}
A sequence $\Phi(g^n\cdot e_X)$ of  this form where $\Phi$ is only
assumed to be continuous is defined as a \emph{$2$-step basic nilsequence} in
\cite{BHK05}; if in addition we assume  that $\Phi$ is Lipschitz, then
we call it a
\emph{nilsequence of bounded complexity} a notion first used in \cite{GT08}.

 We state a
corollary of this result that is better suited for our purposes:
\begin{corollary}[Modified $U^3$-inverse theorem]
\label{cor:inverse}
For every $\ve>0$ there exist
$\delta:=\delta(\ve)>0$, $m:=m(\ve)$,
and a finite family $\CH:=\CH(\ve)$ of
 $2$-step nilmanifolds,
of dimension at most $m$  and having a vertical torus of dimension $1$
(identified with $\T$ as explained above),
such that:
 For
every   sufficiently large $N$, depending only on $\ve$, if
$f\colon\Z_N \to\C$ is a function with $|f|\leq 1$ and $\norm
f_{U^3(\Z_N)}\geq\ve$, then at least one of the following conditions
hold
\begin{enumerate}
\item
\label{it:inverse1}
 $\norm f_{U^2(\Z_N)}\geq\delta$;
\item
\label{it:inverse2} there exist a nilmanifold  $X$ belonging to the
family $\CH$, an element $g\in G$, and a nilcharacter $\Psi$ of $X$
with frequency $1$,
such that
\begin{gather*}
\norm\Psi_{\CC^{2m}(X)}\leq 1 \ \text{ and } \
|\E_{n\in[N]}f(n)\Psi(g^n\cdot e_{X})|\geq\delta.
\end{gather*}
\end{enumerate}
\end{corollary}
\begin{proof}
Let $\ve>0$. In this proof  the constants
$\delta,\delta',\delta'',C,C_1,\dots,$ depend only on $\ve$.

 Let $\wt X=\wt G/\wt \Gamma$ and  $\delta$ be given by
Theorem~\ref{th:inverse}. Let $m$ be the dimension of $\wt X$, and
$r$ be the dimension of  $\wt G_2$. The maximal torus $\wt G/(\wt G_2\wt\Gamma)$ has
dimension $s:=m-r$. As mentioned above, we identify $\wt G_2$ with
$\R^r$  and  $\wt G_2\cap\wt \Gamma$ with $\Z^r$,  thus the vertical
torus $\wt G_2/(\wt G_2\cap\Gamma)$ is identified with $\T^r$.

Let   $f\colon\Z_N\to\C$ be a function with $|f|\leq 1$ and $\norm
f_{U^3}\geq\ve$. Let $\Phi$ and $\wt g$ be given by
Theorem~\ref{th:inverse}, i.e.
\begin{equation}
\label{eq:Phi_Inverse}
|\E_{n\in[N]}f(n)\Phi(\wt g^n\cdot e_{X})|\geq\delta
\end{equation}
and $\norm\Phi_{\lip(\wt X)}\leq 1$.
We can assume that $
\norm\Phi_{\CC^{2m}(\wt X)}\leq 1.$ Indeed, there exists a function
$\Phi'$ with $\norm{\Phi-\Phi'}_\infty\leq\delta/2$ and
$\norm{\Phi'}_{\CC^{2m}(\wt X)}\leq C$ for some constant $C$ depending only on $\delta$ and $\wt X$ and thus only on $\ve$. Up to a change in
the constant $\delta$, the conclusion of Theorem~\ref{th:inverse}
remains valid with $\Phi'$ substituted for $\Phi$.

 We need some preliminary definitions. For $\bk\in \Z^r$, the character $\bk$  of $\wt G_2/(\wt G_2\cap\wt \Gamma)$
  induces a character of $\wt G_2$ given by
  some linear map $\phi_\bk\colon \wt G_2=\R^r\to\R$.
Let $G_\bk$ be the quotient of $\wt G$ by the  subgroup
$\ker(\phi_\bk)$ of $\wt G_2$ and $\Gamma_\bk$ be the image of $\wt
\Gamma$ in this quotient. Then $\Gamma_\bk$ is a discrete and
co-compact subgroup of  $G_\bk$ and we let
$X_\bk:=G_\bk/\Gamma_\bk$. We write $\pi_\bk\colon \wt X\to X_\bk$
for the natural projection and let $e_{X_\bk}:=\pi_\bk(e_X)$.

 If $\bk$ is non-zero, then  $X_\bk$ is a non-Abelian
$2$-step nilmanifold. The commutator $G_{\bk,2}$ of $G_{\bk}$ is the
quotient of $\wt G_2$ with the kernel of $\phi_\bk$ and thus has
dimension $1$ and  the vertical torus $G_\bk/(G_{\bk,2}\Gamma_\bk)$
of $X_\bk$ has dimension $1$.
 If $\bk$ is the trivial character, then $X_\bk$ is
the maximal torus $G/(G_2\Gamma)$ of $X$ and thus is a compact
Abelian Lie group.

We recall the definition of the vertical Fourier transform.  The restriction to $\wt
G_2\cap\wt\Gamma$ of the action by translation of $\wt G$ on
$\wt X$ is trivial, and thus this action induces an action of the
vertical torus on  $\wt X$ by $(\bu,x)\mapsto \bu\cdot x$ for
$\bu\in\T^r$ and $x\in \wt X$. The vertical Fourier series of the function
$\Phi$ is
$$
\Phi=\sum_{\bk\in\Z^r}\Phi_\bk \quad \text{ where } \quad
\Phi_\bk(x)=\int_{\T^r}\Phi(\bu\cdot
x)\e(-\bk\cdot\bu)\,dm_{\T^r}(\bu).
$$
For every $\bk\in \Z^r$, the function $\Phi_\bk$ is a nilcharacter
with frequency $\bk$  and thus can be written as
$$
\Phi_\bk=\Psi_\bk\circ\pi_\bk
$$
for some function $\Psi_\bk$ on $X_\bk$.
 If $\bk\neq 0$, then  $\Phi_\bk$ is a nilcharacter of $X_\bk$ with frequency equal to $1$.
Moreover, for every $\bk\in \Z^r$, since $\norm\Phi_{\CC^{2m}(\wt
X)}\leq 1$, then
 $\norm{\Phi_\bk}_{\CC^{2m}(\wt X)}\leq 1$ and
$|\Phi_{\bk}(x)|\leq C_1(1+\norm\bk)^{-2m}$  for every $\bk\in\Z^r$ and every $x\in\wt X$.
 Since $m>r$, there exists a constant
$C_2$, depending only on $\wt X$ and $\delta$, and thus only on
$\ve$, with $$ \sum_{\bk;\ \norm\bk>C_2} |\Phi_{\bk}(x)|<\delta/2 \
\ \text{ for every }\ x\in \wt X.
$$
Replacing $\Phi$ with its vertical Fourier series
in~\eqref{eq:Phi_Inverse}, this last bound implies that there exists
$\delta'$, depending only on $\wt X$ and $\delta$, and thus only on
$\ve$, such that
\begin{equation}
\label{eq:exists-k} |\E_{n\in[N]}f(n)\Psi_\bk(g_\bk^n\cdot e_{\wt
X})| =|\E_{n\in[N]}f(n)\Phi_\bk(\wt g_\bk^n\cdot e_{\wt X})|
\geq\delta'
\end{equation}
 for some  $\bk\in\Z^r \ \text{ with } \norm\bk\leq C_2$,
where $g_\bk$ is the image of $\tilde{g}$ in $G_\bk$ under the
natural projection. Since $\Phi_\bk=\Psi_\bk\circ\pi_\bk$ and $\norm{\Phi_\bk}_{\CC^{2m}(\wt X)}\leq 1$,
 we have that
$$
\norm{\Psi_k}_{\CC^{2m}(X_\bk)}\leq C_3\  \text{ for every }\
\bk\in\Z^r \ \text{ with }\ \norm\bk\leq C_2.
$$

 We define the family $\CH$ of $2$-step nilmanifolds as follows
 $$
 \CH:=\{ X_\bk\colon \bk\neq 0,\ \norm\bk\leq C_2\}.
 $$
It remains to show that  either   Property~\eqref{it:inverse1} or
 Property~\eqref{it:inverse2} is satisfied.

Suppose  first that  the element $\bk$   in~\eqref{eq:exists-k} is
non-zero. We have $\norm{C_3\inv\Psi_\bk}_{\CC^\infty(X_\bk)}\leq 1$
and $|\E_{n\in[N]} f(n)C_3\inv\Psi_\bk(g_\bk^ne_{X_\bk})|\geq
C_3\inv\delta'$, showing that  Property~\eqref{it:inverse2} is
satisfied.

Otherwise, \eqref{eq:exists-k} holds for $\bk=0$, in which case $X_\bk$ is the maximal torus $\wt G/(\wt G_2\wt\Gamma)\cong\T^s$ and $\Psi_0$ is a function on $\T^s$
with $\norm{\Psi_0}_{\CC^{2m}(\T^s)}\leq C_3$.
Let $\balpha$ be the projection of $g$ in
$\T^s$. For some constant $C_4$, we have
\begin{gather*}
\sum_{\bell\in\Z^s} |\widehat{\Psi_0}(\bell)|\leq C_4\ ;\\
\delta'\leq |\E_{n\in[N]}f(n)\Psi_0(n\alpha)|
\leq\sum_{\bell\in\Z^s}
|\widehat{\Psi_0}(\bell)| |\E_{n\in[N]} f(n)\, \e(n\, \bell\cdot\balpha)|.
\end{gather*}

 Thus, there exists $\bell\in\Z^s$ with
$|\E_{n\in[N]} f(n)\, \e(n\, \bell\cdot\balpha)|\geq \delta'/C_4$.   Since the $l^{4/3}(\Z_N)$-norm of the Fourier
coefficients of any linear phase is bounded by a universal constant,
arguing as in the proof of Lemma~\ref{lem:U2-intervals} we deduce
 that $\norm f_{U^2(\Z_N)}\geq \delta''$, for
some constant $\delta'':=\delta''(\ve)$, showing that
Property~\eqref{it:inverse1} is satisfied. This completes the proof.
\end{proof}

\subsection{A modification of the factorization theorem}
 We plan to use a  decomposition result for
polynomial sequences on nilmanifolds from \cite{GT12a}. We keep the
notation and the conventions of  Section~\ref{subsec:nilmanifolds}
and introduce some additional ones next.

\begin{definition}[Total equidistribution]
Let $X$ be a nilmanifold. We say that the finite sequence $g\colon
[N]\to G$   is \emph{totally $\ve$-equidistributed on $X$} if
$$ | \E_{n\in [N]}
\one_P(n) F(g(n)\cdot e_X)  |\leq \ve,
$$
for all $F\in \lip{(X)}$ with $\norm{F}_{\text{Lip}(X)}\leq 1$ and
$\int F\, dm_X=0$, and all arithmetic progressions $P$ in $[N]$.
\end{definition}
Modulo a change in the constants, our definition of total
equidistribution  is equivalent to the one given in \cite{GT12a},
where  the claimed estimate is $ | \E_{n\in P} F(g(n)\cdot e_X)|\leq
\ve $ for   all  $F\in \lip{(X)}$ with $\norm{F}_{\text{Lip}(X)}\leq
1$ and $\int F\, dm_X=0$, and arithmetic progressions $P$ in $[N]$
with $|P|\geq \ve N$.

\begin{notation} If $(g(n))$ is a sequence in $G$ and $h\in \Z$,
we denote by $(\partial_h(g))$ the sequence defined by
$\partial_hg(n):=g(n+h)g(n)\inv$ for $n\in \N$.
\end{notation}

\begin{definition}[Polynomial sequences]
A \emph{polynomial sequence in} a nilpotent group $G$ is  a sequence
$g\colon \N\to G$ that has the form $g(n)=a_1^{p_1(n)}\cdots
a_k^{p_k(n)}$, where $a_1,\ldots, a_k\in G$ and $p_1,\ldots, p_k\in
\Z[t]$.

An equivalent definition~\cite[Lemma~6.7]{GT12a} is that there
exists an integer $d$ such that
$\partial_{h_{d+1}}\partial_{h_d}\cdots
\partial_{h_1}g(n)=\one_G$ for all $n\in \N$ and all $h_1,\ldots,
h_{d+1}\in \Z$. The smallest integer $d$ with this property is the
\emph{degree} of the sequence.

We say that a sequence $(g(n))$ in $G$ is a degree $2$ polynomial
sequence with \emph{coefficients in the natural filtration} if it
can be written as
$$
g(n)=g_0g_1^ng_2^{\binom n2}\ \text{ where }\ g_0,g_1\in G\ \text{ and }\ g_2\in G_2.
$$
\end{definition}

\begin{definition}[Rational elements]
We say that an element $g\in G$ is \emph{$Q$-rational} for some
$Q\in \N$ if there exists $m\leq Q$ with $g^m \in\Gamma$. We say
that $g$ is \emph{rational} if it is $Q$-rational for some $Q\in
\N$.
\end{definition}
Rational elements form a countable subgroup of $G$~\cite[Lemma A.12]{GT12a}.
\begin{definition}[Smooth and rational sequences]
Given a nilmanifold $G/\Gamma$ and $M,N\in \N$ with $M\leq N$,  we
say that
\begin{itemize}
\item
 the sequence   $\epsilon\colon[N]\to G$
is $(M,N)$-\emph{smooth} if for every $n\in[N]$ we have
$d_G(\epsilon(n), \one_G)\leq M$
and $d_G(\epsilon(n),\epsilon(n-1))\leq M/N$;
\item
 the sequence $\gamma\colon[N]\to G$ is $M$-\emph{rational} if
 for every $n\in[N]$, $\gamma(n)$ is $M$-rational.
\end{itemize}
 \end{definition}
In \cite{GT12a} the next result is stated only for functions of the
form $\omega(M)=M^{-A}$, but the same proof  works for arbitrary
functions $\omega\colon \N\to\R^+$.  For the notion of a rational
subgroup that is used in the next statement we refer the reader to
Appendix~\ref{ap:A}.
\begin{theorem}[Factorization of polynomial sequences~{\cite[Theorem~1.19]{GT12a}}]
\label{th:factorization}
 Suppose that   $X:=G/\Gamma$ is a $2$-step
nilmanifold. For every $M\in\N$ there exists a finite
collection $\CF(M)$ of sub-nilmanifolds of $X$, each of the form
$X':=G'/\Gamma'$, where $G'$ is a rational subgroup of $G$ and
$\Gamma':=G'\cap\Gamma$, such that the following holds:

 For every function 
 $\omega\colon \N\to\R^+$ and every $M_0\in \N$, there exists  $M_1:=M_1(M_0,X,\omega)$, such
that for every $N\in\N$, and for every
 degree $2$ polynomial sequence
$(g(n))_{n\in[N]}$ in $G$  with  coefficients in the natural
filtration of $G$, there exist   $M\in \N$ with $M_0\leq M\leq M_1$,
a nilmanifold $X'$ belonging to the family $\CF(M)$, and a
decomposition
$$
g(n) = \epsilon(n) g'(n)\gamma(n), \quad n\in [N],
$$
 where $\epsilon,g',\gamma\colon [N]\to G$ are
degree $2$ polynomial sequences with coefficients in the natural
filtration of $G$, that satisfy
\begin{enumerate}
\item
\label{it:smooth} $\epsilon$ is $(M,N)$-smooth;
\item
\label{it:equid} $(g'(n))_{n\in[N]}$ takes values in $G'$, and the
finite sequence $(g'(n)\cdot e_{X'})_{n\in[N]}$ is totally
$\omega(M)$-equidistributed in $X'$ with the metric $d_{X'}$;
\item
\label{it:periodic} $\gamma\colon [N]\to G$  is $M$-rational, and
$(\gamma(n)\cdot e_X)_{n\in [N]}$ has period at most $M$.
\end{enumerate}
\end{theorem}
\begin{remark} We emphasize  that in the previous result and subsequent
corollary, the number  $M_1$ is independent of $N$ and the  family
$(\CF(M))_{M\in \N}$ is independent of $\omega$ and $N$. The choice
of the sub-nilmanifold $X'$ may depend on $N$, but it stabilizes for
large enough $N$.
\end{remark}
We  will use  the following corollary of the previous result  that
 gives a more precise factorization for a certain explicit class of polynomial sequences.
 \begin{corollary}[Modified factorization]
\label{cor:factorization1} Let $X:=G/\Gamma$ be a $2$-step
nilmanifold with  vertical torus of dimension $1$.
For every $M\in\N$ there
exists a finite collection
$\CF(M)$ of sub-nilmanifolds of $X$, each  of the form
$X':=G'/\Gamma'$,
   where  $\Gamma':=G'\cap\Gamma$ and
 either
\begin{enumerate}
\item
\label{it:Gprimeabelien}   $G'$ is an Abelian rational subgroup of
G; or
\item
\label{it:Gprimenonabelien} $G'$ is a non-Abelian rational subgroup
of $G$ and $G'_2/(G'_2\cap\Gamma')$ has dimension $1$,
\end{enumerate}
such that the following holds:

For every 
function $\omega\colon \N\to\R^+$  and every $M_0\in \N$, there
exists $M_1:=M_1(M_0, X,\omega)$, such that for every $N\in\N$ and
every $g\in G$, there exist  $M \in \N$ with $M_0\leq M \leq M_1$, a
nilmanifold $X'$ belonging to the family $\CF(M)$, and a
decomposition
$$
g^n = \epsilon(n) g'(n)\gamma(n), \quad n\in [N],
$$
 where $\epsilon,g',\gamma\colon [N]\to G$ are
polynomial sequences that satisfy
\begin{enumerate}
\setcounter{enumi}{2}
\item\label{it:dec3}
$\epsilon$ is $(M,N)$-smooth;
\item\label{it:dec4}
$(g'(n))_{n\in[N]}$  takes  values in $G'$, has the form
\begin{equation}
\label{eq:g-prime} g'(n)=g'_0g_1'^n g_2'^{\binom n 2}\  \text{ where
}\ g'_0,g'_1,g'_2\in G',\ \text{ and moreover } \ g'_2\in G'_2\
\text{ in case~\eqref{it:Gprimenonabelien}},
\end{equation}
and  $(g'(n)\cdot e_{X'})_{n\in[N]}$ is totally
$\omega(M)$-equidistributed in $X'$ with the metric $d_{X'}$;
\item\label{it:dec5}
$\gamma\colon [N]\to G$  is $M$-rational, and
$(\gamma(n)\cdot e_X)_{n\in[N]}$ has period at most $M$.
\end{enumerate}
\end{corollary}
\begin{proof}
Let the integers $M_1, M$ and the nilmanifold  $X'=G'/\Gamma'$
belonging to the family $\CF(M)$ be given
 by Theorem~\ref{th:factorization}. Note that the sequence
$(g^n)_{n\in [N]}$ is a degree $2$ polynomial sequence in $G$ with
coefficients in the natural filtration. Let $(g'(n))_{n\in [N]}$ be
the sequence given by the decomposition
  of Theorem~\ref{th:factorization}. This sequence is a polynomial sequence in $G$ with coefficients in the natural filtration of $G$ and thus it can be written as
 $g'(n)=g'_0g_1'^ng_2'^{\binom n2}$ for some $g'_0,g'_1\in G$ and some $g'_2\in G_2$.
It remains to show   that this sequence
 has the form \eqref{eq:g-prime}, i.e. that $g'_0,g'_1,g'_2\in G'$ and  furthermore that $g'_2\in G'_2$ in case $(ii)$.

Since $g'_0=g'(0)$ we have $g'_0\in G'$. Recall that $\partial_1g(n)=g(n+1)g(n)\inv$.
 Using the fact that $G_2$ is included in the center of $G$ we obtain
 $$
\partial_1  g'(n)= g'_0 g'_1 g_0'^{-1} g_2'^n \quad \text{ and } \quad \partial_1^2 g'(n)=g'_2.
$$
It follows that $g'_1$ and $g'_2\in G'$.

If we are in   case~\eqref{it:Gprimeabelien}, then  $G'$ is Abelian
and we are done. Suppose now that we are in case $(ii)$ where  $G'$
is non-Abelian. Then $G'_2$ is a non-trivial subgroup of $G_2$.
Moreover, $G'_2$ is closed and connected, and by hypothesis $G_2$ is
isomorphic to the torus $\T$. It follows that $G'_2=G_2$. Hence,
$g_2'\in G_2'$. This shows that the sequence $(g'(n))_{n\in [N]}$
has the required properties and completes the proof.
  \end{proof}

\section{Correlation of multiplicative functions with
nilsequences}\label{S:CorrelationNil} The main goal of this section
is to establish some correlation estimates needed in the proof of
the decomposition results given in the next section. We show that
multiplicative functions do not correlate with a class of  totally
equidistributed $2$-step polynomial nilsequences. The precise
statements appear in Propositions~\ref{lem:equid-square} and
\ref{lem:equid-square-ab}.

\subsection{Quantitative equidistribution}
We start with a quantitative equidistribution result for polynomial
sequences on nilmanifolds by Green and Tao \cite{GT12a} that
strengthens an earlier qualitative equidistribution result of
Leibman \cite{L05}.
\begin{definition}
If $g\colon [N]\to \T$ is a finite polynomial sequence in
$\T$, of the form
$$
g(n)=\alpha_0+\alpha_1n+\alpha_2{\binom
n2}+\dots+\alpha_d{\binom nd}\ \text{ where }\  d\in\N\ \text{ and
}\ \alpha_0,\dots\alpha_d\in\T,
$$
then the \emph{smoothness norm} of $g$  is defined by
$$
\norm g_{C^\infty[N]}:=\max_{1\leq j\leq d} N^j\norm{\alpha_j}.
$$
 \end{definition}

\begin{theorem}[Quantitative Leibman Theorem \cite{GT12a}]
\label{th:Leibman} Let  $X:=G/\Gamma$ be a $2$-step nilmanifold.
Then for every  $ \ve>0$ there exists
$D:=D(X, \ve)>0$ with the following property: For every
$N\in\N$,  if $g\colon [N]\to G$ is a polynomial sequence of degree
at most
 $2$ such that the sequence $(g(n)\cdot e_X)_{n\in[N]}$ is not  $\ve$-equidistributed in $X$,
  then there exists
a horizontal character $\eta:=\eta(X,\ve)$  such that
$$
0<\norm\eta\leq D\quad \text{ and }\quad \norm{\eta\circ
g}_{C^\infty[N]}\leq D.
$$
\end{theorem}
We plan to use  a partial converse of this result.

\begin{lemma}[A partial  converse to  Theorem~\ref{th:Leibman}]
\label{lem:Leibman_Inverse} Let $X:=G/\Gamma$ be a $2$-step
nilmanifold. There exists $C:=C(X)>0$ with the following property:
For every $D\in \N$, if $N$ is sufficiently large, depending only on  $D$, and
$(g(n))_{n\in[N]}$ is a degree $2$ polynomial sequence in $G$  such
that there exists a non-trivial horizontal character $\eta$ of $X$
with $\norm\eta\leq D$ and $\norm{\eta\circ g}_{C^\infty[N]}\leq D$,
then the sequence $(g(n)\cdot e_X)_{n\in[N]}$ is not totally
$CD^{-2}$-equidistributed in $X$.
\end{lemma}
\begin{proof}
Since $\norm{\eta\circ g}_{C^\infty[N]}\leq D$ we have
$$
\eta(g(n))=\sum_{0\leq j\leq 2}\alpha_i\binom nj \ \text{ for some }
\ \alpha_j\in\T\ \text{ with } \ \norm{\alpha_j}\leq \frac{D}{N^j}\
\text{ for } \ j=1,2.
$$
It follows  that
$$
\bigl|1-\e\big(\eta(g(n))\big)\bigr|\leq \frac{1}{2}\ \ \text{ for }\
1\leq n\leq c_1\frac{N}{D}
$$
where $c_1$ is a universal constant and we assume that $N\geq D/c_1$. Thus,
$$
\bigl|\E_{n\leq \lfloor c_1N/D\rfloor }\e\big(\eta(g(n))\big)\bigr|\geq
\frac{1}{2},
$$
which gives
$$
\bigl|\E_{n\in [N] }\one_{[\lfloor c_1N/D\rfloor]}(n) \e\big(\eta(g(n))\big)\bigr|\geq
\frac{c_1}{2D}-\frac{1}{N}\geq \frac{c_1}{4D}
$$
assuming that $N\geq 4D/c_1$.

Furthermore, since $\norm\eta\leq D$,  the function $x\mapsto
\e(\eta(x))$, defined on $X$, is Lipschitz with constant at most
$C'D$ for some $C':=C'(X)$, and    has integral $0$ since $\eta$ is
a non-trivial horizontal character. Therefore,  the sequence
$(g(n)\cdot e_X)_{n\in[N]}$ is not totally $(CD^{-2})
$-equidistributed in $X$ where  $C:= c_1/(4C')$.
\end{proof}

\subsection{Discorrelation estimates}
Next we prove the two main results of this section that give
asymptotic orthogonality of multiplicative functions to some totally
equidistributed nilsequences. These results will be  used later in
the proof of Theorem~\ref{th:Decomposition-weakU3-intro}  to treat
each of the two distinct cases arising from an application of
Corollary~\ref{cor:factorization1}. Both proofs are based on
K\'atai's orthogonality criterion (Lemma~\ref{lem:katai}) and the
quantitative Leibman Theorem (Theorem~\ref{th:Leibman}).
\begin{proposition}[Discorrelation estimate I]
\label{lem:equid-square} Let $X:=G/\Gamma$ be a $2$-step nilmanifold
and  $\tau>0$. There exists  $\sigma:=\sigma(X,\tau)>0$ with the
following property: For every sufficiently large $N$, depending only
on $X$ and $\tau$, if
 $(g(n))_{n\in[N]}$ is a degree $2$ polynomial sequence in $G$ with
 coefficients in the natural filtration   that is  totally
$\sigma$-equidistributed in $X$, then
$$
\sup_{m,\chi, \Phi,P}|\E_{n\in[N]} \one_P(n)
\chi(n)\Phi(g(m+n)\cdot e_X)|<\tau
$$
where the sup is taken over  all integers $m$ with $|m|\leq N$, multiplicative
functions $\chi \in\CM$, functions $\Phi\in \lip(X)$ with
 $\norm\Phi_{\lip(X)}\leq 1$ and $\int \Phi\, dm_X=0$, and arithmetic
progressions $P\subset[N]$.
\end{proposition}
  Our proof of  Proposition~\ref{lem:equid-square}
depends on the fact
that if $p, p'$ are distinct primes, then  equidistribution properties of  $(g(n)e_X)_{n\in [N]}$ on $X$
imply equidistribution  of the sequence
$
((g(pn),g(p'n)) \cdot e_{X\times X})_{n\in [N]}
$
on a sub-nilmanifold of $X\times X$ on which the function $\Phi\otimes\overline \Phi$
integrates to zero (assuming that $\Phi$ is a nilcharacter with non-zero frequency).
As the proof
of this fact is rather involved  we defer it to Appendix~B and proceed to prove an easier discorrelation estimate that will also be needed later.

\begin{proposition}[Discorrelation estimate II]
\label{lem:equid-square-ab}
 Let $s\in \N$ and $\tau>0$. There exists
$\sigma:=\sigma(s,\tau)>0$ with the following property: For every
sufficiently large $N$, depending only on  $s$ and $\tau$, if
$(\bg(n))_{n\in[N]}$ is a polynomial sequence in $\T^s$ of the form
$$
\bg(n)=\balpha_0+\balpha_1n+\balpha_2\binom n2, \quad \balpha_i\in
\T^s,\ \  i=0,1,2,
$$
that is  totally $\sigma$-equidistributed in $\T^s$, then
$$
\sup_{m,\chi,\Phi,P}|\E_{n\in[N]}\one_P(n) \chi(n)\Phi(\bg(m+n))|<\tau
$$
where the sup is taken over  all  integers $m$ with $|m|\leq N$,
multiplicative functions $\chi \in\CM$, $\Phi\in \lip(\T^s)$ with
 $\norm\Phi_{\lip(\T^s)}\leq 1$ and $\int
\Phi \, dm_{\T^s}=0$,  and  arithmetic progressions
$P\subset[N]$.
\end{proposition}
\begin{proof}
In this proof  $C_1,C_2,\ldots$ are constants that  depend only on
$s$ and $\tau$.

Without loss of generality, we can assume  that
$\norm\Phi_{\CC^{2s}(\T^s)}\leq 1$. Suppose that
$$
|\E_{n\in[N]}\one_P(n) \chi(n)\Phi(\bg(m+n))|\geq\tau
$$  for some
$\tau>0$, integer $m$ with $|m|\leq N$, $\chi \in\CM$, $\Phi\in
\lip(\T^s)$ with   $\norm\Phi_{\lip(\T^s)}\leq
1$ and $\int \Phi \, dm_{\T^s}=0$, and arithmetic progression $P\subset[N]$. We have
\begin{multline*}
\tau\leq\bigl|\E_{n\in[N]}\one_P(n)\chi(n)\Phi(\bg(m+n))\bigr|
=\bigl|
\sum_{\bk\in\Z^s}\widehat\Phi(\bk)\,\E_{n\in[N]}\one_P(n)\chi(n)\e(\bk\cdot\bg(m+n))
\bigr|\\
\leq \sum_{\bk\in\Z^s}\,\frac
{C_0}{1+\norm\bk^{2s}}\,\bigl|\E_{n\in[N]}\one_P(n)\chi(n)\e(\bk\cdot\bg(m+n))\bigr|
\end{multline*}
for some constant $C_0:=C_0(s)$. It follows that there exist
constants  $C_1$, $\theta:=\theta(s,\tau)>0$, and  $\bk\in\Z^s$,
such that
$$
\norm\bk\leq C_1\quad \text{ and }\quad
|\E_{n\in[N]}\one_P(n)\chi(n)\e(\bk\cdot \bg(m+n))\bigr|\geq\theta.
$$
Let $\delta$ and $K$ be defined by Lemma~\ref{lem:katai}, with
 $\theta$ substituted for $\ve$. Note that
$\delta$ and $K$ depend on $s$ and $\tau$ only.  There exist primes
$p,p'$ with $p<p'\leq K$ such that
$$
\bigl|\E_{n\in[\lfloor
N/p'\rfloor]}\one_P(pn)\one_P(p'n)\,\e\bigl(\bk\cdot
(\bg(m+pn)-\bg(m+p'n))\bigr)\bigr|\geq \delta.
$$
Writing $\beta_1=\bk\cdot\balpha_1$, $\beta_2=\bk\cdot \balpha_2$,
and $\one_P(pn)\one_P(p'n)=\one_{P_1}(n)$ where $P_1\subset [\lfloor
N/p'\rfloor]$ is an arithmetic progression,  we can rewrite the previous
estimate  as
$$|\E_{n\in[\lfloor N/p'\rfloor] } \one_{P_1}(n) \; \e(u(n))|\geq\delta$$
 where
$$
u(n)= \binom n2  \beta_2(p^2-p'^2)+ n\Big( \beta_2\bigl(\binom
p2-\binom{p'}2\bigr)+(m\beta_2+\beta_1)(p-p')\Big).
$$
Since $\lfloor N/p'\rfloor \geq  N/2K$,  the sequence
$(u(n))_{n\in[N]}$ is not totally $\delta/2K$-equidistributed in the
circle. By the Abelian version of Theorem~\ref{th:Leibman}, it
follows that there exists an integer $l$ with $0< l\leq
D:=D(\delta/2K)$ such that
$$
\norm{ l\beta_2(p^2-p'^2)}\leq \frac D{N^2} \quad \text{ and } \quad
\Bigl\Vert l\beta_2\bigl(\binom
p2-\binom{p'}2\bigr)+l(m\beta_2+\beta_1)(p-p')\bigr)\Bigr\Vert\leq\frac
DN.
$$
We deduce first that $\beta_2$ is at a distance $\leq C_2/N^2$ of a
rational with  denominator $\leq C_3$, and then that $\beta_1$ is at
a distance $\leq C_4/N$ of a rational with  denominator $\leq C_5$
(here we used that $|m|\leq N$). Hence,  there exists a non-zero
integer $l'$, bounded by some constant $C_6$, such that
$\norm{l'\beta_2}\leq C_7/N^2$ and $\norm{l'\beta_1}\leq C_8/N$.
Taking $\bk'=l'\bk$, we have
$$
0<\norm{\bk'}\leq C_1C_6\, ; \quad  \norm{\bk'\cdot\balpha_2}\leq
\frac{C_7}{N^2}\, ; \quad \norm{\bk'\cdot\balpha_1}\leq
\frac{C_8}{N}.
$$
 Using this and
Lemma~\ref{lem:Leibman_Inverse}, we deduce that the sequence
$(\bg(n))_{n\in[N]}$ is not totally $\sigma$-equidistributed for
some $\sigma>0$ that depends only on $s$ and  $\tau$, completing the
proof.
\end{proof}

\subsection{Equidistribution on shifted nilmanifolds}
We give one more application of Theorem~\ref{th:Leibman} that
will be needed in the next section:

\begin{lemma}[Shifting the nilmanifold]
\label{lem:equid_conjug} Let $X:=G/\Gamma$ be a $2$-step
nilmanifold, $G'$ be a rational subgroup of $G$, $h\in G$ be a
rational element, $X':=G'\cdot e_X$, $e_Y:=h\cdot e_X$, and
$Y:=G'\cdot e_Y$.
 For every
 $\ve>0$ there exists
$\delta:=\delta(G',X, h,\ve)>0$ with the following property: If
$(g'(n))_{n\in[N]}$ is a polynomial sequence in $G'$ of degree at
most $2$,  such that the sequence $(g'(n)\cdot e_X)_{ n\in [N]}$ is
totally $\delta$-equidistributed in $X'$, then the sequence
$(g'(n)\cdot e_Y)_{n\in[N]}$ is totally $\ve$-equidistributed in
$Y$.
\end{lemma}
By Lemma~\ref{lem:Gprimey} in the Appendix, $\Gamma\cap G'$ is
co-compact in  $G'$  and thus $X'$ is a closed
 sub-nilmanifold of $X$.
In a similar fashion,  $(h\Gamma h\inv)\cap G'$ is co-compact in
$G'$ and $Y$ is a closed sub-nilmanifold of $X$.
\begin{proof}
In this proof, $C_1,C_2,\ldots$  are constants that depend only on
$G',X,$ and $h$.

By Lemma~\ref{lem:finite-index} in the Appendix, the group
$\Gamma\cap h \Gamma h\inv\cap G'$ has  finite index in the two
groups $\Gamma\cap   G'$   and $h \Gamma h\inv\cap G'$. We write
$$
Z':= G'/G'_2(\Gamma\cap G'),\quad Z_1:=G'/G'_2(\Gamma\cap h \Gamma h\inv\cap G'), \  \text{ and } \
Z_2:=G'/G'_2(h \Gamma h\inv\cap G').
$$
Then $Z'$ is the horizontal torus
of $X'$, the nilmanifold $Y$ can be  identified with $G'/(h\Gamma
h\inv\cap G')$, and thus  $Z_2$ is the horizontal torus of $Y$.
Let  $p\colon Z_1\to Z'$  and $q\colon Z_1\to
Z_2$ be the natural projections. These group homomorphisms are finite to one.

Let $\ve>0$, and suppose  that the polynomial sequence $(g'(n)\cdot
e_Y)_{ n\in[N]}$ has degree at most  $2$ and is not totally
$\ve$-equidistributed in $Y$. We denote by $D$ the integer that
Theorem~\ref{th:Leibman} associates to $\ve$ and $Y$. Then there
exists a non-trivial horizontal character $\eta$ of $Y$, with
\begin{equation}
\label{eq:equid1} 0\neq \norm\eta\leq D\ \  \text{  and }\ \
\norm{\eta(g'(n))}_{C^\infty[N]}\leq D.
\end{equation}
Recall that $\eta$ factors to a character of the horizontal torus
$Z_2$ of $Y$; we slightly abuse notation and denote it also by
$\eta$. We have that $\eta\circ q$ is a character of $Z_1$ and since
$q\colon Z_1\to Z_2$ is finite to one, $\norm{\eta\circ q}\leq
C_1\norm\eta$ for some constant $C_1$.

  Since  $\Gamma\cap h \Gamma h\inv\cap G'$ has finite index in $\Gamma \cap G'$, there exists $\ell\in\N$ such that
  $\gamma^\ell\in  \Gamma\cap h \Gamma h\inv\cap G'$ for every $\gamma\in \Gamma\cap G'$.
  Therefore, since the restriction of $\eta\circ q$ to $\Gamma\cap h \Gamma h\inv\cap G'$
  is trivial, for every $\gamma\in\Gamma\cap G'$ we have
  $\ell\eta\circ q(\gamma)=\eta\circ q(\gamma^\ell)=1$.
Hence, $\ell\eta\circ q$ has a trivial restriction to
$G'_2(\Gamma\cap G')$and so  there exists a character $\zeta$ of
$Z'$ with $\ell\eta\circ q=\zeta\circ p$. We have $0\neq
\norm\zeta\leq C_2\norm{\ell\eta\circ q}\leq C_3\norm \eta\leq C_3D$
for some constants $C_2,C_3$. We consider $\zeta$ as a horizontal
character of $X'=G'/(G'\cap\Gamma)$ and thus as a character of $G'$.

For every $n\in \N$ we have  $\zeta(g'(n))= \ell\eta(g'(n))$ and thus, by
hypothesis~\eqref{eq:equid1}, $\norm{\zeta\circ
g'}_{C^\infty[N]}\leq C_4D$ for some constant $C_4$. By
Lemma~\ref{lem:Leibman_Inverse}, the sequence $(g'(n)\cdot e_X)_{
n\in[N]}$ is not totally $C_5D^{-1}$-equidistributed in $X'$ for
some constant $C_5$. Letting $\delta=C_5D^{-1}$ completes the proof.
\end{proof}

\subsection{A model discorrelation result}
Lastly, we give some uniform discorrelation estimates that serve as
a model for the more complicated estimates obtained in the sequel.
The argument is based on
 K\'atai's criterion (Lemma~\ref{lem:katai}) and the Abelian version of Theorem~\ref{th:Leibman} which
  is nothing more than a suitable use of  Weyl's estimates.
\begin{proposition}[Model discorrelation estimates for $\chi_{N,u}$]
Let $\ve>0$. Then there exists $\theta:=\theta(\ve)$ such that, for every
 sufficiently large $N$, depending only on  $\ve$, the following holds:
If
$$
\chi_{N,s}:=\chi_N*\phi_{N,\theta}, \quad  \chi_{N,u}:=\chi_N-\chi_{N,s},
$$
where $\chi_N=\chi\cdot  \one_{[N]}$ and   $\phi_{N,\theta}$ is
the kernel defined by \eqref{eq:fourier-phi},   then
\begin{equation}\label{E:lkj''}
\sup_{\chi\in \CM, \alpha \in \R}|\E_{n\in [\tN]}\chi_{N,u}(n) \e(n^2\alpha)|\leq \ve.
\end{equation}
\end{proposition}
\begin{proof}[Proof (Sketch)] Let $\ve>0$ and $N$ be sufficiently large depending only on $\ve$ (how large will be determined below).

Let  $\theta:=\theta(\ve)>0$   be  given by   \eqref{E:lkjh}
below and for this value of $\theta$ let $\phi_{N,\theta}$ be given
by \eqref{eq:fourier-phi}.
 Theorem~\ref{th:Decomposition-weakU2-intro} implies
that for sufficiently large $N$, depending only on $\ve$, we have
\begin{equation}\label{E:lkjh'}
\norm{\chi_{N,u}(n)}_{U^2(\Z_\tN)}\leq \theta(\ve).
\end{equation}
 We claim that the asserted estimate
\eqref{E:lkj''} holds. Arguing by contradiction,  suppose
that
\begin{equation}\label{E:wanted1}
|\E_{n\in [\tN]}\chi_{N,u}(n) \e(n^2\alpha)|> \ve
\end{equation}
for some $\chi\in \CM$ and  $\alpha\in \R$.
We consider two cases depending on the total equidistribution
properties of the  sequence $(n^2\alpha)_{n\in [\tN]}$.

\medskip

\noindent {\bf Minor arcs.}  We use
    Proposition~\ref{lem:equid-square-ab} with $s=1$ and   $\ve/3$ in place of $\tau$.  We get that there exists
    $\sigma:=\sigma(\ve)$,  such that for all sufficiently large $N$, depending only on $\ve$,
if the sequence $(n^2\alpha)_{n\in [\tN]}$ is totally
$\sigma$-equidistributed, then
$$
\max_{m\in [-\tN,\tN]}|\E_{n\in [N]}\chi(n) \e((m+n)^2\alpha)|\leq
\frac{1}{2}\ve.
$$
Using this and  the  fact that $\chi_{N,u}=\chi_N*(1-\phi)$ where  $\phi$ is a kernel on $\ZN$,
we deduce   (see Section~\ref{SS:ProofFinally} for details) that for
all sufficiently large $N$, depending only on $\ve$, we have
$$
|\E_{n\in [\tN]}\chi_{N,u}(n) \e(n^2\alpha)|\leq \ve
$$
which contradicts \eqref{E:wanted1}.

\medskip

\noindent {\bf Major arcs.}  Suppose now that the sequence
$(n^2\alpha)_{n\in [\tN]}$ is not totally $\sigma$-equidistributed
where  $\sigma$ was defined in the minor arc step. Then, as is well
known (and also follows by Lemma~\ref{lem:Leibman_Inverse}),
$\alpha$ has to be close to a rational with a small denominator,
more precisely, there exist positive integers $Q,R$ that depend only
on $\sigma$, and consequently only on $\ve$, and positive integers
$p,q\leq Q$ such that
$$
\Big|\alpha-\frac{p}{q}\Big|\leq \frac{R}{\tN^2}.
$$
We factor the sequence  $(n^2\alpha)_{n\in [\tN]}$ as follows
$$
n^2\alpha=\epsilon(n)+\gamma(n), \quad \text{ where }  \quad \epsilon(n):=n^2\Big(\alpha-\frac{p}{q}\Big), \ \ \gamma(n):=n^2\frac{p}{q}.
$$
Note that $|\epsilon(n+1)-\epsilon(n)|\leq 2R/\tN$ for $n\in [\tN]$.
Furthermore,  the sequence $\gamma(n)$ is periodic with period $q$. After
partitioning the interval $[\tN]$ into sub-progressions where
 $\epsilon(n)$ is almost constant and $\gamma(n)$ is constant, and using the pigeonhole principle,
  it is not hard to deduce
from \eqref{E:wanted1} (see
Section~\ref{SS:subprogression} for details) that
\begin{equation}\label{E:wanted2}
|\E_{n\in [\tN]}\one_P(n)\cdot \chi_{N,u}(n)|>
\frac{1}{10}\frac{\ve^2}{QR}
\end{equation}
for some arithmetic progression $P\subset[\tN]$ provided that $N$ is
sufficiently large depending only on $\ve$. Using \eqref{E:wanted2}
and Lemma~\ref{lem:U2-intervals} we deduce that
\begin{equation}\label{E:lkjh}
\norm{\chi_{N,u}(n)}_{U^2(\Z_\tN)}> \frac{1}{c_1} \frac{\ve^2}{QR}
=:\theta(\ve)
\end{equation}
where $c_1$ is a universal constant. This contradicts \eqref{E:lkjh'}
and completes the
proof.
\end{proof}
In the next section we prove a strengthening of the previous result where
the place of
 $(\e(n^2\alpha))$ takes
any two step nilsequence $(\Phi(a^n\cdot e_X))$ where $\Phi$ is a
function
 on a $2$-step nilmanifold with Lipschitz norm at most $1$. Our proof is much more
complicated in this case
but  the basic   strategy remains the same as in the previous
argument.

\section{Higher order Fourier analysis of multiplicative functions}\label{S:U^3}
  The goal of
this  section is to  prove the main  decomposition result stated in
Theorem~\ref{th:strong-average-intro}. The key ingredient that
enters its proof is the following ``weaker'' decomposition.
\begin{theorem}[Weak uniform decomposition for the $U^3$-norm]
\label{th:Decomposition-weakU3-intro} For every $\theta_0>0$ and
$\ve>0$, there exist  a positive real $\theta<\theta_0$, and
positive integers  $Q:=Q(\ve,\theta_0)$  and
$R:=R(\ve,\theta_0)$, with the following properties: For
 every  sufficiently large $N$, depending only on  $\theta_0$ and  $\ve$,  and for  every $\chi\in\CM$, the function $\chi_N$ admits
the decomposition
$$
 \chi_N(n)=\chi_{N,s}(n)+\chi_{N,u}(n) \quad \text{ for every }\  n\in
 \ZN,
$$
 where  the functions $\chi_{N,s}$ and $\chi_{N,u}$  satisfy:
\begin{enumerate}
\item
\label{it:weakU3-1} \vide $\chi_{N,s}=\chi_N*\phi_{N,\theta}$, where
$\phi_{N,\theta}$ is the  kernel on $\Z_\tN$ defined  by
\eqref{eq:def-phi} and is independent of $\chi$, and
  the convolution product is  defined in $\ZN$;
\item\label{it:weakU3-2}
 $\displaystyle|\chi_{N,s}(n+Q)-\chi_{N,s}(n)|\leq \frac R\tN$ for every $n\in
\Z_\tN$, where  $n+Q$ is taken $\!\!\! \mod \tN$;
\item
\label{it:weakU3_3}
\vide $\norm{\chi_{N,u}}_{U^3(\ZN)}\leq\ve$.
\end{enumerate}
\end{theorem}
The proof of Theorem~\ref{th:Decomposition-weakU3-intro}  takes the
largest part of this section. The main disadvantage of this result
is that  the bound on the
  uniform component  is not strong enough for our applications. In
  Section~\ref{subsec:proof_strong} we
  combine  Theorem~\ref{th:Decomposition-weakU3-intro} with  an energy increment
   argument to prove Theorem~\ref{th:strong-average-intro} that gives very strong bounds on the uniform component.

\subsection{Some preliminary remarks and proof strategy}
\label{subsec:prelim}   A substantial  part of our proof is consumed
in handling  correlation estimates of  arbitrary multiplicative
functions with  $2$-step nilsequences of bounded complexity. Our
proof strategy  follows the general ideas of an argument of Green
and Tao from \cite{GT08a,GT12b} where uniformity
properties of  the M\"{o}bius function were studied. In our case, we
are faced with a few important additional difficulties   stemming
from the fact that we are forced to work with all multiplicative
functions some of which are not $U^2$-uniform (see the example in
Section~\ref{subsec:decomposition}). Furthermore, we have to
 establish estimates with implied constants independent of
the elements of $\CM$.
We give a brief summary of our strategy next.

 To
compensate for the lack of $U^2$-uniformity of a multiplicative function $\chi$  we  subtract from
 it   a suitable ``structured component''  $\chi_s$ given by
 Theorem~\ref{th:Decomposition-weakU2-intro},
 so
that $\chi_u=\chi-\chi_s$ has extremely small  $U^2$-norm. Our goal
is then to show that  $\chi_u$ has small $U^3$-norm.
 In view of the $U^3$-inverse
theorem of Green and Tao (Theorem~\ref{th:inverse}), this would
follow if we show that $\chi_{u}$  has very small correlation with
all $2$-step nilsequences of bounded complexity. This then  becomes
our new goal.

The factorization theorem for nilsequences
(Theorem~\ref{cor:factorization1}) practically allows us    to treat
correlation with major arc and minor arc $2$-step nilsequences
separately.
 Orthogonality to  major arc (approximately periodic) nilsequences  is easily implied by the  $U^2$-uniformity of $\chi_u$. So our efforts concentrate on the  minor arc (totally equidistributed) nilsequences.
Combining the orthogonality criterion of K\'atai (Lemma~\ref{lem:katai}) with  some
quantitative equidistribution results on nilmanifolds (Theorem~\ref{th:Leibman}), we deduce
 that the arbitrary multiplicative function $\chi$ is asymptotically orthogonal to such sequences (Propositions~\ref{lem:equid-square}
 and \ref{lem:equid-square-ab}).
  The function $\chi_u$ is not multiplicative though, but this is easily taken care by the fact that $\chi_u=\chi-\chi_s$ and
  the fact that $\chi_s$  can be recovered from $\chi$ by taking a convolution product with a positive kernel.
  Using these properties it is an easy matter to transfer estimates from $\chi$ to $\chi_u$.
   Combining the above, we get
  the needed orthogonality of $\chi_u$ to all $2$-step nilsequences of bounded
  complexity. Furthermore, a close inspection of the argument shows that all implied constants are independent of $\chi$.
This suffices to complete the proof of the decomposition result.

 Although the previous sketch communicates the basic ideas behind the proof of Theorem~\ref{th:Decomposition-weakU3-intro},
the various results needed to implement this plan come with a significant number
 of parameters that one has to juggle with,
  making the bookkeeping  rather cumbersome.
We use the next section to organize some of these data.
\subsection{Setting up the stage}
\label{subsec:beginning} In this subsection we  define and organize
some data  that will be used in the proof of
Theorem~\ref{th:Decomposition-weakU3-intro}. We take some extra care
to do this before the main body of the proof in order to make sure
that there is no circularity in the admittedly complicated
collection of
 choices involved.
 The
reader is advised to skip this  subsection on a first reading and
refer back to it only when necessary.

Essentially all objects defined below will depend on a positive
number $\ve$ and   on the  integer $\ell$ defined in
Section~\ref{subsec:decomposition}. We consider  these parameters as
fixed, and to ease notation we leave the dependence on $\ve$ and
$\ell$  implicit most of the time.

Furthermore,  most  objects defined below also depend  on a positive
 integer parameter $M$ that we consider for the moment as a free variable.
 The explicit choice of $M$ takes place in Section~\ref{SS:InvFact} and depends on various other choices
  that will be made subsequently;
what is important though is that it is
 bounded above and below by positive constants that depend only
on $\ve$.
 The parameter $M$ plays a central role in our argument, and to avoid confusion
 we keep the dependence on $M$ explicit
 most of the time.

Most  objects we define use the families of  nilmanifolds $\CF(M)$
introduced in  Corollary~\ref{cor:factorization1}. We would like to
stress that these families do not depend on the function $\omega$
that occurs  in this statement. This allows us to postpone the
definition of $\omega$ until Section~\ref{SS:def-parameters}.

\subsubsection{Objects defined by the inverse theorem}
Throughout the argument we let $\ve$ be a fixed positive number.
 Corollary~\ref{cor:inverse} defines the objects
  $$
  \CH:=\CH(\ve), \quad \delta:=\delta(\ve), \quad m:=m(\ve)
  $$
where $\CH$ is a finite family   of nilmanifolds with vertical torus
of dimension $1$, $\delta$ is a positive number, and $m$ is a
positive integer.
   In the sequel we implicitly assume
  that $N$ is sufficiently large, depending only on $\ve$, so that the conclusion of Corollary~\ref{cor:inverse} holds.

\subsubsection{Objects associated to every fixed nilmanifold in $\CH$}
\label{subsec:new_objects} Let $M\in \N$ be fixed and
$X:=G/\Gamma$ be a nilmanifold in $\CH$. We define below various
objects that depend on $M$ and $X$.

By Corollary~\ref{cor:M-rat}, for every $M\in \N$ there exists  a
finite subset $\Sigma:=\Sigma(M,X)$ of $G$, consisting of
$M$-rational elements, such that for every $M$-rational  element
$g\in G$ there exists $h\in \Sigma$ with  $h\inv g\in\Gamma$, that
is, $g\cdot e_X=h\cdot e_X$. We  assume  that $\one_G\in\Sigma$.

Let $\CF:=\CF(M,X)$ be the family of sub-nilmanifolds of $X$ defined
by Corollary~\ref{cor:factorization1}. We define a larger family of
nilmanifolds $\CF':=\CF'(M,X)$ that have the form
$$
Y:=G'\cdot e_Y\cong G'/(h\Gamma h\inv\cap G')
$$
 where
$$
X':=G'/\Gamma'\in\CF,\ \   h\in \Sigma,\ \ \text{ and } \ \
e_Y=h\cdot e_X.
$$

By Lemma~\ref{lemp:ap1}, there exists a positive constant
$H:=H(M,X)$ with the following properties:
\begin{enumerate}
\item
\label{eq:kappa1}  for every $h\in \Sigma$,  every $g\in G$  with
$d_G(g,\one_G)\leq M$,  and every  $x,y\in X$, we have  $d_X(gh\cdot
x,gh\cdot y)\leq H d_X(x,y)$\ ;
\item
\label{eq:transCm} \vide for every $h\in \Sigma$,   every $g\in G$
with $d_G(g,\one_G)\leq M$,  and every function $\phi\in
\CC^{2m}(X)$, we have $\norm{\phi_{gh}}_{\CC^{2m}(X)}\leq
H\norm\phi_{\CC^{2m}(X)}$ where $\phi_{gh}(x):=\phi(gh\cdot x)$.
\end{enumerate}

The distance on a nilmanifold $Y\in \CF'$ is not the distance
induced by the inclusion in $X$.
 However, the inclusion $Y\subset X$ is smooth and thus we can assume that
\begin{enumerate}
\setcounter{enumi}{2}
\item
\label{eq:LipXXprime} for every nilmanifold $Y\in\CF'$ and every
$x,y\in Y,$ we have $d_X(x,y)\leq Hd_{Y}(x,y)$;
\item
\label{eq:CmY} \vide for every $Y\in\CF'$ and every function $\phi$
on $X$, we have $\norm{\phi|\raise-2mm\hbox{$\scriptstyle
Y$}}_{\CC^{2m}(Y)}\leq H \norm\phi_{\CC^{2m}(X)}$.
\end{enumerate}

By Lemma~\ref{lem:equid_conjug}, for every $X'\in\CF$, every
$\zeta>0$, and every $h\in\Sigma$, there exists
$\rho:=\rho(M,
X,X',h,\zeta)$ with the following property:
\begin{enumerate}
\setcounter{enumi}{4}
\item
\label{eq:equidY} Let $X'=G'/\Gamma'\in\CF$, $h\in\Sigma$,
$e_Y:=h\cdot e_X$, and $(g'(n))_{n\in[N]}$  polynomial sequence in
$G'$ with degree at most $2$; if $\displaystyle
(g'(n)e_X)_{n\in[N]}$ is totally $\rho$-equidistributed in $X'$,
then $\displaystyle (g'(n) e_Y)_{n\in[N]}$ is totally
$\zeta$-equidistributed in $Y:=G'\cdot e_Y$.
\end{enumerate}

\subsubsection{Objects associated to $\CH$}
\label{subsec:all_nils}
We consider now all the
 nilmanifolds belonging to the family $\CH$ and define the finite
 family of nilmanifolds
$$
\CF'(M)=\bigcup_{X\in\CH}\CF'(M,X)
$$
 and the positive numbers
\begin{gather}
\label{E:def-K-rho}
H(M):=\max_{X\in\CH}H(M,X)\ ; \\
\label{E:def-K-rho'} \rho(M, \zeta) :=\min_{\substack{ X\in\CH\\
X'\in \CF(M,X),\ h\in \Sigma(M,X)}}\rho(M,X,X',h, \zeta),
\end{gather}
where $\rho(M,X,X',h, \zeta)$ was defined by item~\eqref{eq:equidY} above.

We let
\begin{gather}
\label{eq:chose_detla1}
\delta_1(M):=\frac{\delta^2}{5\,M^2}\ ;\\
\label{eq:chose_theta} \doublevide\theta(M):=\min\Big(\frac \delta
2\ , \frac{\delta_1(M)} {2\ c_1}\Big),
\end{gather}
where $\delta$ was defined by Section~\ref{subsec:beginning} and
$c_1$ is the universal  constant  defined by
Lemma~\ref{lem:U2-intervals}.

%

To every $\tau>0$ and every nilmanifold $Y$ in the finite collection
$\CF'(M)$, either Proposition~\ref{lem:equid-square} or
Proposition~\ref{lem:equid-square-ab}  (applied for $Y$ in place of
$X$) associates a positive number $\sigma(Y,\tau)$ (depending on
whether $Y$ in non-Abelian or Abelian). We define
\begin{equation}
\label{E:def-sigma} \wt{\sigma}(M)  := \min_{Y\in\CF'(M)}
\sigma\Big(Y\,,\,\frac{\delta_1(M)}{10\,H(M)^2}\Big).
\end{equation}
\subsubsection{The  function $\omega$ and bounds for $M$}
\label{SS:def-parameters} We   let  $\omega\colon \N\to \R^+$ be the
function defined by
\begin{equation}
\label{E:def-omega} \omega(M):=
\rho(M,\tilde{\sigma}(M))
\end{equation}
where $\rho$ is defined in \eqref{E:def-K-rho'}  and
$\tilde{\sigma}$ is defined in \eqref{E:def-sigma}. We also let
\begin{equation}
\label{eq:def_M0} M_0:= \lceil 2/\ve\rceil.
\end{equation}
For this choice of  $\omega$  and  $M_0$,
Corollary~\ref{cor:factorization1} associates to every nilmanifold
$X\in\CH$ a  positive real number $M_1(M_0,X,\omega)$ and we define
\begin{equation}
\label{eq:def-M1}
M_1:=\max_{X\in \CH}M_1(M_0,X,\omega).
\end{equation}
We stress that   $M_1$  depends only on $\ve$.
\subsection{Using the $U^2$-decomposition}\label{SS:Strategy}
After setting up the stage we are now ready to enter the main body
of the  proof of Theorem~\ref{th:strong-average-intro}.

Let $\ve>0$. Let $M_1$ be given by~\eqref{eq:def-M1} and let
\begin{equation}\label{E:theta1}
 \theta_1:=\theta(M_1)
\end{equation}
be given by~\eqref{eq:chose_theta}. Note that  $\theta_1$ depends on
$\ve$ only.
We use Theorem~\ref{th:Decomposition-weakU2-intro} for $\theta$
substituted with this ``very small'' value of $\theta_1$. We get
that for every sufficiently large $N$, depending only on $\ve$,
every
 $\chi\in\CM$   admits the decomposition
$$
\chi_N=\chi_{N,s}+\chi_{N,u}
$$
where $\chi_{N,s}$ and $\chi_{N,u}$ satisfy the conclusions of
Theorem~\ref{th:Decomposition-weakU2-intro};  in particular
\begin{equation}
\label{eq:U2_Chi_u} \norm{\chi_{N,u}}_{U^2(\Z_\tN)}\leq\theta_1.
\end{equation}

We claim that $\chi_{N,s}$ and $\chi_{N,u}$ satisfy the conclusion
of Theorem~\ref{th:Decomposition-weakU3-intro}.
Theorem~\ref{th:Decomposition-weakU2-intro} gives at once that
Properties~\eqref{it:weakU3-1} and \eqref{it:weakU3-2} of
Theorem~\ref{th:Decomposition-weakU3-intro} are satisfied with
$Q:=Q(M_1)$ and $R:=R(M_1)$. It remains to verify
Property~\eqref{it:weakU3_3}, namely that
$$
\norm{\chi_{N,u}}_{U^3(\Z_\tN)}\leq\ve.
$$
We argue by contradiction. We assume  that
\begin{equation}\label{eq:U3_large}
\norm{\chi_{N,u}}_{U^3(\Z_\tN)}>\ve,
\end{equation}
and in the next subsections we are going  to derive a contradiction.
To facilitate reading we split the proof into several parts, and  to
ease notation,
we continue to leave  the dependence on $\ve, \ell$ implicit.

\subsection{Using the inverse and the factorization
theorem}\label{SS:InvFact}
 Suppose that~\eqref{eq:U3_large} holds. We recall that in Section~\ref{subsec:beginning}
we used Corollary~\ref{cor:inverse}  to define the positive integer
$m$, the positive real number   $\delta$, and a finite family $\CH$
of nilmanifolds with vertical torus of dimension $1$. We recall also
that these objects depend only on $\ve$ and  we assume that $N$ is
sufficiently large so that the conclusions of
Corollary~\ref{cor:inverse} hold.

If the alternative~\eqref{it:inverse1} of
Corollary~\ref{cor:inverse} holds, then
$$
\norm{\chi_{N,u}}_{U^2(\Z_\tN)}\geq \delta,
$$
and since by  \eqref{eq:chose_theta} and \eqref{E:theta1}  we have
$\theta_1=\theta(M_1)\leq \delta/2$,  this
contradicts~\eqref{eq:U2_Chi_u}.

 As a consequence, alternative~\eqref{it:inverse2}
of Corollary~\ref{cor:inverse} holds. Our goal is  to show that for
the particular choices made in the previous subsections we get again
a contradiction.

 By our assumption, there exists a
$2$-step nilmanifold $X=G/\Gamma$ belonging to the family $\CH$, a
  nilcharacter $\Psi$ on $X$ with frequency $1$, and an element $g$
of $G$, such that
\begin{align}
\label{eq:PsiCm}
 & \norm\Psi_{\CC^{2m}(X)}\leq 1,\ \ \text{and}   \\
\label{eq:bound1} \quad &
 \big|\E_{n\in[\tN]}\chi_{N,u}(n)\Psi(g^n\cdot e_X) \big|\geq\delta
\end{align}
where $m$ is the dimension of $X$. Note that $X$, $\Psi$, and  $g$,
will depend on $\chi$, but this is not going to create problems for
us.

Recall that for every $M\in\N$ the finite family $\CF(M,X)$ of
sub-nilmanifolds of $X$ was defined   in
Section~\ref{subsec:new_objects}  using
Corollary~\ref{cor:factorization1} of the factorization theorem.
Next we apply this corollary  for the  sequence
$(g^n)_{n\in [\tN]}$ in $G$, the function $\omega\colon\N\to\R^+$
defined by~\eqref{E:def-omega}, and
 the integer  $M_0$ defined by~\eqref{eq:def_M0}.
For $M_1$ given by \eqref{eq:def-M1} we get   an integer   $M$ with
\begin{equation}\label{E:def-M_1}
   M_0\leq M\leq M_1,
\end{equation}
a nilmanifold $X'=G'/\Gamma'$  belonging to the family $\CF(M,X)$
that satisfies either Property~\eqref{it:Gprimeabelien} or Property~\eqref{it:Gprimenonabelien} of Corollary~\ref{cor:factorization1},
and a factorization
\begin{equation}
\label{eq:decomp_gn}
g^n=\epsilon(n)g'(n)\gamma(n)
\end{equation}
into polynomial sequences in $G$ that satisfy
Properties~\eqref{it:dec3}, \eqref{it:dec4}, \eqref{it:dec5} of
Corollary~\ref{cor:factorization1} for  this value of $M$. The
number  $M$  depends on $\chi$, but it belongs on the interval
$[M_0,M_1]$, and since the integers $M_0, M_1$ depend only on $\ve$,
this suffices for our purposes.

From this point on, we work with this choice of $M$.

\subsection{Eliminating $\epsilon$ and $\gamma$ by passing to a
sub-progression}\label{SS:subprogression} Our goal on this and the
next subsection is to get an estimate of the form \eqref{eq:bound1}
with the additional property that the
 sequence $(g^n\cdot e_X)_{n\in [\tN]}$ is sufficiently totally
equidistributed. To achieve this we  pass to an appropriate
sub-progression where the sequences $\epsilon$ and $\gamma$ defined
by \eqref{eq:decomp_gn} are practically constant and then change
the nilmanifold defining the nilsequence to eliminate some extra
terms introduced.

By Property~\eqref{it:dec5} of Corollary~\ref{cor:factorization1},
 the sequence $(\gamma(n))$ is periodic of period at most $M$; we denote its
 period by $p$. Let
\begin{equation}\label{E:L}
L:= \Big\lfloor\frac\delta{2\,M^2}\,\wt N\Big\rfloor.
\end{equation}
We partition $[\wt N]$ into arithmetic progressions of step $p$ and
length $L$ and a leftover set that we can ignore since it will
introduce error terms bounded by a constant times $\delta/M$ and
thus negligible for our purposes (upon replacing $\delta$ with
$\delta/2$ below). Using \eqref{eq:bound1} and the pigeonhole
principle,  we get that there exists a progression $P$ with step $p$
and length $L$
 such that
\begin{equation}
\label{eq:bound1prime} \big|\E_{n\in [\tN]}
\one_{P}(n)\chi_{N,u}(n)\Psi(g^n\cdot e_X) \big|\geq\delta \,\frac
L{\wt N}= \frac{\delta^2}{2\,M^2}-\frac{\delta}{\tN}.
\end{equation}

Let $n_0,n\in P$.  We have $\gamma(n)=\gamma(n_0)$ and   by
Property~\eqref{it:dec3} of Corollary~\ref{cor:factorization1} the
sequence $(\epsilon(n))$ is $(M,\wt N)$-smooth.  Using these
properties and the right invariance of the metric $d_G$, we get
(recall that $g^n=\epsilon(n)g'(n)\gamma(n)$)
$$
d_G(g^n,\epsilon(n_0)g'(n)\gamma(n_0))\leq
d_G(\epsilon(n),\epsilon(n_0)) \leq |n-n_0|\,\frac M{\wt N}\leq pL\,\frac M{\wt N}
\leq \frac{M^2L}{\wt N}.
$$
Since by \eqref{eq:PsiCm} the function  $\Psi$ has Lipschitz
constant at most $1$, it follows that
$$
\bigl|\Psi(g^n\cdot e_X) - \Psi(\epsilon(n_0)g'(n)\gamma(n_0)\cdot e_X)\bigr|
\leq \frac{M^2L}{\wt N}.
$$
From this and $\eqref{E:L}$ we deduce that
$$
\E_{n\in [\tN]} \one_{P}(n)|\,\chi_{N,u}(n)|\,\bigl|\Psi(g^n\cdot
e_X)-\Psi(\epsilon(n_0)g'(n)\gamma(n_0)\cdot e_X) \bigr|\leq \frac
L{\wt N}\, \frac{M^2L}{\wt N}\leq\frac{\delta^2}{4\,M^2}.
$$

Combining this with  \eqref{eq:bound1prime} and the definition of $\delta_1$ given by \eqref{eq:chose_detla1} we get
\begin{equation}
\label{eq:bound2} \bigl|\E_{n\in [\tN]}
\one_{P}(n)\chi_{N,u}(n)\Psi\bigl(\epsilon(n_0)
g'(n)\gamma(n_0)\cdot e_X\bigr)\bigr|\geq \frac{\delta^2}{5\,
M^2}=\delta_1(M)
\end{equation}
 provided that
$
\tN\geq 20M_1^2/\delta$ (then $\delta/\tN\leq \delta^2/(20M^2)$).

\subsection{Changing the nilmanifold}
By Property~\eqref{it:dec5} of Corollary~\ref{cor:factorization1},
$\gamma(n_0)$ is $M$-rational. By the definition of   $\Sigma(M,X)$
in Section~\ref{subsec:new_objects}, we can choose $h_0\in
\Sigma(M,X)$ such that $\gamma(n_0)\cdot e_X=h_0\cdot e_X$. We
define
\begin{gather*}
e_Y:=h_0\cdot e_X\  ;\quad
Y:=G'\cdot e_Y\cong G'/(h_0\Gamma h_0\inv\cap G')\ ;\\
\Psi'(x)= \Psi(\epsilon(n_0)\cdot x).
\end{gather*}
Note that the nilmanifold $Y$ belongs to the family $\CF'(M)$
defined in Section~\ref{subsec:all_nils}. For every $n\in \N$ we
have
\begin{equation}
\label{eq:Psisecond} \Psi\bigl(\epsilon(n_0)g'(n)\gamma(n_0)\cdot
e_X\bigr)=\Psi'(g'(n)\cdot e_Y).
\end{equation}

Since the sequence $(\epsilon(n))_{n\in [\tN]}$ is $(M,\tN)$-smooth,
we have $d_G(\epsilon(n_0),\one_G)\leq M$. Furthermore, since
$\norm{\Psi}_{\CC^{2m}(X)}\leq 1$, by Property \eqref{eq:transCm} of
Section~\ref{subsec:new_objects} we have
$\norm{\Psi'}_{\CC^{2m}(X)}\leq H(M)$ where $H(M)$ was defined by
\eqref{E:def-K-rho}. Since the nilmanifold $Y$
belongs to the family $\CF'(M)$,
 by Property~\eqref{eq:CmY} of Section~\ref{subsec:new_objects}
 we have
\begin{equation}
\label{eq:CmPsiPrime'} \norm{\Psi'|\raise-2mm\hbox{$\scriptstyle
Y$}}_{\CC^{2m}(Y)}\leq H(M)^2.
\end{equation}
Recall that the sequence $(g'(n)\cdot e_X)_{n\in[\tN]}$ arises from the
decomposition~\eqref{eq:decomp_gn} provided by Corollary~\ref{cor:factorization1}; by Property~\eqref{it:dec4} of this corollary,
 the sequence $(g'(n)\cdot e_X)_{n\in[\tN]}$ is
$\omega(M)$-totally equidistributed in $X'$.

Since by the definition  of $\omega$ (see \eqref{E:def-omega}) we
have  $\omega(M)= \rho(M,\tilde{\sigma}(M))$,
Property~\eqref{eq:equidY} of Section~\ref{subsec:new_objects} and
\eqref{E:def-K-rho'}  give that
\begin{equation}
\label{eq:equid_Y} \text{the \  sequence }
\  (g'(n)\cdot e_Y)_{n\in
[\tN]} \ \text{ is totally }\  \tilde{\sigma}(M)\text{-equidistributed in }\ Y.
\end{equation}
Summarizing, we have so far established that
\begin{equation}
\label{eq:bound2'} \bigl|\E_{n\in [\tN]}\one_{P}(n)
\chi_{N,u}(n)\Psi'\bigl(g'(n)\cdot e_Y\bigr)\bigr|\geq \delta_1(M)
\end{equation}
for some arithmetic progression $P\subset [\tN]$ and properties
\eqref{eq:CmPsiPrime'} and \eqref{eq:equid_Y} are satisfied. In the
next subsection we further reduce matters to the case where the function
$\Psi'$ has integral zero.
\subsection{Reducing to the zero integral case}\label{SS:zero}
Our goal is to show that upon replacing   $\Psi'$ with $\Psi'-z$,
where $z=\int_Y \Psi' dm_Y$, we get a bound similar to
\eqref{eq:bound2'}. To do this we make crucial use of the fact that
the $U^2$-norm of $\chi_{N,u}$ is suitably small, in fact, this is
the step that determined our choice of the degree of
$U^2$-uniformity $\theta_1$ of $\chi_{N,u}$.
 We write
$$
z=\int_{Y}\Psi'\,dm_{Y}\ \ \text{ and }\ \ \Psi''=\Psi'-z.
$$
Combining Lemma~\ref{lem:U2-intervals},
equations~\eqref{eq:chose_theta}, \eqref{E:theta1}, and
estimate~\eqref{eq:U2_Chi_u} we get
$$
 \big|\E_{n\in [\tN]}
\one_{P}(n) z \chi_{N,u}(n) \big|\leq
 c_1\norm{\chi_{N,u}}_{U^2(\Z_\tN)}\leq c_1\theta_1=c_1\theta(M_1)\leq
\frac 12\delta_1(M_1)\leq\frac 12\delta_1(M)
$$
where the last estimate follows from \eqref{eq:chose_detla1}  and the fact that
$M\leq M_1$.
From this estimate and  \eqref{eq:bound2'} we deduce that
\begin{equation}
\label{eq:bound3} \bigl| \E_{n\in [\tN]}\one_{P}(n)\chi_{N,u}(n)
\Psi''(g'(n)\cdot e_{Y})\bigr|\geq \frac 12\delta_1(M)
\end{equation}
where
\begin{equation} \label{eq:CmPsiPrime}
\norm{\Psi''|\raise-2mm\hbox{$\scriptstyle Y$}}_{\CC^{2m}(Y)}\leq
H(M)^2\ \text{ and } \ \int_{Y}\Psi''\,dm_{Y}=0.
\end{equation}
We recall that Property~\eqref{eq:equid_Y} is also satisfied.
\subsection{Proof of the weak $U^3$-decomposition
result}\label{SS:ProofFinally} We are now very close to completing
the proof of Theorem~\ref{th:Decomposition-weakU3-intro}. To this
end, we are going to   combine \eqref{eq:bound3},
 the total equidistribution of the sequence $(g'(n)\cdot e_{Y})_{n\in [\tN]}$, and
 Propositions~\ref{lem:equid-square} and  \ref{lem:equid-square-ab}, to deduce a contradiction.

 Recall that  $\chi_{N,s}=\chi_N*\phi$ (the convolution is taken in
$\Z_{\tN}$) where $\phi$ is a kernel in $\Z_\tN$, meaning a
non-negative function with $\E_{n\in\ZN}\phi(n)=1$. Since
$\chi_u=\chi_N-\chi_s$, we can write $\chi_{N,u}=\chi_N*\psi$, where
$\psi$ is a function on $\ZN$ with $\E_{n\in\ZN}|\psi(n)|\leq 2$.
 We deduce from \eqref{eq:bound3}   that there
exists an integer $q$ with $0\leq q < \tN$ such that
\begin{equation}
\label{eq:bound5}   \bigl| \E_{n\in [\wt N]} \one_{P}(n+q\bmod \wt N)\chi_N(n)
\Psi''(g'(n+q\bmod{\tN})\cdot e_{Y})\bigr|\geq \frac 14\, \delta_1(M)
\end{equation}
where the residue class $n+q\bmod{\tN}$ is taken in $[\wt N]$
instead of the more usual interval $[0,\wt N)$. It follows that
$$
 \bigl| \E_{n\in [\wt N]} \one_{P}(n+m)\,\one_J(n)\,
\one_{[N]}(n)\, \chi(n) \Psi''(g'(n+m)\cdot
e_{Y})\bigr|\geq \frac 18\, \delta_1(M)
$$
where either $J$ is the interval $[\wt N-q]$ and $m=q$, or  $J$ is
the interval $(\wt N-q,\wt N]$ and $m=q-\wt N$.   We remark that in
both cases  $\one_{P}(n+m)\,\one_J(n)\,\one_{[N]}(n)=\one_{P_1}(n)$
for some arithmetic progression $P_1\subset[N]$ and thus we have
\begin{equation}
\label{eq:bound5'} \bigl| \E_{n\in [\wt N]} \one_{P_1}(n)\, \chi(n)
\Psi''(g'(n+m)\cdot e_{Y})\bigr|\geq \frac 18\, \delta_1(M).
\end{equation}

Recall that $\norm{\Psi''|\raise-2mm\hbox{$\scriptstyle
Y$}}_{\CC^{2m}(Y)}\leq H(M)^2$. By~\eqref{eq:equid_Y}, the sequence
$(g'(n)\cdot e_Y)_{n\in [\tN]}$ is totally
$\tilde{\sigma}(M)$-equidistributed in $Y$. We remark that,
depending on the case~\eqref{it:Gprimenonabelien}
or~\eqref{it:Gprimeabelien} of Corollary~\ref{cor:factorization1},
Proposition~\ref{lem:equid-square} or
Proposition~\ref{lem:equid-square-ab} correspondingly can be applied
to the nilmanifold $Y$ and the sequence $(g'(n))_{n\in [\tN]}$.
Since by the definition of $\tilde{\sigma}(M)$ (see
\eqref{E:def-sigma}) we have
$$
\tilde{\sigma}(M)\leq\sigma\Bigl(Y,\frac{\delta_1(M)}{10\, H(M)^2}\Bigr),
$$ by using one of the
two propositions we deduce that
\begin{multline*}
 \big|\E_{n\in [\wt N]}\one_{P_1}(n)\chi(n)\Psi'(g'(n+m)\cdot e_Y) \big|\leq \\
 \big|\E_{n\in [N]}\one_{P_1}(n)\chi(n)\Psi'(g'(n+m)\cdot e_Y) \big| \leq
\frac{\delta_1(M)}{10\, H(M)^2}\, \norm{\Psi''|\raise-2mm\hbox{$\scriptstyle
Y$}}_{\CC^{2m}(Y)} \leq \frac 1{10}\delta_1(M)
\end{multline*}
which contradicts~\eqref{eq:bound5}. Hence, \eqref{eq:U3_large}
fails, giving us that $\norm{\chi_{N,u}}_{U_3(\Z_\tN)}\leq \ve$.
This completes the proof of
Theorem~\ref{th:Decomposition-weakU3-intro}. \qed

\subsection{Proof of the strong $U^3$-decomposition on the average}
\label{subsec:proof_strong}
 In this
  subsection we prove  Theorem~\ref{th:strong-average-intro}.
Our basic ingredient is the weak decomposition result of
Theorem~\ref{th:Decomposition-weakU3-intro}. We use it in an
iterative way in an argument of energy increment; this is made
possible because of the simple structure of $\chi_{N,s}$ and   the
particular form of the kernels $\phi_{N,\theta}$ that occur in
Theorem~\ref{th:Decomposition-weakU3-intro}. We recall that these
kernels were defined by~\eqref{eq:def-phi} in
Section~\ref{subsec:kernels}; the most important property  used in
the subsequent argument is the monotonicity of their Fourier
coefficients:
$$
\text{if }\ 0<\theta'<\theta, \ \text{ then }\
\widehat{\phi_{N,\theta'}}(\xi)\geq\widehat{\phi_{N,\theta}}(\xi)\geq
0\ \text{ for every }\ \xi\in \Z_{\tN}.
$$

We fix a function $F\colon \N \times \N \times \R^+\to \R^+$, an
  $\ve>0$, and a finite positive measure $\nu$ on the compact group
$\CM$ of multiplicative functions. We can assume without loss of
generality that $\nu$ is a probability measure.

We define inductively  a sequence  $(\theta_j)$ of positive reals
and sequences $(N_j)$,  $(Q_j)$, $(R_j)$ of positive integers as
follows.
 We let $\theta_1=N_1=Q_1=R_1=1$. Suppose
 that $j\geq 1$ and that the first $j$ terms of the sequences are defined.
 We apply
Theorem~\ref{th:Decomposition-weakU3-intro} with
$$
\theta_{j} \ \text{ substituted for } \ \theta_0\quad \text{ and } \quad  \frac
1{F(Q_j,R_{j},\ve)} \ \text{ substituted for } \ \ve.
$$
Let $N_{j+1}$ be such that the conclusions of
Theorem~\ref{th:Decomposition-weakU3-intro} hold with these data as
input for every $N\geq N_{j+1}$. We get  positive integers $Q$ and
$R$ that we write respectively as  $Q_{j+1}$ and $R_{j+1}$, and a
real $\theta_{j+1}$ with $0<\theta_{j+1}<\theta_{j}$, such that for
every $N\geq N_{j+1}$ and every $\chi\in\CM$, the functions
$$
\chi_{j+1,N,s}:=\chi_N*\phi_{N,\theta_{j+1}}\quad \text{ and } \quad
\chi_{j+1,N,u}:=\chi_N-\chi_{j+1,N,s}
$$
 satisfy the Properties~\eqref{it:weakU3-2} and~\eqref{it:weakU3_3}
 of Theorem~\ref{th:Decomposition-weakU3-intro}, i.e.,
\begin{gather}
\label{eq:decompU32}
 |\chi_{j+1,N,s}(n+Q_{j+1})-\chi_{j+1,N,s}(n)|\leq \frac {R_{j+1}} \tN\quad \text{ for every }\ n\in
\Z_\tN\ ;\\
\label{eq:decompU33} \norm{\chi_{j+1,N,u}}_{U^3(\ZN)}\leq \frac
1{F(Q_j,R_{j},\ve)}.
\end{gather}
Replacing $N_{j+1}$ with $\max_{i\leq j}N_i$ we can assume that the
sequence $(N_j)$ is increasing. By construction, the sequence
$(\theta_j)$ is decreasing.

Let $J=1+\lceil 2\ve^{-2}\rceil$ and let $N_0=N_{J+1}$. For every
$N\geq N_0$ we have
\begin{multline*}
\sum_{j=2}^J\int_\CM
\norm{\chi_{j+1,N,s}-\chi_{j,N,s}}_{L^2(\ZN)}^2\,d\nu(\chi)=\\
\int_\CM \sum_{\xi\in\ZN} |\widehat{\chi_N}(\xi)|^2 \,\sum_{j=2}^J
|\widehat{\phi_{N,\theta_{j+1}}}(\xi)-\widehat{\phi_{N,\theta_{j}}}(\xi)|^2\,d\nu(\chi)
\leq\\
 2 \int_\CM \sum_{\xi\in\ZN}|\widehat{\chi_N}(\xi)|^2
\,\sum_{j=2}^J
\bigl(\widehat{\phi_{N,\theta_{j+1}}}(\xi)-\widehat{\phi_{N,\theta_j}}(\xi)\bigr)\,d\nu(\chi),
\end{multline*}
where to get the last estimate we used that  $\theta_{N,j+1} \leq
\theta_{N,j}$ and thus
$\widehat{\phi_{N,\theta_{j+1}}}(\xi)\geq\widehat{\phi_{N,\theta_j}}(\xi)\geq
0$ for every $\xi$ by~\eqref{eq:phi-increases}. Since
$|\widehat{\phi_{N,\theta}}(\xi)|\leq 1$, the last quantity in the
estimate is  at most
$$
 2\int_\CM \sum_{\xi\in \ZN} |\widehat{\chi_N}(\xi)|^2 \, d\nu(\chi)\leq 2.
$$
Therefore,  for every $N\geq N_0$ there exists $j_0:=j_0(F,N,\ve,
\nu)$ with
\begin{equation}\label{E:j0isbounded}
2\leq j_0\leq J
\end{equation} such that
\begin{equation}
\label{eq:decompU34} \int_\CM
\norm{\chi_{j_0+1,N,s}-\chi_{j_0,N,s}}_{L^2(\ZN)}^2\,d\nu(\chi)\leq
\frac 2{J-1}\leq \ve^2.
\end{equation}
 For   $N\geq N_0$, we  define
\begin{gather*}
\psi_{N,1}:= \phi_{N,\theta_{j_0}}\ ;\quad \psi_{N,2}:=\phi_{N,\theta_{j_0+1}}\ ;\\
 \chi_{N,s}:=\chi_N*\psi_{N,1}=\chi_{j_0,N,s}\ ;\quad
 \chi_{N,u}:=\chi_N-\chi_N*\psi_{N,2}= \chi_{j_0+1,N,u}\ ;\\
\chi_{N,e}:=\chi_N*(\psi_{N,2}-\psi_{N,1})= \chi_{j_0+1,N,s}-\chi_{j_0,N,s}\ ;\\
Q:=Q_{j_0}\ \text{ and }\ R:=R_{j_0}.
\end{gather*}
Then we have the decomposition
$$
\chi_N=\chi_{N,s}+\chi_{N,u} +\chi_{N,e},
$$
and furthermore, Property~\eqref{it:decompU32} of Theorem~\ref{th:strong-average-intro} follows from~\eqref{eq:decompU32} (applied for $j=j_0-1$),
 Property~\eqref{it:decomU33} follows from~\eqref{eq:decompU33} (applied for $j=j_0$),  and Property~\eqref{it:decomU34}
 follows from~\eqref{eq:decompU34} and the Cauchy-Schwarz
 estimate. Furthermore, it follows from \eqref{E:j0isbounded} that the integers $N_0,Q,R$  are bounded by a  constant that depends on $F$ and $\ve$ only.
 Thus, all the announced properties are satisfied,
completing the proof of Theorem~\ref{th:strong-average-intro}.
\qed

 \appendix
\section{Rational elements in a  nilmanifold}
\label{ap:A}
In this section, working with  $2$-step nilmanifolds does
not provide any simplification and so the results are stated for
general nilmanifolds.

Let $X=G/\Gamma$ be an $s$-step nilmanifold of dimension $m$. As
everywhere in this article we assume that $G$ is connected and
simply connected, and endowed with a Mal'cev basis. Recall that we
write $e_X$ for the image in $X$ of the unit element  $\one_G$ of
$G$.

From  Property~\eqref{it:Malcev-Gamma} of Mal'cev bases stated in
Section~\ref{subsec:nilmanifolds} we immediately deduce:
\begin{lemma}
\label{lem:finite-Gamma}
 $\Gamma$ is a finitely generated group.
\end{lemma}

\subsection{Rational elements}
We recall that an element $g\in G$ is $Q$-rational if
$g^m\in\Gamma$ for some  $m\in \N$ with $ m\leq Q$.
 We collect here some properties of rational elements. We note that all quantities  introduced  below depend implicitly on the nilmanifold $X$.
\begin{lemma}[{\cite[Lemma A.11]{GT12a}}]
\label{lem:rational}
\begin{enumerate}
\item
\label{it:rat1} For every $Q\in\N$ there exists $Q_1\in\N$ such that
the product of any two $Q$-rational elements is $Q_1$-rational; it
follows  that the set of rational elements is a subgroup of $G$.

\item
\label{it:rat2}
For every $Q\in \N$ there exists $q\in\N$ such that the Mal'cev coordinates of any $Q$-rational element are rational with denominators at most $q$; it follows that the set of $Q$-rational elements is a discrete subset of $G$.
\item
\label{it:rat3}
Conversely, for every $q\in\N$ there exists $Q\in\N$ such that, if the  Mal'cev coordinates of $g\in G$  are rational with denominators at most $q$, then $g$ is $Q$-rational.
\end{enumerate}
\end{lemma}

\begin{corollary}
\label{cor:M-rat}
For every $Q\in\N$  there exists a finite set $\Sigma:=\Sigma(Q)$ of $Q$-rational
 elements such that  all $Q$-rational elements belong to $\Sigma(Q)\Gamma$.
\end{corollary}

\begin{proof}
Let $K$ be a compact subset of $G$ such that $G=K\Gamma$.

Let $Q\in\N$. Let $Q_1$ be associated to $Q$ by Part~\eqref{it:rat1}
of Lemma~\ref{lem:rational}, and let $\Sigma_1$ be the set of
$Q_1$-rational elements of $K$.  By Part~\eqref{it:rat2} of
Lemma~\ref{lem:rational}, $\Sigma_1$ is finite. Let $g$ be a
$Q$-rational element of $G$. There exists $\gamma\in\Gamma$ such
that $g\gamma\inv\in K$.  Since $\gamma$ is obviously $Q$-rational,
$g\gamma\inv$ is $Q_1$-rational and thus belongs to $\Sigma_1$. For
each element $h$ of $\Sigma_1$ obtained this way we choose a
$Q$-rational point $g$ such that $h\in g\Gamma$. Let
$\Sigma:=\Sigma(Q)$ be the set consisting of all elements obtained
this way. We have that $\Sigma\Gamma$ contains all $Q$-rational
elements.  Furthermore,  $|\Sigma|\leq|\Sigma_1|$ and so $\Sigma$ is
finite, completing the proof.
 \end{proof}

\subsection{Rational subgroups} We gather here some basic properties of rational
subgroups that we use in the main part of the article.
\begin{definition}
A \emph{rational subgroup} $G'$ of $G$ is a closed and  connected subgroup
of $G$ such that its Lie algebra $\mathfrak g'$ admits a base
that has  rational coordinates in the Mal'cev basis of $G$.
\end{definition}
An equivalent definition is that $\Gamma':=\Gamma\cap G'$ is
co-compact in $G$. In this case, $G'/\Gamma'$ is  called a
\emph{sub-nilmanifold} of $X$.
\begin{lemma}[\cite{GT12a}, Lemma A.13]
\label{lem:Ap0} If $G'$ is a rational subgroup of $G$ and $h$ is a
rational element, then $hG'h\inv$ is a rational subgroup of $G$.
\end{lemma}
\begin{proof}
The conjugacy map $h\mapsto g\inv hg$ is a polynomial map with
rational coefficients and thus the linear map $\mathrm{Ad}_h$ from
$\mathfrak g$ to itself has rational coefficients. Since $\mathfrak
g'$ has a base consisting of vectors with rational coefficients, the
same property holds for $\mathrm{Ad}_h\mathfrak g$, that is, for the
Lie algebra of $hG'h\inv$. This proves the claim.
\end{proof}

 The argument used to deduce  Lemma~\ref{lem:finite-Gamma} shows that
  the group $\Gamma\cap
hG'h\inv$ is finitely generated.

\begin{lemma}
\label{lem:finite-index} Let $X=G/\Gamma$ be an $s$-step
nilmanifold, $G'\subset G$ be a rational subgroup, and $g\in G$ be a
rational element. Then
\begin{enumerate}
\item
\label{it:rat10bis}
$\Gamma\cap g\inv \Gamma g\cap G'$ is a subgroup of finite index  of
$\Gamma\cap G'$;
\item
\label{it:rat11bis}
 $\Gamma\cap g\inv \Gamma g\cap
G'$ is a subgroup of finite index  of $g\inv \Gamma g\cap G'$.
\end{enumerate}
\end{lemma}
\begin{proof}
By Part~\eqref{it:rat1} of Lemma~\ref{lem:rational},  all  elements
of $g\Gamma g\inv$  are rational.
 Hence, if $\gamma\in \Gamma\cap G'$ there exists $k\in\N$ with
$(g\gamma g\inv)^k\in \Gamma$ and so we have $\gamma^k\in g\inv\Gamma g\cap \Gamma\cap G'$.
Applying Lemma~\ref{lem:finite-Gamma} to $G'$ and $\Gamma\cap G'$ we get that this last group is finitely generated. By induction on $s$ it is easy to deduce that $g\inv\Gamma g\cap \Gamma\cap G'$ has finite index in $\Gamma\cap G'$. This proves \eqref{it:rat10bis}.
Since $g\Gamma g\inv$  is co-compact in $G$,
substituting this group for $G$ and $g\inv$ for $g$ in the preceding
statement, we get  \eqref{it:rat11bis}.
\end{proof}

\begin{lemma}
\label{lem:Gprimey} Let $g\in G$ be a rational element and  $G'$ a
rational subgroup of $G$. Then $G'g\cdot e_X:=\{hg\cdot e_X\colon
h\in G'\}$ is a closed sub-nilmanifold of $X$.
\end{lemma}
\begin{proof}
By Lemma~\ref{lem:Ap0}, $g\inv G'g$ is a rational subgroup of $G$.
Therefore, $\Gamma\cap g\inv G'g$ is co-compact in $g\inv G'g$ and
thus $g\Gamma g\inv\cap G'$ is co-compact in $G'$. But $g\Gamma
g\inv\cap G'$ is the stabilizer $\{h\in G'\colon h\cdot g\cdot
e_X=g\cdot e_X\}$ of $g\cdot e_X$ in $G'$ and thus the orbit
$G'\cdot(g\cdot e_X)$ is compact and can be identified with the
nilmanifold $G'/(g\Gamma g\inv\cap G')$.
\end{proof}

\section{Proof of Proposition~\ref{lem:equid-square}}

\subsection{The problem}
We remind the reader of our setup.
We are given a $2$-step nilmanifold $X=G/\Gamma$   with $G$ connected and simply connected, where for the context of this proof, the groups $G$ and $\Gamma$ are determined uniquely given $X$.
We are also given a polynomial sequence $(g(n))$ of the form
$$
g(n)=g_0g_1^ng_2^{\binom n2}\ \text{ for some }\ g_0,g_1\in G\ \text{ and }\ g_2\in G_2,
$$
and a function $\Phi\colon X\to \C$   with
$$
\norm{\Phi}_{\text{Lip}(X)}\leq 1
\quad \text{and}  \quad \int \Phi\, dm_X=0.
$$
 We assume that
 $\chi\in \CM$ is a multiplicative function such that
\begin{equation}\label{E:lower1}
|\E_{n\in[N]} \one_P(n)
\chi(n)\Phi(g(m+n)\cdot e_X)|>\tau
\end{equation}
for some $m\in \N$ with $|m|\leq N$, $\tau>0$, $N\in \N$,
and  arithmetic
progression $P\subset[N]$.
Our goal is to show  that there exists $\sigma:=\sigma(X, \tau)$ and $N_0:=N_0(X,\tau)$ such that
 if \eqref{E:lower1} holds for some $N\geq N_0$, then
\begin{equation}
\label{eq:conclusion}
\text{the sequence} \  \ (g(n)e_X)_{n\in [N]} \
\text{ is not } \   \sigma\text{-equidistributed in } \  X.
\end{equation}
\begin{convention}
In this proof, by a constant we mean a quantity depending only on $X$ and $\tau$.
A quantity is said to be bounded if it is bounded by a constant of this type.
We write $\tau_1,\tau_2,\dots$, $\sigma_1,\sigma_2,\dots$, and $C_1,C_2,\dots$ for constants of this type.
\end{convention}

\subsection{Step 1: Reduction to some particular nilmanifold}
We first proceed to the same reduction as in the proof of Corollary~\ref{cor:inverse}. We write $r:=\dim(G_2)$, $r+s:=\dim(G)$,  and identify the vertical torus $G_2/(G_2\cap\Gamma)$ with $\T^r$. Using \eqref{E:lower1}  we find $\bk\in\Z^r=\widehat{\T^r}$ and a nilcharacter $\Phi_\bk$, defined as in Corollary~\eqref{cor:inverse}, with
$$
\norm\bk\leq C_1,\
\norm{\Phi_\bk}\leq 1,\ \text{ and }\
|\E_{n\in[N]} \one_P(m+n)
\chi(n)\Phi_\bk(g(n+m)\cdot e_X)|\geq\tau_1
$$
for some constants $C_1>0$ and $\tau_1>0$.  If $\bk=0$, then we can conclude by using Proposition~\ref{lem:equid-square-ab}. We assume now that $\bk\neq 0$, and  continue as in the proof of  Corollary~\ref{cor:inverse} to reduce to the case where
\begin{gather}
\notag
r=\dim(G_2)=1;\\
\notag
\Phi\text{ is a nilcharacter of frequency $1$; }\\
\notag
\norm\Phi_{\lip(X)}\leq 1;\\
\label{E:lower2}
|\E_{n\in[N]} \one_P(n) \chi(n)\Phi(g(n+m)\cdot e_X)|\geq\tau_2
\end{gather}
for some constant $\tau_2>0$.

\subsection{Step 2: Reduction to $g_0=\text{id}_G$  and $m=0$.}
We proceed to  some further reductions.
 Suppose that the conclusion~\eqref{eq:conclusion}
 holds  for some $N_0$ and $\sigma$,
  under the stronger assumption that the hypothesis~\eqref{E:lower1}
 holds for $m=0$ and a sequence $(g(n))$ given by
$$
g(n)=g_1^ng_2^{\binom n2}\quad \text{ where }g_1\in G\ \text{ and }\ g_2\in G_2.
$$

 Let $\tau>0$ and $N\geq N_0$. Let $F_1\subset G$ be a bounded fundamental domain of the projection $G\to X$ (we assume that $F_1$ is fixed given $X$).  By the first statement of Lemma~\ref{lemp:ap1} there exists a constant $C_2>0$ such that
\begin{equation}
\label{E:C2}
d_X(g\cdot x,g\cdot x')\leq C_2 d_X(x,x')\  \text{ for every }g\in F_1 \text{ and for all }x,x'\in X.
\end{equation}

Given $g_0, g_1\in G$, $g_2\in G_2$ and   $m\in \N$,  we write
\begin{gather*}
g_0g_1^mg_2^{\binom m2}=a_m\gamma_m\ \text{ where }\ a_m \in F\ \text{ and }\ \gamma_m\in\Gamma;\\
g_{m,1}:=\gamma_m g_1g_2^m\gamma_m\inv.
\end{gather*}
Then for $n\in \N$ we have
$$
g(n+m)\cdot e_X=a_mg_{m,1}^{n}g_2^{\binom {n}2}\cdot e_X.
$$
We let
$$
\Phi_m(x):=\Phi(a_m\cdot x).
$$
For every $m\in \N$ with $|m|\leq N$, $\Phi_m$ is a nilcharacter of frequency $1$.  Since $a_m$ belongs to  $F_1$ for every $m\in \N$ and $\norm{\Phi}_{\lip}\leq 1$,  we get by~\eqref{E:C2} that  $\norm{\Phi_m}_{\lip(X)}\leq C_2$. Estimate \eqref{E:lower2} can be rewritten as
$$
|\E_{n\in[N]} \one_P(n)
\chi(n)\Phi_m\bigl(g_{m,1}^ng_2^{\binom n2}\cdot e_X\bigr)|\geq\tau_2.
$$
We are now in a situation where the additional hypothesis are satisfied. We deduce that
the sequence $\bigl(g_{m,1}^ng_2^{\binom n2}\cdot e_X\bigr)_{n\in[N]}$ is not totally $\sigma_1$-equidistributed in $X$ for some $\sigma_1>0$.
Applying Theorem~\ref{th:Leibman} and then Lemma~\ref{lem:Leibman_Inverse} we deduce
that there exist an integer constant
$N_0'$ and a constant $\sigma_2>0$  such that, if $N\geq N'_0$,
the sequence $(g(n)\cdot e_X)_{n\in[N]}$ is not totally $\sigma_2$-equidistributed in $X$. Hence, in establishing  Proposition~\ref{lem:equid-square} we can take $m=0$ and   $g_0=\text{id}_G$.

In the rest of this proof, we constantly assume that these hypotheses are satisfied.
\subsection{Using Mal'cev coordinates}
Using the conventions of Section~\ref{subsec:nilmanifolds}
we use a Mal'cev basis and identify
$$
G=\R^{s}\times\R \ \text{ and }\ \Gamma=\Z^{s}\times\Z.
$$
 Points of $G$ are written as $(\bx,y)$, where $\bx=(x_1,\dots,x_s)\in\R^s$ and $y\in\R$.  The multiplication in $G$ is given by
$$
(\bx,y)\cdot(\bx',z)=\bigl(\bx+\bx', y+z+ \sum_{1\leq j<i\leq s}B_{i,j}x_ix_j'\bigr)
$$
where the $B_{i,j}$'s are integer constants,  not all equal to $0$ since $G_2$ is non-trivial.
 We have $G_2=\{\bzero\}\times\R$
where $\bzero= (0,\dots,0)\in\R^s$.
  We write
 \begin{equation}
\label{eq:defg1}
g_1=(\balpha,\beta)\ \text{ where }\ \balpha=(\alpha_1,\dots,\alpha_s)\in\R^s\ \text{ and }\ \beta\in\R.
\end{equation}

\subsection{Step 3: Using K\'atai's Lemma}
Combining \eqref{E:lower2} and  Lemma~\ref{lem:katai} we get that there exists a positive integer constant $K$,  primes $p,p'$ with $p<p'<K$, and a positive constant $\tau_3$ such that
$$
|\E_{n\in[N]} \one_{P}(pn) \one_{P}(p'n)
\Phi(g(pn)\cdot e_X)\cdot \overline{\Phi}(g(p'n)\cdot e_X)|>\tau_3.
$$
Since $p$ and $p'$ belong to a finite set  of primes that depends only on $\tau$, we can consider these numbers as fixed.
Let $P_1\subset[N]$ be the arithmetic progression such that $\one_{P}(pn)\one_P(p'n)=\one_{P_1}(n)$. Then the last inequality can be rewritten as
\begin{equation}
\label{eq:PhiKatai}
|\E_{n\in[N]} \one_{P_1}(n)
\Phi(g(pn)\cdot e_X)\cdot \overline\Phi(g(p'n)\cdot e_X)|\geq\tau_3.
\end{equation}

Our goal is to  show that this lower bound  implies  lack of total equidistribution of    the sequence
$
(g(pn)e_X,g(p'n)e_X)_{n\in [N]}
$
 on certain sub-nilmanifolds of $X\times X$  (denoted by $Y$ and $\wt Y$ below) on which the function $\Phi\otimes\overline \Phi$
integrates to zero. This will then imply lack of total equidistribution for  the sequence $(g(n)e_X)_{n\in [N]}$  on $X$ and will  complete the proof of Proposition~\ref{lem:equid-square}.

\subsection{The nilmanifold $Y$}
We define
$$
H:=\bigl\{ (a,a')\in G\times G\colon a^{p'}=  a'^p\bmod G_2\bigr\}
=\bigl\{ (b^p,b^{p'}u)\colon b\in G,\ u\in G_2\bigr\}.
$$
Then $H$ is a rational  subgroup of $G\times G$,  connected and simply connected,   and
$$
\Lambda:=(\Gamma\times\Gamma)\cap H
$$
is co-compact in $H$. We have that
$$
Y:=H/\Lambda
$$
 is a closed sub-nilmanifold of $X\times X$ and
$$
Y=\bigl\{(x,x')\in X\times X\colon \pi(x)^{p'}=\pi(x')^p\bigr\}
$$
where $\pi$ is the natural projection of $X$ onto the maximal torus $G/(G_2\Gamma)$.
As usual, we write $e_Y$ for the image of the unit element of $H$ in $Y$, and we have
$e_Y=(e_X,e_X)$.
We remark that
$$
H_2=\{(p^2z,p'^2z)\colon z\in G_2\}.
$$

We  now define some more convenient coordinates in $H$.
In the Mal'cev coordinates of $G\times G$, the subgroup $H$ of $G\times G$ is given by
$$
H= \bigl\{\bigl( (p\bx,y)\,,\,(p'\bx,y')\bigr)\colon \bx\in\R^s,\ y,y'\in\R\bigr\}.
$$
We choose two integers $q,q'$ with
\begin{equation}\label{E:qq'}
qp^2+q'p'^2=1
\end{equation}
 and identify
$$
H=\R^{s+2}=\R^s\times\R\times\R
$$
by mapping
$$
\bigl((p\bx,y)\,,\,(p'\bx,y')\bigr)\  \longmapsto\ (\bx\,,\, p'^2y-p^2y'\,,\,qy+q'y')
\ \text{ for }\bx\in\R^s\text{ and }y,y'\in\R.
$$
That this mapping is a bijection follows from \eqref{E:qq'}. A direct computation shows that in this new system of coordinates, the multiplication is given by
\begin{equation}
\label{eq:multH}
(\bx,w,z)\cdot(\bx',w',z')=\Bigl(\bx+\bx',w+w',z+z'+\sum_{1\leq j<i\leq s}B_{i,j}x_ix'_j \Bigr).
\end{equation}
Furthermore,
\begin{equation}
\label{eq:Lambda-H2}
\Lambda=\Z^s\times\Z\times\Z \quad  \text{and} \quad  \ H_2=\{\bzero\}\times\{0\}\times\R.
\end{equation}
Following our conventions, $H$ is endowed with some Mal'cev basis and $H$ and $Y$ are endowed with the associated  distances.

In our system of coordinates, for $w\in \R$,  the element $(\bzero,0,w)$ belongs to $H_2$ and, when viewed as an element of $G\times G$, it is equal to $\bigl((\bzero, p^2w),(\bzero,p'^2w)\bigr)$. Therefore, for $y=(x,x')\in Y$, we have
$$
(\Phi\otimes\overline\Phi)\bigl((\bzero,0,w)\cdot y\bigr)=\Phi(p^2w\cdot x)\overline\Phi(p'^2w\cdot x')=\e\bigl((p^2-p'^2)w) (\Phi\otimes\overline\Phi)(y)
$$
where we used that $\Phi$ is a nilcharacter of $X$ with frequency $1$.
Therefore,  the restriction to $Y$ of the function $\Phi\otimes\overline\Phi$ is a nilcharacter of frequency $p^2-p'^2$ and therefore has  zero integral with respect to the Haar measure on $Y$. Moreover, this restriction satisfies
$$
\norm{\Phi\otimes\overline\Phi\vert_Y}_{\lip(Y)}\leq C_3
$$
for some positive constant $C_3$.

Lastly, for future use, we choose a bounded fundamental domain  $F_2\subset H$  of the projection $H\to Y$ (we assume that $F_2$ is fixed given $Y$).
Then by Lemma~\ref{lemp:ap1}  there exists a constant $C_4$ such that
\begin{equation}
\label{eq:C6}
d_Y(h\cdot y,h\cdot y')\leq C_4d_Y(y,y')
\end{equation}
for every $h\in F_2$ and $y,y'\in X$.

\subsection{Step 4: Non-equidistribution on $Y$}
 We let
$$
h(n):=\bigl(g(pn)\,,\,g(p'n)\bigr)= h_1^n h_2^{\binom n2}
$$
where
$$
h_1:=\bigl( g_1^pg_2^{\binom p2}\,,\, g_1^{p'}g_2^{\binom {p'}2}\bigr)\in H\ \text{ and }
\ h_2:=\bigl(g_2^{p^2}\,,\, g_2^{p'^2}\bigr)\in H_2.
$$
 We have   $h(n)\cdot e_Y=\bigl(g(pn)\cdot e_X\;,\;g(p'n)\cdot e_X\bigr)\in Y$ for every $n\in \N$. The bound~\eqref{eq:PhiKatai} can be rewritten as
\begin{equation}
\label{eq:PhiKatai2}
|\E_{n\in[N]} \one_{P_1}(n)
(\Phi\otimes\overline\Phi)(h(n)\cdot e_Y)|\geq\tau_3.
\end{equation}
As remarked before,   the function $\Phi\otimes\overline\Phi\vert_Y$
has a bounded Lipschitz constant and integrates to $0$  with respect to the Haar measure of $Y$.
This and~\eqref{eq:PhiKatai2} imply that
\begin{fact}
\label{fa:non-equi-Y}
The sequence $(h(n)\cdot e_Y)_{n\in[N]}$  is not totally $\tau_4$-equidistributed in  $Y$
for some constant $\tau_4>0$.
\end{fact}
By Theorem~\ref{th:Leibman}, there exists a non-trivial horizontal character $\eta$ of $Y$, such that $\norm\eta\leq C_5$ and
\begin{equation}
\label{E:C_5}
 \norm{\eta\circ h}_{C^\infty[N]}\leq C_5
\end{equation}
for some positive constant $C_5$.

As explained in Section~\ref{subsec:nilmanifolds}, $\eta\colon H=\R^{s+2}\to\T$ is a group homomorphism of the form
\begin{equation}
\label{eq:defphixyy}
\eta(\bx,w,z)=\ba\cdot\bx+ bw\bmod 1
\end{equation}
for some $\ba\in\Z^s$ and $b\in\Z$, not both equal to zero as $\eta$ is non-trivial. Recall that $\norm\eta$ is the sum of the absolute value of the coefficients of $\eta$ in the Mal'cev basis of $H$ and that these coefficients are integers. Since $\norm\eta$ is bounded, it follows that $\eta$ belongs to some finite family that depends only on $X$ and $\tau$, and thus there exists a constant $C_6>0$ such that
\begin{equation}
\label{E:C_5'}
\norm\ba+|b|\leq C_6.
\end{equation}
 In the sequel, we consider  $\ba$ and $b$ as fixed.

An immediate computation gives that, in our system of coordinates,
\begin{equation}\label{E:h_1}
h_1=(\balpha,\beta,\kappa)\ \text{ where }\balpha\text{ was defined in~\eqref{eq:defg1} and }
 \beta,\kappa\in\R.
\end{equation}
The values of $\beta$ and $\kappa$ are not important.
Since $h_2\in H_2$,  we have $\eta(h(n))=n\eta(h_1)\bmod 1$ for every $n$, and thus  $\norm{\eta\circ h}_{C^\infty[N]}=N\norm{\eta(h_1)}_\T$.  By \eqref{E:C_5} and  the definition of the smoothness norm we deduce that
\begin{equation}\label{E:C_5''}
\norm{\ba\cdot\alpha+b\beta}_\T=\norm{\eta(h_1)}_\T\leq \frac{C_5}{N}.
\end{equation}

If $b=0$, then \eqref{E:C_5'} and \eqref{E:C_5''} give that
 for the horizontal character $\rho$  of $G$ defined by $\rho(\bx,y)=\e(\ba\cdot\bx)$ we have that
$\norm{\rho}$ and $\norm{\rho\circ g}_{C^\infty[N]}$ are bounded.
Using Lemma~\ref{lem:Leibman_Inverse} we deduce
 that the sequence $(g(n)\cdot e_X)_{n\in[N]}$ is not totally $\sigma$-equidistributed on $X$ for some positive constant $\sigma:=\sigma(X,\tau)$.

 It remains to deal with the case where $b\neq 0$. In this case,
  condition \eqref{E:C_5''} on its own
  does not  imply that the sequence $(g(n)\cdot e_X)_{n\in [N]}$ is not sufficiently equidistributed on the nilmanifold $X$.  To prove this we have to obtain an additional non-equidistribution property on
a sub-nilmanifold $\wt Y$ of $Y$.

\subsection{Step 5: Reduction to  a  primitive horizontal character}
For technical reasons  it is convenient to work  with a horizontal character that has relatively prime coefficients.
To this end, let $\ba=(a_1,\ldots,a_s)$ and define
$$
k:=\mathrm{gcd}(a_1,\dots,a_s,b).
$$
Note that $k\leq C_6$ by \eqref{E:C_5'}.
Let also $a'_i=k\inv a_i$ for $1\leq i\leq s$,
$\ba'=(a'_1,\dots,a'_s)=k\inv\ba$, $b'=k\inv b$, and let the maps $\phi'\colon H\to\R$ and $\eta'\colon H\to\T$ be defined by
\begin{align*}
\phi'(\underline{x},w,z) &:=\underline{a}'\cdot\underline{x}+ b'w;\\
\eta'(\underline{x},w,z)&:=\phi'(\underline{x},w,z)\bmod 1.
\end{align*}
Then $a'_1,\dots,a'_s,b'$ are relatively prime and $\eta'$ is a horizontal character of $H$ with $k\eta'=\eta$.
Furthermore,
 \eqref{E:C_5''} gives that
\begin{equation}\label{E:nonequiY'}
 \norm{\phi'(h_1^k)\bmod 1}_\T=\norm{\eta(h_1)}_\T\leq \frac{C_5}{N}.
\end{equation}

\subsection{Step 6: Non equidistribution on $\wt Y$}
Let
\begin{equation}
\label{eq:deftildeH}
\wt H:=\ker(\phi')=\bigl\{(\bx,w,z)\in H=\R^s\times\R\times\R\colon\ba'\cdot\bx+ b'w= 0\bigr\}.
\end{equation}
Then  $\wt H$  is a  connected, simply connected, rational subgroup of $H$. Furthermore,
 using our working assumption $b\neq 0$ and a  direct computation we get
$$
\wt H_2=\{ \bzero\}\times\{0\}\times\R=H_2.
$$
The discrete subgroup
$$
\wt\Lambda:=\Lambda\cap\wt H
$$
is co-compact in $\wt H$. We define the nilmanifold
$$
\wt Y:=\wt H/\wt\Lambda.
$$
As usual, we write $e_{\wt Y}$ for the image of $1_{\wt H}$ in $\wt Y$. Following our conventions, $\wt H$ is endowed with a Mal'cev basis, that we keep unspecified.

By \eqref{E:nonequiY'}  we have that  $\phi'(h_1^k)$ is at a distance not greater than $C_5/N$ of an integer.
Therefore, there exists  $\omega\in H$ such that $\phi'(\omega\inv h_1^k)\in\Z$ and
\begin{equation}
\label{E:C_7}
d_H(\omega,\text{id}_H)\leq\frac{C_7}{N}
\end{equation}
for some positive constant $C_7$.
Since  the coefficients of $\phi'$   are relatively prime integers, we get that
$\phi'$ maps $\Lambda$ onto $\Z$. Hence,   there exists $\gamma\in\Lambda$ such that $\phi'(\omega\inv h_1^k\gamma\inv)=0$, that is,
\begin{equation}
\label{eq:defovh1}
\ov h_1:=\omega\inv h_1^k\gamma\inv\in \wt H.
\end{equation}

We define
\begin{equation}
\label{eq:defL}
L:=\Big\lfloor N\min\big(\frac{\tau_3}{4C_3C_7}, \frac 1{2C_6}\big)\Big\rfloor
\end{equation}
 and assume that $N$ is sufficiently large so that $L\geq 1$.
Since $N\geq 2kL$, we  can make a  partition of the interval $[N]$ into arithmetic progressions of step $k$ and length between $L$ and $2L$. Since $N$ is bounded by a constant multiple of $L$, the number of these progressions is bounded and from~\eqref{eq:PhiKatai2}   we deduce that there exists a progression  $P_2$ of this type such that
$$
|\E_{n\in[N]} \one_{P_1}(n)\one_{P_2}(n)
(\Phi\otimes\overline\Phi)(h(n)\cdot e_Y)|\geq\tau_4
$$
for some constant $\tau_4>0$. Let $k_0$ be the smallest element of $P_2$ and let $P_3$ be the arithmetic progression defined by $\one_{P_3}(n)=\one_{P_1}(kn+k_0)\one_{P_2}(kn+k_0)$.
We have that $P_3\subset[2L]$  and
\begin{equation}
\label{eq:PhiKatai3}
\bigl|\E_{n\in [N]}\one_{P_3}(n)(\Phi\otimes\overline\Phi)(h(kn+k_0)\cdot e_Y)\bigr|\geq\tau_4.
\end{equation}

From the definition of the sequence $(h(n))$ and a direct computation that uses~\eqref{eq:defovh1}, we get that there exist $h_0\in H$,  $v\in H_2$, and $\wt h_2\in \wt H_2=H_2$ such that
\begin{equation}
\label{eq:hknk0}
h(kn+k_0)=\omega^nh_0(\ov h_1v)^n\wt h_2^{\binom n2}\gamma^n
\end{equation}
for every $n$
 (the precise values of $h_0,v$ and $\wt h_2$ are not important).  We choose $h'_0\in F_2$ and $\lambda$ in $\Lambda$ such that
$h_0=h'_0\lambda$ and define
\begin{gather}
\label{eq:deftildeh1}
\wt h_1=\lambda \ov h_1v\lambda\inv;\\
\label{eq:deftileh}
\wt h(n)=\wt h_1^n\wt h_2^{\binom n2} \text{ for every }n;\\
\label{eq:defPsi}
\Psi(y)=(\Phi\otimes\overline\Phi)(h'_0\cdot y)\text{ for }y\in Y.
\end{gather}

We remark that $\wt h_1$ belongs to $\wt H$ because $\wt h_1= [\lambda, \ov h_1]v\ov h_1$  and $[\lambda, \ov h_1]v\in H_2=\wt H_2$. Therefore, the sequence $(\wt h(n))$ is a polynomial sequence  in $\wt H$, of degree $2$ with respect to the usual filtration.   For every $n$, by~\eqref{eq:hknk0}, \eqref{eq:deftildeh1}, and~\eqref{eq:deftileh} we have
$$
h(kn+k_0)\cdot e_Y= \omega^n h'_0\wt h(n)\cdot e_Y.
$$

For $n\in P_3$ we have $n\leq 2L$, and using the triangle inequality and the right invariance of the metric $d_H$ we get $d_H(\omega^n,1_H)\leq 2LC_7N\inv\leq\tau_4/2C_3$ by~\eqref{E:C_7} and~\eqref{eq:defL}. Hence,
$$
d_Y\bigl(h(kn+k_0)\cdot e_Y,h'_0\wt h(n)\cdot e_Y \bigr)\leq \frac{\tau_4}{2C_3}.
$$
Since $\norm{\Phi\otimes\overline\Phi}_{\lip(Y)}\leq C_3$, it follows  that
$\bigl|(\Phi\otimes\overline\Phi)(h(kn+k_0)\cdot e_Y)-(\Phi\otimes\overline\Phi)
(h'_0\wt h(n)\cdot e_Y)\bigr| \leq\tau_4/2$.
From~\eqref{eq:PhiKatai3} and the definition~\eqref{eq:defPsi} of the function $\Psi$ we deduce
$$
\bigl|\E_{n\in [ N]}\one_{P_3}(n)
\Psi(\wt h(n)\cdot e_Y)\bigr|\geq\frac{\tau_4}{2}.
$$

Since $h'_0\in F_2$, we get from \eqref{eq:C6} that the restriction of $\Psi$ to $\wt Y$ has a bounded Lipschitz constant.
Moreover, since  $\Phi\otimes\overline\Phi$ is a nilcharacter of $Y$ of  non zero frequency  and $\wt H_2=H_2$,  it follows that $\Psi$ is a nilcharacter of $\wt Y$ of the same frequency.
Therefore, $\Psi$  has a zero integral with respect to the Haar measure on $\wt Y$.
From this and the  last inequality we deduce that
\begin{fact}
The sequence $(\wt h(n)\cdot e_Y)_{n\in[N]}$ is not totally
$\tau_5$-equidistributed in $\wt Y$ for some constant $\tau_5>0$.
\end{fact}

\subsection{Last Step.}
Fact~$2$ and  Theorem~\ref{th:Leibman} combined imply that
 there exists a non-trivial horizontal character $\theta$ of $\wt Y$ such that  $\norm\theta$ and $\norm{\theta\circ\wt h}_{C^\infty[N]}$ are bounded.

 First we obtain a more explicit formula for $\theta$.   As explained in Section~\ref{subsec:nilmanifolds}, $\theta$ has the form $\theta=f\circ\pi$, where
$\pi\colon \wt H\to\wt H/\wt H_2$ is the natural projection and $f\colon \wt H/\wt H_2\to\T$ is  a group homomorphism with a trivial restriction to  $\wt\Lambda\wt H_2/\wt H_2$.
 We assume that $H$  is endowed with the group structure arising from its identification  with $\R^{s+2}$. We remark that the natural projection $H\to H/H_2$ is still a group homomorphism.
 It follows from~\eqref{eq:deftildeH} that  $\wt H$ is a subgroup of $H$ and we assume that
  $\wt H$ is endowed with the induced group structure.
Since $\wt H_2=H_2$, the projection $\pi$ is the restriction to $\wt H$ of the projection  $H\to H/H_2$ and thus is a group homomorphism.  Therefore,
 $\theta\colon\wt H\to\T$ is a group homomorphism.

  On the other hand,
$\wt H$ was defined in~\eqref{eq:deftildeH}
as the kernel of a group homomorphism from $H$ to $\R$  with  integer coefficients. It follows that
there  exists a linear projection of $H$ onto $\wt H$ that has  integer coefficients and thus maps $\Lambda$ to $\wt\Lambda$. Composing the  homomorphism $\theta\colon\wt H\to\T$  with this projection, we see that $\theta$ can be extended to a group homomorphism
form $H$ to $\T$ vanishing on $\Lambda$, that we denote also by $\theta$.
 Since the restriction of $\theta$ to $H_2=\{ \bzero\}\times\{0\}\times\R$ is trivial, this map is given by
$$
\theta(\bx,w,z)=\bc\cdot\bx+ dw\bmod 1\quad \text{ for } \ (\bx,w,z)\in  H
$$
where $\bc\in\Z^s$ and $d\in\Z$.

 Recall that $\norm\theta$ is as usual computed in the Mal'cev basis of $\wt H$.  Since $\norm\theta$ is bounded,   it follows as in Step~4 that $\theta$ belongs to some finite family depending only on $X$ and $\tau$ and thus
\begin{equation}
\label{eq:C9}
\norm{\bc}+|d|\leq C_8
\end{equation}
 for some constant $C_8$.

Next we use the fact that $\norm{\theta\circ\wt h}_{C^\infty[ N]}$ is bounded.
Since  $\wt h_2
  \in\wt H_2$ we have  $\theta(\wt h(n))=n\theta(\wt h_1)\bmod 1$ for every $n\in \N$ and   $\norm{\theta\circ\wt h}_{C^\infty[ N]}= N\norm{\theta(\wt h_1)}_{\T}$. Hence,
\begin{equation}
\label{eq:thetah1}
\norm{\theta(\wt h_1)}_{\T}\leq \frac{C_9}{N}.
\end{equation}

Furthermore, by~\eqref{eq:defovh1} and~\eqref{eq:deftildeh1} we have
$\wt h_1= \lambda \omega\inv h_1^k\gamma\inv v\lambda\inv$. Note that  $\theta(\lambda)=\theta(\gamma)=0\bmod 1$ since $\theta$ and $\lambda$ belong to $\wt\Lambda$ and $\theta(v)=0\bmod 1$ since $v\in\wt H_2$.  Therefore,
$\theta(\wt h_1)= -\theta(\omega)+k\theta(h_1)\bmod 1$.
By~\eqref{E:C_7} and~\eqref{eq:C9}, we have  $\norm{\theta(\omega)}_\T\leq C_7C_8/N$ and thus
$\norm{k\theta(h_1)}_\T\leq (C_9+C_7C_8)/N $. Taking into account \eqref{E:h_1} we deduce
\begin{equation}
\label{E:C_11}
\norm{k\bc\cdot\balpha+kd\beta}_\T\leq \frac{C_9+C_7C_8}N.
\end{equation}
Recall by  \eqref{E:C_5''} that
\begin{equation}\label{E:C_11'}
\norm{\ba\cdot\balpha+b\beta}_\T \leq \frac{C_5}N.
\end{equation}
Since the restriction of  $\theta$ to $\wt H$ is non-trivial,
the vectors
$(\ba,b)$ and $(\bc,d)$ of $\R^{s+1}$ are not collinear.
Keeping this in mind, and combining the estimates \eqref{E:C_11} and \eqref{E:C_11'}, we get that there exists
 a  non zero vector $\bt\in\Z^s$ with
$$
\norm\bt\leq C_{10}\ \text{ and }\ \norm{\bt\cdot \balpha}_{\T}\leq \frac{C_{10}}{N}
$$
for some constant $C_{10}$.
We proceed as at the end of Step~4 to deduce
 that the sequence $(g(n)\cdot e_X)_{n\in[N]}$ is not totally $\sigma$-equidistributed on $X$ for some positive constant $\sigma:=\sigma(X,\tau)$. This completes the proof of Proposition~\ref{lem:equid-square}.

\section{Solution sets related to some homogeneous quadratic forms} \label{SS:AppNumberTheory}
We give a proof  of    Proposition~\ref{prop:linearfactors} from the
 introductory section.
 We recall the statement for reader's convenience.
\begin{proposition*}
Let the quadratic form $p$ satisfy the hypothesis of
Theorem~\ref{th:partition-regular1}. Then there exist
$\ell_0,\ell_1$ positive and $\ell_2,\ell_3$ non-negative integers
with $\ell_2\neq \ell_3$,  such that for every $k,m,n\in \N$, the
integers $x=k\ell_0 m(m+\ell_1n)$ and $y=
k\ell_0(m+\ell_2n)(m+\ell_3n)$ satisfy the equation $p(x,y,z)=0$
for some $z\in \N$.
\end{proposition*}

\begin{proof}
Let
\begin{equation}\label{E:Q2'}
ax^2+by^2+cz^2+dxy+exz+fyz=0
\end{equation}
be the equation we are interested in solving. Recall  that by
assumption $a,b,c$ are non-zero integers, and that  all three
integers
$$
\Delta_1:=e^2-4ac,  \quad \Delta_2:=f^2-4bc,  \quad
\Delta_3:=(e+f)^2-4c(a+b+d)
$$
are non-zero squares.

A direct computation shows that if $(x_0,y_0,z_0)$  is a solution of
\eqref{E:Q2'}, then also the following is  a solution
\begin{align*}
x=& k(-(ax_0+dy_0+ez_0)m^2-(2by_0+fz_0)mn+bx_0n^2) \\
y=& k(ay_0m^2-(2ax_0+ez_0)mn-(by_0+dx_0+fz_0)n^2)\\
z=& kz_0(am^2+dmn+bn^2)
\end{align*}
where $k,m,n,\in \Z$.\footnote{One finds these values by looking for
solutions of the form $x=tx_0+m$, $y=ty_0+n$, $z=tz_0$, which leads
to the choice
$t=-(am^2+dmn+bn^2)/((2ax_0+dy_0+ez_0)m+(2by_0+dx_0+fz_0)n)$. }

The discriminant of the quadratic form
$$
Q_1(m,n):=-(ax_0+dy_0+ez_0)m^2-(2by_0+fz_0)mn+bx_0n^2
$$ turns out to be $z_0^2\Delta_2$
which is a square since by assumption $\Delta_2$ is a square. Hence,
$Q_1(m,n)$ factors into a product of linear forms with rational
coefficients. Similarly, the discriminant of the quadratic form
$$
Q_2(m,n):=ay_0m^2-(2ax_0+ez_0)mn-(by_0+dx_0+fz_0)n^2
$$
turns out to be $z_0^2\Delta_1$ which is a square since by
assumption $\Delta_1$ is a square. Hence, $Q_2(m,n)$ factors into a
product of linear forms with rational coefficients.

The assumption that $\Delta_3$ is a  square is used to guarantee
that a choice of  $x_0,y_0,z_0$ with $z_0\neq 0$ can be made so that
the coefficients of $m^2$ in the quadratic forms $Q_1(m,n)$ and
$Q_2(m,n)$ are equal, i.e. $-(ax_0+dy_0+ez_0)=ay_0$ is satisfied.
Indeed, if we multiply equation \eqref{E:Q2'} by $e^2$ and insert
$-(ax_0+(d+a)y_0)$ in place of $ez_0$, we are lead  to the  equation
$$
a^2cx_0^2+a(2cd+2ac-e^2-ef)x_0y_0+(be^2+cd^2+a^2c+2acd-def-aef)y_0^2=0.
$$
A direct computation shows that its discriminant is
$a^2e^2\Delta_3$, which is a square since by assumption $\Delta_3$
is a square. This leads to  the following solution of \eqref{E:Q2'}
\begin{align*}
x_0&=2ac+2cd-e^2-ef\pm e\sqrt{\Delta_3}\\
y_0&=-2ac\\
z_0&=a\big(e+f\pm \sqrt{\Delta_3}\big).
\end{align*}
Note that since $\Delta_3\neq 0$ we can choose the sign so that
$z_0\neq 0$. We work with
 such a  choice of $x_0, y_0,z_0$ next.

Combining the above, we deduce that under the stated assumptions on
$a,b,c$,  there exist    $l_1,\ldots, l_8\in
\Z$, with $l_1 l_3=l_5l_7\neq 0$, such that for every $k,m,n\in\Z$ the integers
\begin{align*}
x&=k(l_1m+l_2n)(l_3m+l_4n),\\
 y&=k(l_5m+l_6n)(l_7m+l_8n)
 \end{align*}
satisfy  equation \eqref{E:Q2'} for some $z:=z_{m,n}\in \N$.
Inserting $l_1l_3l_5l_7n$ in place of $n$, we get that there exist
$l_1',\ldots, l_4'\in \Z$, such that for every $k,m,n\in\Z$ the
integers
\begin{align*}
x&=k\ell_0(m+l_1'n)(m+l_2'n),\\
y&=k\ell_0(m+l_3'n)(m+l_4'n),
\end{align*}
where $\ell_0=|l_1l_3|=|l_5l_7|\neq 0$, satisfy  equation
\eqref{E:Q2'} for some $z':=z'_{m,n}\in \N$.

Since $z_0,\Delta_1, \Delta_2$ are non-zero, the quadratic forms
$Q_1$ and $Q_2$ have non-zero discriminant, we have $l_1'\neq l_2'$
and $l_3'\neq l_4'$. Without loss of generality we can assume that
$l_1'\leq l_i'$ for $i=2,3,4$. Inserting $m-l_1'n$ in place of $m$,
we get that for every $k,m,n\in\Z$ the integers
\begin{align*}
x&=k\ell_0m(m+(l_2'-l_1')n),\\
y&=k\ell_0(m+(l_3'-l_1')n)(m+(l_4'-l_1')n)
\end{align*}
 satisfy  equation \eqref{E:Q2'} for some $z'':=z''_{m,n}\in \N$.
 Letting
$\ell_1=l_2'-l_1'$, $\ell_2=l_3'-l_1'$, $\ell_3=l_4'-l_1'$, we get
the asserted conclusion with  $z''_{m,n}$ in place of $z$.
\end{proof}
Alternatively, a   proof that is free of computations can be
given using the fact that the discriminant of the forms $p(x,0,z)$,
$p(0,y,z)$, $p(x,x,z)$ is a non-zero    square. We chose a more
hands on argument  since it determines the integers $\ell_1$,
$\ell_2$, $\ell_3$ explicitly.

\end{document}